\documentclass[a4paper]{amsart}
\usepackage{amsmath,amsthm,amssymb,amscd}
\begin{document}
	
\theoremstyle{plain}	
\newtheorem{thm}{Theorem}[section]
\newtheorem{lemma}[thm]{Lemma}
\newtheorem{proposition}[thm]{Proposition}
\newtheorem{corollary}[thm]{Corollary} 

\theoremstyle{definition}
\newtheorem{defn}{Definition}[section]

\theoremstyle{remark}
\newtheorem{remark}{Remark}[section]
\newtheorem{remarks}{Remarks}[section]
\newtheorem*{notation}{Main notation}

\newcommand{\arctg}{\operatorname{arctg}}
\newcommand{\sech}{\operatorname{sech}}
\newcommand{\ve}{\varepsilon}
\newcommand{\Ltau}{L^2_{\tau}(\mathbb{R})}
\newcommand{\Lltau}{L^1_{\tau}(\mathbb{R})}
\newcommand{\lu}{\tilde{u}}
\newcommand{\lv}{\tilde{v}}
\newcommand{\dt}{\lVert t \rVert}
\newcommand{\dtau}{\lVert \tau \rVert}
\newcommand{\supp}{\operatorname{supp}}

\numberwithin{equation}{section}

\title[Variational construction of positive entropy measures]{Variational construction of positive entropy invariant measures of Lagrangian systems and Arnold diffusion}

\author{Sini\v{s}a Slijep\v{c}evi\'{c}, Zagreb}

\date{21 June, 2016}

\begin{abstract} We develop a variational method for constructing positive entropy invariant measures of Lagrangian systems without assuming transversal intersections of stable and unstable manifolds, and without restrictions to the size of non-integrable perturbations. We apply it to a family of two and a half degrees of freedom a-priori unstable Lagrangians, and show that if we assume that there is no topological obstruction to diffusion (precisely formulated in terms of topological non-degeneracy of minima of the Peierl's barrier function), then there exists a vast family of "horsheshoes", such as "shadowing" ergodic positive entropy measures having precisely any closed set of invariant tori in its support. Furthermore, we give bounds on the topological entropy and the "drift acceleration" in any part of a region of instability in terms of a certain extremal value of the Fr\'{e}chet derivative of the action functional, generalizing the angle of splitting of separatrices. 
The method of construction is new, and relies on study of formally gradient dynamics of the action (coupled parabolic semilinear partial differential equations on unbounded domains). We apply recently developed techniques of precise control of the local evolution of energy (in this case the Lagrangian action), energy dissipation and flux. In Part II of the paper we will apply the theory to obtain sharp bounds for topological entropy and drift acceleration for the same class of equations in the case of small perturbations.
\end{abstract}
 
\maketitle

%{ \fontsize{8}{10}
	
%\endkeywords
\vspace{2ex}

%}
%To be completed
%\endsubjclass

\section{Introduction}

\vspace{5ex}

Consider a $C^2$, Tonelli Lagrangian $L : \mathbb{T}^N \times \mathbb{R}^N \times {\mathbb{T}} \rightarrow \mathbb{R}$ (see \cite{Sorrentino10} for definitions). The deep motivation for the paper is the Birkhoff ergodicity hypothesis on equivalence of space and time averages for Hamiltonian systems (see \cite{Avila:15} for an overview and recent results). Related to that, it is important to know if the "size" (i.e. the natural measure) of the "chaotic", or "unstable" part of the phase space for a typical Hamiltonian is non-zero, which is essentially open even in the simplest case of area-preserving twist diffeomorphisms \cite{Knill:99}.

A more focused approach to investigate instability of Hamiltonian systems is to consider existence of orbits which "drift" in phase space, or in other words the existence of Arnold diffusion, following Arnold's construction \cite{Arnold:64} in the case of a weakly coupled rotator and pendulum with a weak periodic forcing. Typical considered questions are on existence of such orbits in specific examples \cite{Arnold:64,Bernard:96,Bernard:08,Bessi:96,Cherchia:94}, on genericity of existence of such orbits \cite{Bernard:11,Cheng:04,Cheng:09,Cheng10,Kaloshin:12,Kaloshin:15,Mather:93,Mather:12,Treschev:04,Treschev:12}, and on the fastest possible drift \cite{Berti:02,Berti:03,Zhang:11} (the references include only a small sample of the relevant results). Two typical approaches to construction of Arnold diffusion orbits are "geometric" and "variational" (see \cite{Bernard:10,Llave06} for an introduction and further references). The geometric approach essentially relies on finding a normally hyperbolic "scaffolding" of a perturbed integrable ("a-priori stable"), or integrable weakly coupled with an "a-priori unstable" (e.g. a pendulum, or a kicked pendulum) Hamiltonian. Futhermore, the geometric method typically relies on transversal intersections of stable and unstable manifolds of the "scaffolding" (e.g. the remaining KAM tori, or a normally hyperbolic cylinder), and construct an orbit which shadows it typically by an application of the implicit function theorem. On the other hand, the variational approach typically relies on minimizing the action under carefully constructed constraints. The variational approach is frequently complemented by leveraging the weak-KAM theory \cite{Fathi}, or the description of the action-minimizing Mather's sets and invariant measures and their extensions \cite{Mather:93,Sorrentino10}, which can result with insightful descriptions of the regions of instability \cite{Bernard:08,Cheng:09}.

Our aim is to propose an alternative technique for construction of Arnold diffusion, enabling in addition to construction of individual orbits, also a construction of "shadowing" invariant measures in an a-priori specified region of the phase space. Specifically, we construct a rich family of positive entropy ergodic measures, and are able to estimate their metric entropy. As a result, we can relate the speed of drift in the phase space to the average of locally the largest Lyapunov exponents along a drift trajectory. We thus describe dynamics in a significant region of the phase space (though still most likely of the measure zero with respect to the natural measure of the manifold). Importantly, for our construction to hold, it suffices that there is no topological obstacle to diffusion, and we require no transversal intersections of stable and unstable manifolds. Even though our construction is essentially variational, it is precise enough to incorporate "geometric" information if available.

To introduce the method, we recall an alternative construction of shadowing orbits constructed by Mather \cite{Mather:91} in the case of area-preserving twist diffeomorphisms on the cylinder, or equivalently of 1 1/2 degrees of freedom Tonelli Lagrangians on the torus. In \cite{Slijepcevic:99}, we considered formally gradient dynamics of the action (see Section \ref{s:existence} and the equations (\ref{r:Grad}) for details) for that system. This is an extension of the variational approach, where we consider evolution of approximate orbits along the gradient of the action, until they "relax" to an equilibrium, which is by LaSalle principle the actual solution of the Euler-Lagrange equations. One of the novelties in \cite{Slijepcevic:99} is that the evolution dynamics is considered on an unbounded domain, i.e. for all times $t \in \mathbb{R}$, when the dynamics is not gradient-like any more (see Remark \ref{r:extended} for details). The construction in  \cite{Slijepcevic:99} is simple: heteroclinic orbits are constructed whenever there is no obstruction to diffusion (in that case homotopically non-trivial invariant circles, or KAM-circles). We simply let any function $q(t)$ asymptotic to two Mather's sets at $t \rightarrow -\infty$, respectively $t \rightarrow \infty$, in the same region of instability evolve (or relax) along the formal gradient of the action, denoting the relaxation time by $s$. As there are no invariant KAM circles between these two Mather's sets, we show that the $s$-evolution of $q$ must asymptotically stop after a finite distance (we show that otherwise there would be a KAM-circle in the limit set - a contradiction), so the "tails" must remain asymptotic to either Mather's sets. The configuration $\lim_{s \rightarrow \infty} q(s)$ is the required heteroclinic orbit crossing an arbitrarily large part of a region of instability. 

The main tool in \cite{Slijepcevic:99} - the order-preserving property of the dynamics - does not extend to higher dimensions. The techniques of study of formally gradient systems (called also extended gradient systems), introduced in \cite{Gallay:01}, have recently matured enough \cite{Gallay:12,Gallay:14,Gallay:15}, so that we can extend the approach to more degrees of freedom. We are now able to replace the monotonicity techniques by "energy methods", or specifically, by considering local interplay of energy, energy dissipation and energy flux, energy being in this case the Lagrangian action.

To explain it, we compare the approach with the variational method introduced by Bessi \cite{Bessi:96} in the Arnold's example. Bessi somewhat implicitly considered gradient dynamics of the action on a large, but still finite domain. He was then able to "control" evolution of an approximate shadowing orbit, by showing that the {\it total available action} along the entire constructed (finite, but very long) orbit is less than what is needed for every single section between two "jumps" (corresponding to one heteroclinic orbit in a diffusion "chain") to significantly move and "escape" from the desired region of the phase space. With this method it is difficult to construct orbits with "infinitely many jumps" (as the total "available action" is infinite), and the "control" decays proportionally to the length of the considered orbit. 

Thierry Gallay and the author recently developed techniques establishing stability results for dissipative partial differential equations independent of the size of the domain. For example, in \cite{Gallay:14,Gallay:15}, we established a-priori bounds for relaxation of unforced Navier-Stokes equations on a strip, independent of the domain size, thus holding for the equations on the unbounded domain. Applying and extending these ideas to formally gradient dynamics of the action, we are thus able to construct orbits of infinite length and invariant measures. We are also able to obtain sharp estimates on the drift acceleration, matching (the case of non-degenerate Melnikov function and small perturbation), or improving (the case of degenerate Melnikov function) the results obtained by the geometric or an alternative approach (further details will be reported in the Part II \cite{Slijepcevic:16}).

After introducing the general method of constructing shadowing invariant measures, we apply it here to a family of 2 1/2-degrees of freedom a-priori unstable Lagrangians. We, however, believe that the approach can eventually be extended to more general Tonelli Lagrangians, as long as there is a rich family of partially hyperbolic Mather's sets (not necessarily invariant tori), and as long as there is no topological barrier to diffusion, expressed as a certain topological non-degeneracy of the Peierl's barrier function, or of weak KAM solutions.

\subsection{Statements of the main results} We first develop new tools for construction of orbits and invariant measures for Lagrangians of the type $L(q,q_t,t) = \frac{1}{2}q_t^2 + V(q,t)$ (in Remark \ref{rem:Tonelli} we explain how the tools can be applied to the entire class of Tonelli Lagrangians).
We then apply the general theory to a family of a-priori unstable Lagrangians with 2 1/2 degrees of freedom, already considered in e.g. \cite{Arnold:64,Bernard:96,Berti:02,Berti:03,Treschev:04}, given with
\begin{equation}
\begin{split}
L(u,v,u_t,v_t,t) &=\frac{1}{2}u_t^{2}+\frac{1}{2}v_t^{2}+V(u,v,t),  \\
V(u,v,t)&=\ve\left( 1-\cos u \right) \left( 1-\mu f(u,v,t)\right),
\end{split} \label{d:unstable}
\end{equation}
where $(u,v,u_t,v_t,t) \in \mathbb{T}^2 \times \mathbb{R}^2 \times \mathbb{T}$, $\mathbb{T}$ is parametrized with $[0,2\pi)$, $\ve, \mu \geq 0$ are parameters, and $f$ is $2 \pi$-periodic in all the variables. Our standing assumptions are as follows:
\begin{description}
	\item[(A1)] $f$ is $C^{4+\gamma}(\mathbb{R}^3)$, $\gamma>0$ and $|f| \leq 1$, $|f_v| \leq 1$, 
	
	\vspace{1ex}
	
	\item[(A2)] $0 \leq 16 \mu \leq \ve \leq 1$. 
\end{description}
Note that the bounds on $f$ in (A1) are not an essential restriction, as we can always scale $f$ and adjust $\mu$ for this to hold. The restricted range of parameters is also used mainly for convenience in the calculations.  In any case, (A2) allows the case of "a-priori unstable" small perturbations in $\mu$, as well as other physically relevant cases such as those considered in \cite{Simo:00}.

The main tool, but also an object of study, is the formally gradient dynamics associated to (\ref{d:unstable}), given by the equations
\begin{subequations} \label{r:grad}
	\begin{align}
	u_s &=  u_{tt}-\partial_{u}V(u,v,t),  \label{r:gradu} \\ 
	v_s &= v_{tt}-\partial_{v}V(u,v,t),  \label{r:gradv} \\ 
	u(0,t)& = u^{0}(t). \nonumber
	\end{align} 
\end{subequations}
The techniques we develop here can also be interpreted as new results on uniformly local stability of the equation (\ref{r:grad}) and similar equations on unbounded domains. We hope to make it more explicit in future research.

Consider solutions of the Euler-Lagrange flow induced by the Lagrangian (\ref{d:unstable})
\begin{subequations}
	\begin{align}
	u_{tt} & = V_u = \ve \sin u (1-\mu f(u,v,t))) + \ve \mu (1-\cos u)f_u(u,v,t), \label{r:ELu} \\
	v_{tt} & = V_v = \ve \mu (1-\cos u)f_v(u,v,t). \label{r:ELv}
	\end{align}
	\label{r:EL}
\end{subequations}
Let $\phi$ be the non-autonomous flow induced by (\ref{r:EL}) on $\mathbb{T}^2 \times \mathbb{R}^2 \times \mathbb{T}$. We use the notation $(u,v,u_t,v_t,t) \in \mathbb{T}^2 \times \mathbb{R}^2 \times \mathbb{T}$, and always parametrize tori with $[0,2\pi)$. We denote by $\sigma = \phi^{2\pi}$ the time-$2 \pi$ map, and then $\sigma$ is a diffeomorphism of $\mathbb{T}^2 \times \mathbb{R}^2$. As in the classical Arnold's example \cite{Arnold:64}, for each $\ve,\mu \in \mathbb{R}$, and for each "speed of the rotator" $\omega \in \mathbb{R}$, the invariant tori $\mathbb{T}_{\omega}:=\lbrace (0,v,0,\omega), v\in \mathbb{R} \rbrace $ are $\sigma$-invariant (and the sets $(0,v,0,\omega,t),\: (v,t)\in \mathbb{R} \times \mathbb{T})$ are $\phi$-invariant), i.e. these invariant tori persist for all perturbations.

We define regions of instability, generalizing an analogous notion for area-preserving twist maps \cite{Mather:91} as follows. Given $\omega \in \mathbb{R}$, let $S_{\omega}$ be the Peierl's barrier function (closely related to weak-KAM solutions of Hamilton-Jacobi equations \cite{Fathi,Sorrentino10}), defined on $\mathbb{R}^2$ with
\[
S_{\omega}(t_0,v_0)=\inf \left\lbrace \int_{-\infty}^{\infty}L_{\omega}(q,q_t,t)dt, \: q=(u,v) \in H^1_{\text{loc}}(\mathbb{R})^2, \: q(t_0)=(\pi,v_0), \: \lim_{t \rightarrow -\infty}u(t)=0, \: \lim_{t \rightarrow \infty}u(t)=2 \pi \right\rbrace,
\]
where $L_{\omega}(q,q_t,t)=u_t^2/2+(v_t-\omega)^2/2+V(u(t),v(t),t)$ is the adjusted Lagrangian. It is well-known, and recalled in Section \ref{s:five}, that $S_{\omega}$ is continuous and $2\pi$-periodic in each variable, that for each $(t_0,v_0)$ there exist $q \in H^1(\mathbb{R})^2$ for which the infimum is attained, and that for $(t_0,v_0)$ which are critical points of $S_{\omega}$, these $q$ correspond to solutions of (\ref{r:EL}) homoclinic to $\mathbb{T}_{\omega}$. We say that $\omega \in \mathbb{R}$ is non-degenerate, if each connected component of its set of global minima of $S_{\omega}$ is bounded. We interpret it as "no topological obstacle to diffusion"$^1$. The set of non-degenerate $\omega$ is open (see Remark \ref{rem:instability}), and we call each its connected component a region of instability. We frequently restrict attention to a closed subset $[\omega^-,\omega^+]$ of a single region of instability. We will require an additional, technical assumption that the bounded components of the global minima of $S_{\omega}$ are not too large, i.e. we assume the following:
\begin{itemize}
	\item[(S1)] For each $\omega \in [\omega^-,\omega^+]$ and for each global minimum $(t_0,v_0)$ of $S_{\omega}$, there exists a closed neighbourhood $\mathcal{N}$ of $(t_0,v_0)$ such that for each $(t_1,v_1) \in \partial \mathcal{N}$,
	\begin{equation}
	  S_{\omega}(t_1,v_1) - S_{\omega}(t_0,v_0) \geq 3 \Delta_0,  \label{r:DDelta0}
	\end{equation} 
	where $\Delta_0> 0$ is an uniform constant over $[\omega^-,\omega^+]$, and the diametar of $\mathcal{N}$ is not greater than $R$, such that
	\begin{equation}
	R \sqrt{\ve} \leq 1/144. \label{r:Rcond}
	\end{equation}
\end{itemize}
Apart from the technical restriction$^2$ (\ref{r:Rcond}), the condition (S1) is equivalent to the definition of the region of instability (see Remark \ref{rem:instability}). Our definition of the region of instability is closely related to the results of Cheng$^3$ \cite{Cheng:04,Cheng:09,Cheng10} and Bernard$^4$ \cite{Bernard:08,Bernard:10}. We note that the quantity $3\Delta_0$ in (\ref{r:DDelta0}) is analogous to the quantity $\Delta W$ quantifying transport in the case of area-preserving twist maps$^5$ \cite{MacKay:88}. 

Importantly, we do not require any non-degeneracy of the minima of $S_{\omega}$, for example we do not require that the second derivative of $S_{\omega}$ at these minima is positive definite. As perhaps noted first by Angenent \cite{Angenent:93}, such non-degeneracy would be equivalent to requiring that the stable and unstable manifolds (the "whiskers") of $\mathbb{T}_{\omega}$ intersect transversally, and would lead to the "geometric" approach to diffusion phenomena (\cite{Llave06} and references therein). On the contrary, we allow the whiskers to intersect non-transversally, and construct possibly non-uniformly hyperbolic invariant sets. The main result is now a construction of a large number of "horsheshoes", i.e. of ergodic positive entropy measures roughly contained in an arbitrary part of a region of instability:
\begin{thm} \label{t:main1}
Assume $[\omega^-,\omega^+]$ satisfies (S1). Then for each closed subset $\mathcal{O}$ of $[\omega^-,\omega^+]$, there exist an ergodic, $\phi$-invariant, positive entropy Borel probability measure $\mu$ on $\mathbb{T}^2 \times \mathbb{R}^2 \times \mathbb{R}$, such that 
\begin{equation}
 \cup_{\omega \in \mathcal{O}}\mathbb{T}_{\omega} \subset \supp \mu, \quad \cup_{\omega \in  \mathbb{R}-\mathcal{O}}\mathbb{T}_{\omega} \subset (\supp \mu)^c. \label{r:shadow}
\end{equation}
\end{thm}	
We call the invariant measures described by (\ref{r:shadow}) the shadowing measures, as they indeed "shadow" an arbitrary closed set of invariant tori. A direct corollary of the Theorem \ref{t:main1} is that we can find a single, "dense" orbit whose closure contains the entire $\cup_{\omega \in [\omega_1,\omega_2]} \mathbb{T}_{\omega}$ (we choose $\mathcal{O}=[\omega^-,\omega^+]$, and by ergodicity we find an orbit which is dense in the support of the measure $\mu$). We actually by our method also give a direct proof of existence of such and more elaborated shadowing orbits in Theorem \ref{t:shadowing}. One can apply it to construct other complex structures in a vicinity of a chosen $\mathbb{T}_{\omega}$ for non-degenerate $\omega$, such as an analogoue of the Mather's construction in the case of twist maps$^6$ \cite{Mather:85}.

Importantly, we are also able to obtain estimates of the metric entropy of constructed measures, topological entropy of $\sigma$ and $\phi$ (by the variational principle for metric and topological entropy), and of the drift acceleration (or, using a less precise term, the speed of diffusion). Specifically, we can associate to each closed interval $[\omega^-,\omega^+]$ a value $\Delta_1$, expressed as a certain extremal value of the norm of the Fr\'{e}chet derivative of the action, and defined precisely in Proposition \ref{p:dissipation}. The value $\Delta_1$ can be understood as a generalization of the lower bound on the angle of splitting of separatrices$^7$. Importantly, $\Delta_1 >0$ whenever (S1) holds. Let $\varpi=\max \lbrace |\omega^-|,|\omega^+|,1 \rbrace$. To avoid repetition in the statements, we say that the topological entropy on $[\omega^-,\omega^+]$ is $O(h)$, if there exists an absolute constant $1 \geq c_0 > 0$ such that the the topological entropy of $\phi$ and $\sigma$ restricted to an invariant subset of 
\begin{equation} \label{r:smallset}
\begin{split}
|u_t| & \leq c_0 \varpi \sqrt{\ve}, \\ v_t & \in [\omega^- - c_0 \varpi \sqrt{\ve}, \omega^+ + c_0 \varpi \sqrt{\ve}]
\end{split}
\end{equation}
 is at least $c_0 h$. We say that the {\it drift acceleration} on $[\omega^-,\omega^+]$ is $O(d)$, if for each $\delta>0$ there exists an absolute constant $1 \geq c(\delta) > 0$ (depending only on $\delta$), a solution $q=(u,v)$ of (\ref{r:EL}) satisfying (\ref{r:smallset}) for all $t \in \mathbb{R}$, and the times $t^- < t^+$ such that $|v_t(t^-)-\omega^- | \leq \delta$, $|v_t(t^+)-\omega^+ | \leq \delta$, and such that
 \begin{equation*}
   d \geq c(\delta) \frac{|\omega^+-\omega^-|}{|t^+-t^-|}.
 \end{equation*}

\begin{thm} \label{t:main2}
	Assume $\omega^- \leq \omega^+$ such that $[\omega^-,\omega^+]$ satisfies (S1), and let $\Delta_1>0$ be as is defined in Proposition \ref{p:dissipation}. Then: 
	\begin{itemize}
		\item[(i)] The topological entropy of $\phi$ and $\sigma$ on $[\omega^-,\omega^+]$ is 
		$$
		O \left( \frac{\Delta_1}{\varpi^5 |\log \Delta_1|} \right),
		$$
		
		\item[(ii)] The drift acceleration on $[\omega^-,\omega^+]$ is $$O\left( \frac{\Delta_0 \Delta_1}{\varpi^6 (R \vee \mu) |\log \Delta_1|} \right).$$
	\end{itemize}
\end{thm}
Our estimates are consistent with upper bounds on the drift acceleration in the cases considered by Nekhoroshev$^8$ \cite{Nekhoroshev:77}, as well as on upper bounds in \cite{Berti:03} (see below).

The emerging picture of the Arnold diffusion is actually more subtle. We can set $\omega^-=\omega^+$, in Theorem \ref{t:main2}, and find $\Delta_1=\Delta_1(\omega)$ and an ergodic positive entropy measure $\mu_{\omega}$ in a neighborhood of $\mathbb{T}_{\omega}$ with the locally maximal metric entropy as in Theorem \ref{t:main2}, (ii). By the Margulis-Ruelle inequality, there is a positive Lyapunov exponent with respect to $\mu_{\omega}$, which is at least $\sim \Delta_1(\omega)/|\log \Delta_1(\omega)|$. The drift acceleration seems to be proportional to the integral of these Lyapunov exponents with respect to $\omega$ along a diffusion path. This picture somewhat explains, and provides tools to study the dynamics of various "time-scales", i.e. the observed diffusion in Hamiltonian systems with transport speed substantially varying in different regions of the phase space (see e.g. \cite{Berti:02b,Gelfreich:14} and references therein). 

Consider now the case of small perturbations (i.e. small $\mu >0$). Recall the definition of the Melnikov primitive $M_{\omega}:\mathbb{R}^2 \rightarrow \mathbb{R}$,
\begin{equation}
M_{\omega}(t_0,v_0)= - \ve \int_{-\infty}^{\infty} (1-\cos (u^{\ve}(t-t_0))f(u^{\ve}(t-t_0),v_0+\omega(t-t_0),t)dt,
\end{equation}
where $u^{\ve}=4 \arctg e^{\sqrt{\ve}\: t}$ is the separatrix of the unperturbed system, i.e. the case $\mu=0$ when (\ref{r:EL}) reduces to an uncoupled pendulum and a rotator. It is well-known that for sufficiently small $\mu$ and fixed $\omega$, the oscillations of $M_{\omega}$ approximate well the oscillations of $S_{\omega}$, the minima of $M_{\omega}$ approximate the minima of $S_{\omega}$, and non-degeneracy of the minima of $M_{\omega}$ implies non-degeneracy of minima of $S_{\omega}$ (i.e. transversal intersection of "whiskers"). To get the conclusions of Theorems \ref{t:main1} and \ref{t:main2}, it thus suffices to assume an analogue of (S1) for $M_{\omega}$:
\begin{itemize}
	\item[(S2)] For each $\omega \in [\omega^-,\omega^+]$, and for each global minimum $(t_0,v_0)$ of $M_{\omega}$, there exists a closed neighbourhood $\mathcal{N}$ of $(t_0,v_0)$ such that for each $(t_1,v_1) \in \partial \mathcal{N}$,
	\begin{equation}
	M_{\omega}(t_1,v_1) - M_{\omega}(t_0,v_0) \geq 4 \tilde{\Delta}_0,  \label{r:DDDelta0}
	\end{equation} 
	where $\tilde{\Delta}_0> 0$ is an uniform constant over $[\omega^-,\omega^+]$, and the diametar of $\mathcal{N}$ is not greater than $R$, where $R$ satisfies (\ref{r:Rcond}).
\end{itemize}
\begin{thm} \label{t:main3} Assume that (S2) holds. Then there exists $\mu_0 >0$, such that for for each $0 < \mu \leq \mu_0$, the conclusions of Theorem \ref{t:main1} and \ref{t:main2} hold, with $\Delta_0=\mu \tilde{\Delta}_0$.
\end{thm}

In the Part II of the paper \cite{Slijepcevic:16}, we will explicitly estimate $\Delta_0$, $\Delta_1$, the topological entropy and drift acceleration for small $\mu$ under different assumptions on the Melnikov primitive. We will show that the approach seems to give optimal estimates as compared to known results for (\ref{d:unstable}). For example, we will obtain the drift acceleration $O(\mu/|\log \mu|)$ and topological entropy $O(1/|\log \mu|)$ which is optimal$^9$, in the case of Melnikov primitive with non-degenerate minima and small $\mu$. We will strengthen known results, and also show that for such fast drift acceleration, it suffices that $||D^2M_{\omega}(t_0(\omega),v_0(\omega)||^{-1}$ (where $(t_0(\omega),v_0(\omega))$ minimize $M_{\omega}$) is integrable along a diffusion path (equivalently, that the inverse of the splitting angles of separatrices is integrable), as long as (S2) holds. Furthermore, we will show that for small $\mu \leq \mu_0$, (\ref{r:Rcond}) is not needed (however, $\mu_0$ is then inverse proportional to $R$), and will obtain new estimates for topological entropy and drift acceleration for $M_{\omega}$ with degenerate minima as a function of the leading term in the Taylor expansion of $M_{\omega}$ at the minimum.

Finally, we remark that all the results hold in the classical Arnold's example \cite{Arnold:64} with $f(u,v,t)=\cos v+\cos t$. In that case, the Melnikov primitive $M_{\omega}$ can be explicitly calculated \cite{Arnold:64,Bessi:96}, and it satisfies (S2) with the regions of instability $(-\infty,0)$ and $(0,\infty)$. We can thus obtain diffusion orbits and shadowing measures in the Arnold's example for $\ve \leq 1$ and sufficiently small$^{10}$ $\mu>0$, without restrictions to $\omega$, as long as the sign of $\omega$ does not change. One can tailor the argument to also obtain accelerating orbits, i.e. orbits with unbounded $\omega$ in that case (and any case when $f$ does not depend on $u$).

\vspace{1ex}

{ \fontsize{8}{10}
\begin{remarks}
\noindent (1) Consider the stable and unstable manifolds of $\mathbb{T}_{\omega}$ of the time-$2\pi$ map $\sigma$. Then an unbounded family of global minima of $S_{\omega}$ corresponds to an unbounded, connected family of homoclinic orbits of $\sigma$, which can not be "crossed" by other orbits on 2-dimensional stable and unstable manifolds of $\sigma$. This would prevent drift and the complex dynamics we describe in Theorem \ref{t:main1} in that region of the phase space.

\vspace{1ex}

\noindent (2) The restriction on $R$ is used only in the proof of Proposition \ref{p:invariant} at the end of Section \ref{s:invariant}, to assist in the cases of topologically complex $\mathcal{N}$ in (S1). We will show in \cite{Slijepcevic:16} that this is not needed in the case of small perturbations (sufficiently small $\mu$). Alternatively, one can instead assume that $\mathcal{N}$ is convex, and bound the interval $[\omega^-,\omega^+]$ away from zero. Note that Bernard in \cite{Bernard:96} considered the same equation for sufficiently small $\mu$ with an  assumption analogous to Theorem \ref{t:main3}, but assuming $\mathcal{N}$ to be rectangular. Bernard used the method of Bessi \cite{Bessi:96} (see the earlier discussion on the Bessi method for a comparison), and constructed diffusion orbits of finite length in a restricted range of $\omega$.

\vspace{1ex}
\noindent (3) The definition of the region of instability by Cheng and Yan \cite{Cheng:04,Cheng:09,Cheng10} seems to be essentially equivalent to ours (they give it in a more general and abstract setting). Our understanding is that the method in these papers does not result with diffusion orbits of infinite length, thus does not imply existence of invariant measures, and that it does not immediately give estimates of the drift acceleration, topological entropy and Lyapunov exponents.

\vspace{1ex}
\noindent (4) Our results in Theorem \ref{t:main2} can be interpreted in the sense of Bernard's {\it forcing relation} of cohomology classes \cite{Bernard:08}: if $\omega$ and $\tilde{\omega}$ are in the same region of instability, then $(0,\omega),(0,\tilde{\omega}) \in H^1(\mathbb{T}^2,\mathbb{R})$ are related.

\vspace{1ex} 
\noindent (5) Assume we take $\omega^-=\omega^+$, and find the largest $3\Delta_0$ such that (S1) holds. Then $3\Delta_0$ is the difference of actions of a minimizing and a "minimax" (in this case a "saddle") homoclinic orbit. In the case of area-preserving twist maps, one would analogously obtain exactly the quantity $\Delta W$ quantifying transport through gaps in Cantori \cite{MacKay:88}.

\vspace{1ex}
\noindent (6) Mather in \cite{Mather:85} constructed uncountably many minimal sets of twist maps with the same irrational angular rotation $\omega$. One can adapt our construction in the proof of Theorem \ref{t:main2} to construct uncountably many minimal sets for irrational $\omega$, supported on the set $\lim_{T \rightarrow \infty}(v(T)-v(-T))/2=\omega$, by essentially constructing orbits shadowing a couple of orbits homoclinic to $\mathbb{T}_{\omega}$ (jumping "forward" and "backward") with the time between "jumps" in the set $nL$, $n$ in a fixed subset of $\mathbb{N}$. We intend to provide details separately. 

\vspace{1ex}
\noindent (7) We explain it by analogy. Consider a sufficiently smooth function $S: \mathbb{R}^n \rightarrow \mathbb{R}$ with a local minimum $x=0$. If the minimum is nondegenerate and $(D^2S(0)x,x) \geq A||x||^2$, for small enough $\delta >0$ on the level sets $S(x)=A \delta^2$ in a neighborhood of $0$ we trivially have that $||DS(x)|| \geq 2A \delta  + O(\delta^2)$. If, however, we merely assume that the set of local minima of $S$ is bounded, we can by Morse-Sard theorem find level sets of $S$ arbitrarily close to the set of minima, such that $||DS(x)|| \neq 0$ on that level set, and by compactness of level sets close enough to the bounded set of minima and smoothness of $S$, we can find a lower bound $||DS(x)|| \geq \Delta_1 >0$ on any such level set. Now we take $S_{\omega}$ instead of $S$, acting on a suitable Banach space (see \cite{Angenent:93} or Section \ref{s:dissipation}). If the stable and unstable manifolds of $\mathbb{T}_{\omega}$ intersect transversally, $D^2S_{\omega}$ is in a certain sense non-degenerate \cite{Angenent:93}, the constant $A$ can be interpreted as the angle of splitting, and $\Delta_1 \sim $ the lower bound on the norm of the Fr\'{e}chet derivative of $S_{\omega}$ on a level set of $S_{\omega}$ is proportional to $A$ (we actually take the square of the norm and find a level set which maximizes $\Delta_1$). If, however, we only assume that the set of minima of $S_{\omega}$ is bounded, we analogously to the finite-dimensional case obtain $\Delta_1>0$ by an application of an infinite dimensional analogoue of the Morse-Sard theorem (see the proof of Proposition \ref{p:dissipation}, and Part II for further discussion and examples \cite{Slijepcevic:16}).

\vspace{1ex} 
\noindent (8) In the cases considered by Nekhoroshev such as the Arnold's example \cite{Arnold:64}, both $\Delta_0$ and $\Delta_1$ are exponentially small in $\ve$, which is consistent with \cite{Nekhoroshev:77}. They are, however, polynomial in $\mu$ - see below.

\vspace{1ex} 
\noindent (9) The "fast diffusion" (with respect to the perturbation $\mu$) for sufficiently small $\mu$ has the drift acceleration $O(\mu/|\log \mu|)$, as conjectured by Lochak, and proved for a class of a-priori unstable systems similar to (\ref{d:unstable}) (also allowing the dimension of the rotator variable $v$ to be $\geq 1$) with non-transveral intersection of whiskers by Berti, Biasco, Bolle, and Treschev \cite{Berti:03,Treschev:04,Treschev:12}. In \cite{Berti:03}, it was established that this is under certain assumptions the largest possible drift acceleration. If we only assume (S2), then the drift acceleration is $O(\mu^2)$, as shown in \cite{Bernard:96,Bessi:96}.

\vspace{1ex}
\noindent (10) One could extend the results in the Arnold's example to arbitrary $\mu$ (and other cases of (\ref{d:unstable})), by developing a computer-assisted proof verifying (S1) in the spirit of \cite{Simo:00}, by using all the a-priori bounds we develop here.
\end{remarks}

}
	
\subsection{The structure of the proof and notation}

In Sections 2-4, we consider the Lagrangian of the type $L(q,q_t,t)=\frac{1}{2}q_t^2+V(q,t)$ on $\mathbb{T}^N \times \mathbb{R}^N \times \mathbb{T}$, and develop general tools for constructing solutions and "shadowing" invariant measures of Euler-Lagrange equations. Specifically, in Section \ref{s:existence} we prove existence of solutions of the formally gradient dynamics of the action on unbounded domains, on function spaces large enough to contain solutions of Euler-Lagrange equations with merely a bounded momentum. We then in Section \ref{s:equilibrium} show that invariant sets with respect to the formally gradient dynamics bounded in norm contain in its closure the Euler-Lagrange equations. The key tool is then developed in Section \ref{s:measure}, where we extend these ideas to construction of shadowing invariant measures.

In Sections \ref{s:five}-\ref{s:invariant}, we then focus on the a-priori unstable Lagrangian (\ref{d:unstable}), and construct invariant sets of the formally gradient flow as required by the general setting. The construction is based on the following simple idea. Assume for the moment that $\xi$ is an abstract continuous semiflow on a separable metric space $\mathcal{X}$, and let $\mathcal{A}$, $\tilde{\mathcal{B}}$ be subsets of $\mathcal{X}$. Assume they satisfy the following:

\begin{description}
	\item[(B1)] $\tilde{\mathcal{B}}$ is $\mathcal{A}$-relatively $\xi-$invariant set. That means if $q(s_0) \in \tilde{\mathcal{B}}$, and if there exists $s_1 > s_0$ such that for all $s \in [s_0,s_1]$, $\xi^{s-s_0}(q)\in  \mathcal{A}$, then for all $s \in [s_0,s_1]$, $\xi^{s-s_0}(q)\in  \tilde{\mathcal{B}}$.
	
	\item[(B2)] There exists $\lambda >0$ such that, if $q(s_0) \in \mathcal{A} \cap \tilde{\mathcal{B}}$, then for all $s \in [s_0,s_0+\lambda]$,  $q(s) \in \mathcal{A}$.
\end{description}

\begin{lemma} \label{l:combine}
	Assume $\mathcal{A}$, $\tilde{\mathcal{B}}$ are subsets of a separable metric space $\mathcal{X}$ satisfying (B1), (B2) with respect to a continuous semiflow $\xi$ on $\mathcal{X}$. Then $\mathcal{B}=\mathcal{A} \cap \tilde{\mathcal{B}}$ is $\xi$-invariant.
\end{lemma}

\begin{proof}
	Assume the contrary and find a semi-orbit $q(s)$ of $\xi$, $s \geq s_0$, $q(s_0) \in \mathcal{A}\cap \tilde{\mathcal{B}}$ which violates the conclusion of the Lemma. Let 
	$$ s_2 = \sup \left\lbrace s_1\geq 0, \: q(s)\in \mathcal{A}\cap \tilde{\mathcal{B}}
	\text{ for all } s\in [0,s_1] \right\rbrace .$$
	Then if $s_3 = \max \lbrace s_0, s_2-\lambda/2 \rbrace$, by construction $q(s_3) \in \mathcal{A}\cap \tilde{\mathcal{B}}$. Now by (B2), for all $s \in [s_3,s_3 + \lambda]$, we have $q(s) \in \mathcal{A}$, and by (B1), for all $s \in [s_3,s_3 + \lambda]$ we obtain $q(s) \in \tilde{\mathcal{B}}$. But $s_3 + \lambda > s_2$, which is a contradiction. 
\end{proof}
We construct the sets $\mathcal{A},\tilde{\mathcal{B}}$ as follows. Let $\xi$ be the "formally gradient semiflow" introduced in Section \ref{s:existence}. We define the set $\mathcal{A}$ in Section \ref{s:six} by very roughly fixing the times of "jumps" between invariant tori $\mathbb{T}_{\omega}$. The set $\tilde{\mathcal{B}}$ satisfying the conditions (B1), (B2) is then built using the "Russian doll" approach. We construct a decreasing sequence of sets $\mathcal{B}_1 \supset \mathcal{B}_2 \supset ... \supset \mathcal{B}_6$, showing inductively in each step that they are $\mathcal{A}$-relatively $\xi$-invariant. Finally, in $\mathcal{B}_6$ we establish sufficient control to also prove by an energy method the condition (B2), as required by Lemma \ref{l:combine}. Specifically, in Sections \ref{s:five}, \ref{s:hetero} we recall the known results on homoclinic, heteroclinic orbits and the Peierl's barrier, and prove a-priori bounds on minimizing homoclinics and heteroclinics. In Section \ref{s:six}, we fix a sequence of tori $\mathbb{T}_{\omega_k}$, $k \in \mathbb{Z}$, and construct a rough approximation of a shadowing orbit. In our method, it is not required that this approximation is very precise. We then in Sections \ref{s:seven} and \ref{s:eight} construct sets $\mathcal{B}_1$, $\mathcal{B}_2$, $\mathcal{B}_3$, $\mathcal{B}_4$, by establishing $L^{\infty}$-bounds, as well as weighted $L^2$-bounds on the first, second and third derivatives. The core of the argument is then in Sections \ref{s:dissipation} and \ref{s:upper}, where we establish local control of the dynamics $\xi$ between two "jumps" roughly independently of the behavior away from these jumps. The argument relies on precise control of the local "energy", "energy dissipation" and "energy flux" with respect to $\xi$, where "energy" is in this case the Lagrangian action.  The approach is inspired by the results from \cite{Gallay:01,Gallay:12,Gallay:14,Gallay:15}. One of the novelties is the use of an infinite-dimensional version of the Morse-Sard Theorem, enabling us to establish lower $>0$ bounds on the dissipation on certain "action" levels arbitrarily close to the minimal action along a heteroclinic orbit. We then show that these action levels can not be crossed by $\xi$, as the action dissipation is larger than the action flux through the boundary of the considered interval $t\in [\tilde{T}_k-L,\tilde{T}_k+L]$, where $\tilde{T}_k$ is the approximate time of a "jump" and $L$ the minimal time between jumps, for $L$ large enough.

We thus establish uniformly local control of the dynamics $\xi$, enabling us to construct invariant sets independently of the number of "jumps" between invariant tori. This allows the number of jumps to be infinite, and establishes "variational" control for all $t \in \mathbb{R}$. We complete the construction of an invariant set $\mathcal{B}$, as a function of a given sequence of tori to be shadowed, in Section \ref{s:invariant}. The last step of the construction is somewhat subtle (the set $\mathcal{B}_6$), and uses in a topological way the existence of the semiflow $\xi$.

In Sections 13 and 14, we then focus on proving Theorems \ref{t:main1}-\ref{t:main3}. Several technical results needed throughout the paper are given in Appendices A-D at the end of the paper. 

\begin{notation} We denote by $\mathcal{X}$ the phase space on which we consider the formally gradient semiflow $\xi$, introduced in Section \ref{s:existence}. The elements of $\mathcal{X}$ are always denoted by $q$, and in the case $N=2$ consistently with $q=(u,v)$. By $\mathcal{E}$ we denote the set of equilibria of $\xi$, which by correspondences $\pi_t$ and $\pi$ specified in Section \ref{s:equilibrium} are the solutions of the Euler-Lagrange equations. The fixed velocities $v_t$ specifying the invariant tori $\mathbb{T}_{\omega}$ are denoted by $\omega$. 
The sequence of "jump" times is given with $\tilde{T}_k$, and $T_k = \tilde{T}_k \mod 2\pi$, $T_k \in [0, 2\pi)$. An approximate shadowing orbit is denoted by $q^0$, and defined in Section \ref{s:six}. Various constants are fixed throughout the proof: the constants $L$ (minimal length of time $t$ between the "jumps"), $M$ (maximal oscillations of the "rotator" variable $v$) depend on the particular choices of $f$, and the sequence of tori to be shadowed. The constants $R$ (introduced in (S1)) and $\varpi=\max \lbrace |\omega^-|,|\omega^+|,1 \rbrace$ depend on the choice of $f$ and the segment $[\omega^-,\omega^+]$ in a region of instability. We denote by $c_1,c_2,...$ fixed absolute constants, though they may change within the proof when introduced.	The symbol $g \ll h$ always means $g \leq c \cdot h$ for some absolute constant $c$. If the constant depends on $\ve$ or $f$, we write $g \ll_{\ve} h$ or $g \ll_{f} h$. We use $g \wedge h$ instead of $\min \lbrace g,h \rbrace$ and $g \vee h$ instead of $\max \lbrace g,h \rbrace$. Given two measurable spaces $(\Omega_1,\mathcal{F}_1)$ and $(\Omega_2,\mathcal{F}_2)$, where $\mathcal{F}_1$ and $\mathcal{F}_2$ are $\sigma$-algebras, if $\mu$ is a probability measure on $(\Omega_1,\mathcal{F}_1)$ and $g : \Omega_1 \rightarrow \Omega_2$ measurable, we denote by $f^*\mu$ the pulled measure on $(\Omega_2,\mathcal{F}_2)$ given with $f^*\mu(\mathcal{C})=\mu(f^{-1}(\mathcal{C}))$.
\end{notation}

\vspace{2ex}

\vspace{2ex}

\centerline{I: VARIATIONAL CONSTRUCTION OF ORBITS AND INVARIANT MEASURES}

\vspace{2ex}

\section{Existence of solutions and the function spaces} \label{s:existence}

In this and the next two sections, we consider the Lagrangian $L(q,q_t,t)=\frac{1}{2}q_t^2+V(q,t)$, $L : \mathbb{T}^N \times \mathbb{R}^N \times \mathbb{T} \rightarrow \mathbb{R}$, where $\mathbb{T}$ is always parametrized with $[0,2\pi)$, and $V$ is $C^2$, $2\pi$-periodic in all coordinates. Here we prove existence of solutions of the formally gradient dynamics of the action, given with:
\begin{equation}
q_s = q_{tt} - \frac{\partial}{\partial q}V(q,t). \label{r:Grad}
\end{equation}
We consider (\ref{r:Grad}) on the Banach space $\mathcal{X}$ of all  $q=(u,v) \in H^2_{\text{loc}}(\mathbb{R})^N$, such that $q_t \in H^1_{\text{ul}}(\mathbb{R})^N$ (in Appendix A we recall the definition of the uniformly local spaces $H^k_{\text{ul}}(\mathbb{R})^N$, the associated norms and their properties). The norm on $\mathcal{X}$ is given with
$$ 
||q||_{\mathcal{X}_{\text{ul}}}=\left( q(0)^2+||q_t||^2_{H^1_{\text{ul}}(\mathbb{R})^N} \right)^{1/2}.
$$

We will frequently require an alternative, weaker, localized topology on $\mathcal{X}$, induced by the $H^1_{\text{loc}}(\mathbb{R})^N$ topology on $\mathcal{X}$. We use the notation $\mathcal{X}_{\text{ul}}$ and $\mathcal{X}_{\text{loc}}$ respectively to distinguish the topologies on the same set $\mathcal{X}$. Now $\mathcal{X}_{\text{loc}}$ is a normed (but not complete) space with the norm
$$
||q||_{\mathcal{X}_{\text{loc}}}=\left( \int_{-\infty}^{\infty} e^{-|t|} \left( q^2(t)+q_t^2(t) \right) dt \right)^{1/2}.
$$
Denote by $\varphi^yq(t)=q(y+t)$ the translation (corresponding to the time evolution of the Euler-Lagrange equations, as discussed in detail in the next section). By definition of uniformly local spaces, $\varphi$ is a continuous flow on $\mathcal{X}_{\text{ul}}$, and by definition of the topology, also on $\mathcal{X}_{\text{loc}}$. The existence of solutions of (\ref{r:Grad}) is given with:

\begin{thm} \label{t:exist} Assume $q^0 \in \mathcal{X}$ at $s_0$ is the initial condition. Then:
	
	(i) There exists unique $q(s)\in \mathcal{X}$ for all $s \in [ s_0,\infty )$, $q(s_0)=q^0$, so that
	\[
	q-q^0 \in C^0(\left[ s_0,\infty \right),H^2_{\text{ul}}(\mathbb{R})^N) \cap  C^1(\left( s_0,\infty \right),H^2_{\text{ul}}(\mathbb{R})^N),
	\]
	 and so that for all $s>s_0$, $q$ is a solution of (\ref{r:Grad}).
	
	(ii) The system (\ref{r:Grad}) generates a continuous semiflow $\xi$ on $\mathcal{X}_{\text{loc}}$ and $\mathcal{X}_{\text{ul}}$.
	
	(iii) The semiflow $\xi$ and the flow $\varphi$ commute.
	
	(iv) If $V \in H^k(\mathbb{T}^{N+1})$, $k \geq 1$, then for all $s > s_0$ we have that $q(s) \in H^k_{\text{ul}}(\mathbb{R})^N$.
 \end{thm}

The proof of Theorem \ref{t:exist} follows the standard approach for semilinear parabolic equations \cite{Henry:81}, only on a larger space than usual, and is given in the Appendix A.

\begin{remark} \label{rem:Tonelli}
	All the results of the Part I hold if we consider a more general, $C^2$ Tonelli Lagrangian $L(q,q_t,t)$ on $\mathbb{T}^N \times \mathbb{R}^N \times \mathbb{T}$ (see e.g. \cite{Mather:93,Sorrentino10} for background and definitions). In that case, instead of (\ref{r:Grad}), we consider
	\begin{equation}
	q_s = q_{tt} + \left( \frac{\partial^2}{\partial q_t^2}L(q,q_t,t) \right)^{-1} \left(   \frac{\partial^2}{\partial q \partial q_t}L(q,q_t,t) +  \frac{\partial^2}{\partial t \partial q_t}L(q,q_t,t) - \frac{\partial}{\partial q}L(q,q_t,t)  \right). \label{r:Tonelli}
	\end{equation}
For example, stationary points of (\ref{r:Tonelli}) are indeed solutions of the Euler-Lagrange equations. One can in particular verify that (\ref{r:Tonelli}) on $\mathcal{X}$ is an extended gradient system in the sense of \cite{Gallay:01,Gallay:12}, and that the proofs of Theorems \ref{t:contains} and \ref{t:variational} can be generalized to hold. We develop the theory in a simpler case for clarity of the introduced ideas.
\end{remark}

\section{Existence of Euler-Lagrange orbits in invariant sets} \label{s:equilibrium}

Consider equilibria (or stationary solutions) of (\ref{r:Grad}), i.e. the solutions of
\begin{equation}
  q_{tt} = \frac{\partial}{\partial q}V(q,t). \label{r:ELN}
\end{equation}
We denote by $\mathcal{E}$ the set of all $q \in \mathcal{X}$ satisfying (\ref{r:ELN}). Let $\pi_t : \mathcal{E} \rightarrow \mathbb{T}^N \times \mathbb{R}^N \times \mathbb{T}$, $\pi_t(q)=(q(t) \mod 2\pi, q_t(t),t \mod 2\pi)$. Then by the continuous dependence on initial conditions of (\ref{r:ELN}), $\pi_t$ is continuous in both of the topologies $\mathcal{X}_{\text{ul}}$, $\mathcal{X}_{\text{loc}}$ induced on $\mathcal{E}$. Furthermore, we have that for any $t_1,t_2 \in \mathbb{R}$, $\pi_{t_1+t_2} \circ \varphi^{t_1}=\phi^{t_2} \circ \pi_{t_1}$, i.e. $\mathcal{E}$ correspond to the solutions of (\ref{r:Grad}), and the $t$-translation on $\mathcal{E}$ corresponds to the $t$-evolution of a solution of (\ref{r:ELN}). Similarly, if $\pi : \mathcal{E} \rightarrow \mathbb{T}^N \times \mathbb{R}^N $ is the projection in the zero-coordinate, $\pi(q)=(q(0) \mod 2\pi, q_t(0))$, and $S = \varphi^{2 \pi}$ is the $t$-translation for one period, then $\pi \circ S = \sigma \circ \pi$.

The space $\mathcal{X}$ is large enough so that the projection of $\mathcal{E}$ to $\mathbb{T}^N \times \mathbb{R}^N$ contains a rich set of orbits. Specifically, we show in Lemma \ref{l:Xcontain} in the Appendix A that if a solution of (\ref{r:ELN}) satisfies $||q_t||_{L^{\infty}(\mathbb{R})^N} < \infty$, then $q \in \mathcal{X}$, thus $q \in \mathcal{E}$.

The result of this section is that, to construct elements of $\mathcal{E}$, i.e. to find solutions of (\ref{r:ELN}), it suffices to find an invariant set of (\ref{r:Grad}):

\begin{thm} \label{t:contains}
	Assume $\mathcal{B}$ is a non-empty, $\xi$-invariant subset of $\mathcal{X}$, bounded in the $\mathcal{X}_{ul}$-norm. Then there is a $q\in \mathcal{E}$ in the closure of $\mathcal{B}$ in $\mathcal{X}_{\text{loc}}$.
\end{thm}

\begin{remark} \label{r:extended}
We comment why Theorem \ref{t:contains} is not entirely straightforward. The system (\ref{r:Grad}) belongs to a class of extended gradient systems (or formally gradient systems), introduced in a general setting in \cite{Gallay:01,Gallay:12}. These are dynamical systems which, when considered on bounded domains, are gradient-like; but on an unbounded domain may behave differently. For example, for the system of equations $q_{tt}= \Delta q - \partial_q V(q)$, if we require that $q$ decays fast enough at $\infty$, the system is gradient-like, and $\omega$-limit sets (considered in a sufficiently weak topology so that orbits uniformly bounded in norm are relatively compact) by LaSalle principle consist of equilibria. If $q$ is merely bounded, and we consider solutions $q : \mathbb{R}^M \rightarrow \mathbb{R}^N$, then for $M=1,2$, $\omega$-limit sets may contain non-equilibria, but always contain at least one equilibrium, and for $M \geq 3$, there are examples of $\omega$-limit sets without equilibria at all \cite{Gallay:01,Gallay:12}. As in (\ref{r:Grad}), $M=1$ (the dimension of the variable $t$), Theorem \ref{t:contains} holds. We adapt the proof from \cite{Gallay:01,Gallay:12} to our setting. 
\end{remark}
In the first lemma below, we establish a bound on the action dissipation, then we establish relative compactness of the required set, and finally construct $q \in \mathcal{E}$ by a variational argument. 

\begin{lemma} \label{l:41}
	Assume $q(s)\in \mathcal{B}$, $s \geq s_0$ is an orbit of $\xi$. There exists an absolute constant $c_1>0$ and a sequence of relaxation times $s_n$, $n\in \mathbb{Z}$, so that
	 $$ \int_{-n}^n q_s(t,s_n)^2 dt \leq \frac{c_1
	 	}{n} \left( ||q(s_0)||^2_{\mathcal{X}_{\text{ul}}} + 1 \right).$$
\end{lemma}

\begin{proof}
Let $\delta > 0$, and denote by $E_{\delta}$, $D_{\delta}$ the weighted action and action dissipation,
\begin{equation*}
E_{\delta}(q)  =  \int_{-\infty}^{\infty} e^{-\delta|t|}L(q(t),q_t(t),t)dt, \hspace{60pt}
D_{\delta}(q)  =  \int_{-\infty}^{\infty} e^{-\delta|t|} q_s(t)^2 dt.
\end{equation*}
It is straightforward to verify that $E_{\delta}$, $D_{\delta}$ are well-defined on $\mathcal{X}$ (the integrals are absolutely integrable). Furthermore, we can differentiate with respect to $s$, by calculating on a dense subset and then extending the final result to the entire $\mathcal{X}$ by continuity. We partially integrate and apply the Young's inequality in the second line below:
\begin{align*}
\frac{d}{ds} E_{\delta}(q(s)) & =  \int_{-\infty}^{\infty} e^{-\delta|t|} \left( q_tq_{ts}+\frac{\partial}{\partial q}V(q,t)q_s \right) dt 
 \leq  \delta \int_{-\infty}^{\infty} e^{-\delta|t|}|q_tq_s| dt - \int_{-\infty}^{\infty}e^{-\delta|t|} q_s^2 dt  \\
&\leq  \frac{\delta^2}{2}  \int_{-\infty}^{\infty}e^{-\delta|t|}q_t^2 dt +\frac{1}{2}D_{\delta}(q(s))-D_{\delta}(q(s))  \\
&\leq  \delta^2 E_{\delta}(q(s)) - \frac{1}{2}D_{\delta}(q(s)). \nonumber
\end{align*}
Now by the Gronwall Lemma, integrating it over $[s_0,s_0+1/\delta^2 ]$, we have
\begin{equation}
e^{-1}E_{\delta}(q(s_0+1/\delta^2)) + \frac{1}{2}\int_{s_0}^{s_0+1/\delta^2}e^{-(s-s_0)\delta^2}D_{\delta}(q(s)) ds \leq E_{\delta}(q(s_0)). \label{r:grnwll}
\end{equation}
It is easy to see that $E_\delta(q)$ can be bounded by $O\left( \left(||q_t||^2_{L^2_{\text{ul}}(\mathbb{R})^2}+1 \right) /\delta \right)$, thus by definition of the $\mathcal{X}_{\text{ul}}$-norm,
$$ E_\delta(q) \ll \frac{1}{\delta} \left( ||q||^2_{\mathcal{X}_{\text{ul}}} +1 \right). $$
Also by definition, $L(q,q_t,t) \geq 0$, thus $E_\delta(q) \geq 0$.
Inserting it in (\ref{r:grnwll}) we otain
$$ \int_{s_0}^{s_0+1/\delta^2}D_{\delta}(q(s)) ds \ll \frac{1}{\delta} \left( ||q(s_0)||^2_{\mathcal{X}_{\text{ul}}} + 1 \right),$$
thus by definition of $D_{\delta}$,
\begin{equation}
 \int_{s_0}^{s_0+1/\delta^2}\left( \int_{-1/\delta}^{1/\delta}q_s(t)^2dt \right) ds \ll \frac{1}{\delta} \left( ||q(s_0)||^2_{\mathcal{X}_{\text{ul}}} + 1 \right). \label{r:Dirichlet}
\end{equation}
Now set $\delta = 1/n$. From (\ref{r:Dirichlet}) it follows immediately that there exists the required $s_n$, $s_0 \leq s_n \leq s_0+n^2$, so that the claim holds with $c_1$ being the absolute constant in (\ref{r:Dirichlet}).	
\end{proof}

\begin{lemma} \label{l:42}
	If $\mathcal{B}$ is bounded in $\mathcal{X}_{ul}$, then it is relatively compact in the closure $\bar{\mathcal{X}}_{\text{loc}}$ of $\mathcal{X}_{\text{loc}}$ in $H^1_{\text{loc}}(\mathbb{R})^N$.
\end{lemma}

\begin{proof}
	It is easy to check from the definition of the $\mathcal{X}_{\text{ul}}$-norm, that boundedness of $\mathcal{B}$ in $\mathcal{X}_{ul}$ implies boundedness of $q|_{[-n,n]}$ in $H^2([-n,n])^N$, uniformly for $q \in \mathcal{B}$, for any $n>0$. Thus for any sequence $q^{(j)}$ in $\mathcal{B}$, by compact embedding we can find a subsequence (again denoted by $q^{(j)}$) so that $q^{(j)}|_{[-n,n]}$ converges in $H^1([-n,n])^N$; and by diagonalization a further subsequence converging in $H^1_{\text{loc}}(\mathbb{R})^N$ (which induces the $\mathcal{X}_{\text{loc}}$ topology by definition).
\end{proof} 

\begin{lemma} \label{l:43} Assume $s_n \rightarrow \infty$ as $n \rightarrow \infty$ is a sequence of times such that 
$$ \lim_{n \rightarrow \infty}	\int_{-n}^n q_s(t,s_n)^2 dt \rightarrow 0.$$
Then any limit point of $q(s_n)$ in $\bar{\mathcal{X}}_{\text{loc}}$ is in $\mathcal{E}$.
\end{lemma}

\begin{proof}
	Fix $m \in \mathbb{Z}$, and choose a test function $g \in H^1_0([-m,m])^2$ (that is, vanishing at $t=-m,m$). Now by partial integration and Cauchy-Schwartz,
	\begin{eqnarray}
	 \int_{-m}^{m}\left( q_t(t,s_n)g_t(t)+\frac{\partial V}{\partial q}(q(t,s_n),t)g(t) \right) dt
	&=& \int_{-m}^{m} \left( -q_s(t,s_n)g(t) \right)dt \nonumber \\
	& \leq & \left( \int_{-m}^{m}q_s(t,s_n)^2dt \right)^{1/2} ||g||_{L^2([-m,m])^2}. \nonumber
	\end{eqnarray} 
Now if $q(s_n,.)$ converges to some $q^0$ in $\bar{\mathcal{X}}_{\text{loc}}$, their restrictions to $[-m,m]$ converge in $H^1([-m,m])^2$. We deduce that
$$ \int_{-m}^{m}\left( q^0_t(t)g_t(t)+\frac{\partial V}{\partial q}(q^0(t),t)g(t)dt \right) dt = 0.$$
As it holds for an arbitrary test function $g$, we conclude that the variation of the action at $q^0$ is 0, so $q^0$ is a solution of (\ref{r:ELN}). By construction, $q_t \in L^2_{\text{ul}}(\mathbb{R})^N$, thus by Lemma \ref{l:Xcontain} we have that $q \in \mathcal{E}$.
\end{proof}

Theorem \ref{t:contains} follows by combining Lemmas \ref{l:41}, \ref{l:42} and \ref{l:43}.

\section{Construction of shadowing invariant measures} \label{s:measure}

In Section \ref{s:equilibrium}, we showed how to construct a solution of (\ref{r:ELN}), given an invariant set with respect to the semiflow $\xi$. In this section we develop a measure-theoretical analogue to that. In the first subsection, we propose an abstract notion of a shadowing measure and derive its properties. In the second subsection, we prove existence of such measures, if a certain sub-algebra of Borel sets with certain invariance property with respect to the semiflow $\xi$ is given.

\subsection{Shadowing of invariant measures in an abstract setting} \label{ss:shadowing}

We propose an abstract definition of a shadowing invariant measure as follows. In this subsection we will always consider a measurable space $(\Omega,\mathcal{F})$, where $\Omega$ is a compact metric space and $\mathcal{F}$ the Borel $\sigma$-algebra. Let $S$ be a homeomorphism on $\Omega$ and $\mu$ a $S$-invariant probability measure on $(\Omega,\mathcal{F})$. Recall that $\mu$ is a factor of a $S$-invariant probability measure $\nu$ on the same space $(\Omega,\mathcal{F})$, if there exist two Borel-measurable sets $\mathcal{M}_1$, $\mathcal{M}_2$ such that $\mu(\mathcal{M}_1)=1$, $\nu(\mathcal{M}_2)=1$, and a measurable map $\theta : \mathcal{M}_2 \rightarrow \mathcal{M}_1$, such that $\theta \circ S|_{\mathcal{M}_2} = S \circ \theta|_{\mathcal{M}_1}$, and such that $\theta$ pulls the measure $\nu$ into $\mu$, i.e. for any set $\mathcal{D} \in \mathcal{F}$, $\nu(\theta^{-1}(\mathcal{D}))=\mu(\mathcal{D}))$ (where we extended $\theta$ to a measurable function on the entire $\Omega$ in an arbitrary way).

\begin{defn} Let $\mathcal{G}$ be a $\sigma$-subalgebra of $\mathcal{F}$. We say that a $S$-invariant Borel-probability measure $\nu$ $\mathcal{G}$-shadows a $S$-invariant probability measure $\mu$ on $(\Omega,\mathcal{F})$, if $\mu$ is a factor of $\nu$, and if for each $\mathcal{D} \in \mathcal{G}$, we have $\mu(\mathcal{D})=\nu(\mathcal{D})$. 
\end{defn}
	
We now in several lemmas show relation of the shadowing property to the support of a measure, ergodicity and entropy. To control certain topological properties of the shadowing measure, we introduce the notion of the conditional support of a probability measure $\mu$ with respect to a $\sigma$-subalgebra of Borel sets $\mathcal{G}$. We denote it by $\supp(\mu | \mathcal{G})$, and define it as the set of all $x \in \Omega$ such that there exists a sequence of closed sets $\mathcal{D}_j \in \mathcal{G}$, $j \in \mathbb{N}$, $\mu(\mathcal{D}_j)>0$ such that $\cap_{j \in \mathbb{N}}\mathcal{D}_j  = \lbrace x \rbrace$. Furthermore, let $\supp^c(\mu | \mathcal{G})$ be the complement-conditional support, defined as the set of all $x \in \Omega$ for which there exists an open $\mathcal{D} \in \mathcal{G}$ such that $\mu(\mathcal{D})=0$ and $x \in \mathcal{D}$. Clearly, if $\mathcal{G}=\mathcal{F}$, we have $\supp(\mu | \mathcal{F})=\supp(\mu)$, and $\supp^c(\mu | \mathcal{F})=\supp(\mu)^c$. In general, it is easy to deduce from the definition of the support of a measure that we have
\begin{equation}
	\supp(\mu | \mathcal{G}) \subseteq \supp(\mu) \subseteq \supp^c(\mu | \mathcal{G})^c.
\end{equation} \label{l:support}
The following Lemma follows directly from the definitions:
\begin{lemma}
	Assume that $\nu$ $\mathcal{G}$-shadows $\mu$. Then $\supp(\nu | \mathcal{G})=\supp(\mu | \mathcal{G})$ and $\supp^c(\nu | \mathcal{G})=\supp^c(\mu | \mathcal{G})$. 
\end{lemma}
The relation of shadowing to ergodicity is important and somewhat more involved:
\begin{lemma} \label{l:muergodic}
	Assume that $\nu$ $\mathcal{G}$-shadows $\mu$, that $\mu$ is $S$-ergodic, and that $\mathcal{G}$ satisfies the following: for each $\mathcal{D} \in \mathcal{G}$, $\theta^{-1}(\mathcal{M}_1 \cap \mathcal{D}) \subset \mathcal{D}$. Then almost every measure in the ergodic decomposition of $\nu$ $\mathcal{G}$-shadows $\mu$.  
\end{lemma}

\begin{proof} Consider the ergodic decomposition of $\nu$, i.e. a Borel-probability measure $\chi$ on the compact, metrizable space of probability measures $\mathcal{M}(\Omega)$ (equipped with the weak$^*$-topology), such that $\chi$-a.e. measure is $S$-invariant and ergodic, and such that the usual representation formula for $\nu$ in terms of $\chi$ holds \cite{Walters}. Then it is straightforward to check by verifying the definition of the ergodic decomposition \cite{Walters} that $(\theta^*)^*\chi$ is the ergodic decomposition of $\mu$, where $(\theta^*)^*$ is the double pull defined in a natural way. However, the ergodic decomposition is unique, and as $\mu$ is ergodic, $(\theta^*)^*\chi$ must be concetrated on $\mu$. That means that for $\chi$-a.e. measure $\tilde{\nu}$ (i.e. almost every measure in the ergodic decomposition of $\nu$), we have $\theta^*(\tilde{\nu})=\mu$. By construction, $\mu$ is then a factor of $\tilde{\nu}$.
	
It remains to show the shadowing property. As $\mu(\mathcal{M}_1)=1$, we have
	\begin{equation}
		\tilde{\nu}(\mathcal{\mathcal{D}}) \geq \tilde{\nu}(\theta^{-1}(\mathcal{M}_1 \cap \mathcal{D})) = \mu ( \mathcal{M}_1 \cap \mathcal{D}) = \mu (\mathcal{D}), \label{r:factor}
	\end{equation}
	and analogously $\tilde{\nu}(\mathcal{D}^c)\geq \mu (\mathcal{D}^c)$. However, $1 = \mu (\mathcal{D})+\mu (\mathcal{D}^c)= \tilde{\nu}(\mathcal{\mathcal{D}}) +\tilde{\nu}(\mathcal{D}^c)=1$. We conclude that the equality in (\ref{r:factor}) must hold.
\end{proof}
Finally, we establish relation of shadowing to the metric (or Kolmogorov-Sinai) entropy $h_{\mu}(S)$ of a measure $\mu$.

\begin{lemma} If $\nu$ $\mathcal{G}$-shadows $\mu$, then $h_{\nu}(S) \geq h_{\mu}(S)$.
	\label{l:entropy}
\end{lemma}

\begin{proof} This holds, as entropy is always non-increasing under factor maps and $\mu$ is a factor of $\nu$ \cite{Walters}.
\end{proof}

\subsection{Variational construction of shadowing measures} In this subsection $\mathcal{X}$ will always be equipped with the topology $\mathcal{X}_{\text{loc}}$. Prior to the variational construction of measures, we introduce the required spaces and projections. Recall the projections $\pi : \mathcal{X} \rightarrow \mathbb{T}^N \times \mathbb{R}^N$, given with $\pi(q)=(q \mod 2\pi,q_t)$. Let $\hat{\mathcal{X}}$ be the quotient space induced by the relation of equivalence: $q \sim \tilde{q}$ whenever there is $k \in \mathbb{Z}$, such that $q-\tilde{q}=2k \pi$, and with the induced topology. Let $\iota : \mathcal{X} \rightarrow \hat{\mathcal{X}}$ be the canonical projection, 
and let $\hat{\mathbb{\xi}}$, $\hat{\varphi}$ be the induced semi-flow $\xi$ and flow $\varphi$ on $\hat{\mathcal{X}}$. By (\ref{r:grad}) and by definition, $\hat{\xi}$ and $\hat{\varphi}$ are well-defined. If $\hat{\mathcal{E}}=\iota(\mathcal{E})$, $S = \varphi^{2 \pi}$, $\hat{S}=\hat{\varphi}^{2 \pi}$ are the 
$2 \pi$-shifts in the variable $t$, and $\hat{\pi} : \hat{\mathcal{X}} \rightarrow \mathbb{T}^N \times \mathbb{R}^N$ is defined with $\hat{\pi}(\hat{q}) = (\hat{q},\hat{q}_t)$, then the following commutative diagrams hold:

$$
\begin{CD}
\mathcal{X} @>i>> \hat{\mathcal{X}} \\
@VV \xi   V @VV \hat{\xi}  V \\
\mathcal{X} @>i>> \hat{\mathcal{X}}
\end{CD} \hspace{50pt}
\begin{CD}
\mathcal{E} @>i>> \hat{\mathcal{E}} @>\hat{\pi}>>   \mathbb{T}^N \times \mathbb{R}^N\\
@VV S V @VV \hat{S} V @VV \sigma V \\
\mathcal{E} @>i>> \hat{\mathcal{E}} @>\hat{\pi}>>   \mathbb{T}^N \times \mathbb{R}^N
\end{CD}
$$
\vspace{1ex}

By the continuous dependence on  initial conditions of (\ref{r:EL}), $\pi|_{\mathcal{E}}$ and $\hat{\pi}|_{\hat{\mathcal{E}}}$ are continuous. As for notation, we will always denote the functions on the quotient set $\hat{\mathcal{X}}$ by $\hat{.}$\hspace{2pt}. To simplify the notation, the subsets and elements of $\mathcal{X}$ and $\hat{\mathcal{X}}$ will be denoted by the same symbol, as the meaning will always be clear from the context.

We now focus on constructing $\phi$-, or equivalently $\sigma$-invariant measures of (\ref{r:EL}) (we always implicitly assume that the measures are Borel probability measures). We denote by $\mathcal{M}(\mathcal{X})$, $\mathcal{M}(\hat{\mathcal{X}})$ and $\mathcal{M}(\mathbb{T}^N \times \mathbb{R}^N)$ the spaces of $S$-, $\hat{S}$-, respectively $\sigma$-invariant measures on these spaces, equipped with the weak$^*$-topology. Analogously we define $\mathcal{M}(\mathcal{E})$, $\mathcal{M}(\hat{\mathcal{E}})$. We always denote by $.^*$ the functions, flows and semi-flows pulled to these spaces of measures. By all the commutative relations established so far, it is straightforward to check that the objects below are well-defined, and that the following commutative diagrams hold:

$$
\begin{CD}
\mathcal{M}(\mathcal{X}) @>i^*>> \mathcal{M}(\hat{\mathcal{X}}) \\
@VV \xi^* V @VV \hat{\xi}^* V \\
\mathcal{M}(\mathcal{X}) @>i^*>> \mathcal{M}(\hat{\mathcal{X}})
\end{CD} \hspace{50pt}
\begin{CD}
\mathcal{M}(\mathcal{E}) @>i^*>> \mathcal{M}(\hat{\mathcal{E}}) @>\hat{\pi}^*>>   \mathcal{M}(\mathbb{T}^N \times \mathbb{R}^N)\\
@VV S^* V @VV \hat{S}^* V @VV \sigma^* V \\
\mathcal{M}(\mathcal{E}) @>i^*>> \mathcal{M}(\hat{\mathcal{E}}) @>\hat{\pi}^*>>   \mathcal{M}(\mathbb{T}^N \times \mathbb{R}^N)
\end{CD}
$$
\vspace{1ex}

\noindent (by definition, $S^*$, $\hat{S}^*$ and $\sigma^*$ are identities). Thus constructing invariant measures of (\ref{r:EL}), i.e. elements of $\mathcal{M}(\mathbb{T}^N \times \mathbb{R}^N)$, is equivalent to finding required objects in $\mathcal{M}(\hat{\mathcal{E}})$, i.e. fixed points of $\hat{\xi}^*$ on $\mathcal{M}(\hat{\mathcal{X}})$, or fixed points of $\xi^*$ on $\mathcal{M}(\mathcal{X})$. The approach to constructing such measures is as follows: we will construct an element $\mu \in \mathcal{M}(\hat{\mathcal{X}})$ (typically not supported on $\hat{\mathcal{E}}$), e.g. by embedding a Bernoulli shift. We will then find an element of $\nu \in \mathcal{M}(\hat{\mathcal{E}})$ which shadows $\mu$, as an element of the $\omega$-limit set of $\mu$ with respect to $\hat{\xi}^*$. To achieve the shadowing property, given a fixed $\mu \in \mathcal{M}(\hat{\mathcal{X}})$, we will require that a $\sigma$-subalgebra $\mathcal{G}$ of the $\sigma$-algebra of Borel sets on $\hat{\mathcal{X}}$ satisfies the following conditions:
\begin{itemize}
	\item[(M1)] {\bf The separation property.} There exists a Borel-measurable set $\mathcal{M}_1 \subset \hat{\mathcal{X}}$ such that $\mu(\mathcal{M}_1)=1$, and such that $\left\lbrace \mathcal{D} \cap \mathcal{M}_1, \: \mathcal{D} \in \mathcal{G} \right\rbrace$ generates all Borel-measurable sets on $\mathcal{M}_1$. Specifically, for each $q \in \mathcal{M}_1$, there exists $\mathcal{D}_q \in \mathcal{G}$ such that 
	if $q,\tilde{q} \in \mathcal{M}_1$, $q \neq \tilde{q}$, then
	$\mathcal{D}_q \cap \mathcal{D}_{\tilde{q}} = \emptyset$. Furthermore, for any $q \in \mathcal{M}_1$, $\mathcal{D}_{\hat{S}(q)}=\hat{S}(\mathcal{D}_q)$.
	
	\item[(M2)] {\bf The $\xi$-invariance.} For each $q \in \mathcal{M}_1$ and each $\mathcal{D} \in \mathcal{G}$, if $q \in \mathcal{D}$, then for all $s \geq 0$, $\hat{\xi}^s(q) \in \mathcal{D}$.
	
	\item[(M3)] {\bf Measurability.} If $\mathcal{M}_2 = \cup_{q \in \mathcal{M}_1}{\mathcal{D}_q}$, then the map $\hat{\theta} : \mathcal{M}_2 \rightarrow \mathcal{M}_1$ given with $\hat{\theta}(\mathcal{D}_q)=q$ is Borel-measurable. Specifically, $\mathcal{M}_2$ is Borel-measurable.
	
	\item[(M4)] {\bf The closed-sets property.} There exists a family $\mathcal{D}_i \in \mathcal{G}$ of closed sets, $i \in \mathcal{I}$, such that $\mathcal{G}$ is generated by this family (i.e. $\mathcal{G}$ is the smallest $\sigma$-algebra containing all $(\mathcal{D}_i)_{i \in \mathcal{I}}$). Furthermore, for each $i_1 \in \mathcal{I}$ there exists a sequence $i_n \in \mathcal{I}$, $n \in \mathbb{N}$ such that $\mathcal{D}_{i_n}$ are pairwise disjoint, and such that $\mu(\cup_{n=1}^{\infty}\mathcal{D}_{i_n})=1$.
\end{itemize}
In applications, (M1), (M3) and (M4) will follow relatively easily from the construction of $\mu$, and the focus will be on ensuring the $\xi$-invariance of the constructed $\sigma$-algebra $\mathcal{G}$. An important tool in the construction of the shadowing measure, already suggested in \cite{Slijepcevic:00}, is that $\xi^*$ and $\hat{\xi}^*$ are gradient-like semiflows with the Lyapunov function (given below for $\hat{\xi}^*$)
\begin{equation}
	\mathcal{\hat{L}}^*(\mu)=\int_{\hat{\mathcal{X}}} \int_0^{2\pi} L(q(t),q_t(t),t)dt \: d\mu(q)
\end{equation}
(see Lemma \ref{l:gradient} below). We fix $\mu \in \mathcal{M}(\hat{\mathcal{X}})$ and denote by $\mu(s)=\hat{\xi}^*(\mu,s)$ for $s \geq 0$, i.e. $\mu(0)=\mu$, and $\mu(s)$ is the pulled measure $\mu$ with respect to the map $\hat{\xi}^s$.

\begin{thm} \label{t:variational}{\bf Variational construction of shadowing measures.}
	Assume $\mu \in \mathcal{M}(\hat{\mathcal{X}})$ such that $||q_t(s)||_{H^1_{\text{ul}}(\mathbb{R})^N}$ is bounded on the support of $\mu(s)$, uniformly in $s \geq 0$. Assume $\mathcal{G}$ is a $\sigma$-subalgebra of Borel sets on $\hat{\mathcal{X}}$ satisfying (M1)-(M4). Then there exists $\nu \in \mathcal{M}(\hat{\mathcal{E}})$ which $\mathcal{G}$-shadows $\mu$.
	
	Furthermore, if $\mu$ is $\hat{S}$-ergodic, we can choose $\nu$ to be $\hat{S}$-ergodic.
\end{thm}	
To prove the theorem, we first construct a measure $\nu \in \mathcal{M}(\hat{\mathcal{E}})$ in two Lemmas, and then show that it is indeed the shadowing measure by using (M1)-(M4). 

\begin{lemma} \label{l:gradient} 
	The function $s \rightarrow \mathcal{\hat{L}}^*(\mu(s))$ is strictly decreasing, unless $\mu(0) \in \mathcal{M}(\hat{\mathcal{E}})$, in which case it is constant. Furthermore,
	\begin{equation}
		\frac{d}{ds}  \mathcal{\hat{L}}^*(\mu(s)) = - \int_{\tilde{\mathcal{X}}} \int_0^{2\pi} q_s(t)^2 dt \: d\mu(s)(q). \label{r:balancemeasure}
	\end{equation}
\end{lemma}

\begin{proof} We first note that by the uniform bound on $||q_t(s)||_{H^1_{\text{ul}}(\mathbb{R})^N}$, we have that $\mathcal{\hat{L}}^*(\mu(s)) < \infty$ for all $s \geq 0$. By the smoothening property Theorem \ref{t:exist}, (iv), for any $s > 0$ and any $q \in \supp \mu(s)$, such $q$ is smooth enough so that we can differentiate as follows:
	\begin{align*}
		\frac{d}{ds} \int_0^{2\pi}L(q,q_t,t)dt & = \int_0^{2\pi}\left(\frac{d}{ds}L(q,q_t,t) \right)dt 
		= \int_0^{2\pi}\left( q_tq_{ts}+\frac{\partial}{\partial q}V(q,t)q_s \right) dt ds  \\
		& =  \int_0^{2\pi} \left(-q_{tt}q_s+\frac{\partial}{\partial q}V(q,t)q_s \right)dt + q_t(2\pi)q_s(2\pi) -  q_t(0)q_s(0) \\ & = -\int_0^{2\pi}q_s^2 dt + q_t(2\pi)q_s(2\pi) -  q_t(0)q_s(0).  
	\end{align*}
	Now for $0 < s_0 < s_1$, we have
	\begin{align}
		\int_0^{2\pi}L(q(s_0),q_t(s_0),t)dt - \int_0^{2\pi}L(q(s_1),q_t(s_1),t) & = - \int_{s_0}^{s_1}\int_0^{2\pi}q_s(s,t)^2 dt ds \notag \\ & \quad + \int_{s_0}^{s_1} \left(q_t(s,2\pi)q_s(s,2\pi) -  q_t(s,0)q_s(s,0) \right)ds. \label{r:Ldifference}
	\end{align}
	By the dominated convergence theorem, we can extend (\ref{r:Ldifference}) also to $0 \leq s_0 < s_1$.
	By the assumptions we have for any $q \in \supp \mu$,
	\begin{align*}
		\int_{\tilde{\mathcal{X}}} \int_{0}^{s_1} \int_{0}^{2\pi} |q_t(s,t)q_s(s,t)|dt 
		ds d\mu(q) & \ll  \int_{\tilde{\mathcal{X}}} \int_{0}^{s_1} \left( \int_{0}^{2\pi} (q_t(s,t)^2 + q_{tt}(s,t)^2 dt \right)^{1/2} ds d\mu(q) \\
		& \ll \int_{\tilde{\mathcal{X}}} \int_{0}^{s_1}  ||q_t(s)||_{H^1_{\text{ul}}(\mathbb{R})^N} ds d\mu(q) \ll A \cdot s_1,
	\end{align*}
	where $A$ is the uniform bound $||q_t(s)||_{H^1_{\text{ul}}(\mathbb{R})^N}$. Thus without loss of generality, we can assume that the function $\int_{0}^{s_1}q_t(s,0)q_s(s,0)ds$ is absolutely integrable with respect to $\mu$ (otherwise we choose some other $T \in [0,2\pi)$ instead of $T=0$ and repeat the argument over the interval $[T,T+2\pi]$). By the $\hat{S}$-invariance of $\mu$, we now have for any $s_0$, $0 \leq s_0 < s_1$,
	\begin{equation*}
		\int_{\tilde{\mathcal{X}}} \int_{s_0}^{s_1}q_t(s,0)q_s(s,0)ds d\mu(q) =  \int_{\tilde{\mathcal{X}}} \int_{s_0}^{s_1}q_t(s,2\pi)q_s(s,2\pi)ds d\mu(q).
	\end{equation*}
	Integrating (\ref{r:Ldifference}) with respect to $\mu$, we now get for any $0 \leq s_0 < s_1$,
	\begin{equation*}
		\mathcal{\hat{L}}^*(\mu(s_0))-\mathcal{\hat{L}}^*(\mu(s_1)) = - \int_{\tilde{\mathcal{X}}} \int_{s_0}^{s_1}\int_0^{2\pi}q_s(s,t)^2 dt ds d\mu(q).
	\end{equation*}	
	By the Fubini theorem, we can swap integrals over $ds$ and $d\mu$, which completes the proof.
\end{proof}

\begin{lemma} \label{l:limmeasure} There exists $\nu \in \mathcal{M}(\hat{\mathcal{E}})$ which is a weak$^*$-limit of a subsequence of $\mu(s)$, $s \geq 0$.
\end{lemma}

\begin{proof} As $\mathcal{X}$ equipped with the localized topology is not complete, to establish compactness required for the construction of $\nu$, we need to consider its closure in $H^1_{\text{loc}}(\mathbb{R})^N$, denoted by $\mathcal{Y}$. Let $\hat{\mathcal{Y}}$ be the quotient set with the same relation of equivalence $\sim$ and the induced topology, and 
	$\hat{\mathcal{X}} \hookrightarrow \hat{\mathcal{Y}}$ the natural embedding. It is straightforward to check that the closure of the set of all $q$ satisfying $||q_t||_{H^1(\mathbb{R})} \leq A$ is compact in $\hat{\mathcal{Y}}$ (we choose representatives in $\mathcal{Y}$ such that $q(0)\in [0,2\pi]^N$ and find a convergent subsequence by diagonalization). Thus by the Banach-Alaoglu theorem, $\mu(s)$, $s\geq 0$ has a convergent subsequence $\mu(s_n)$ which converges to some measure $\nu$ on $\hat{\mathcal{Y}}$ in the weak$^*$ topology induced by the induced $H^1_{\text{loc}}(\mathbb{R})^N$ topology.
	
	It suffices to show that $\nu \in \mathcal{M}(\hat{\mathcal{E}})$. Choose $h \in H^1(\mathbb{R})^N$ with compact support, say in $[-2n\pi ,2n\pi]$, $n \in \mathbb{N}$. Then by the $\hat{S}$-invariance of $\mu$ in the third row below, we obtain
	\begin{align*} 
		\int_{\tilde{\mathcal{X}}} | \partial L(q,q_t,t)h| d\mu(s_n)(q) & = \int_{\tilde{\mathcal{X}}} \left| \int_{-2n\pi}^{2n\pi} \left(q_t(s_n,t)h_t(t)+DV(q(s_n,t),t)h(t)\right) dt \right| d\mu(q) \\
		& \leq ||h||_{L^2(\mathbb{R})^N} \left( \int_{\tilde{\mathcal{X}}} \int_{-2n\pi}^{2n\pi} q_s^2(s_n,t)dt d\mu(q) \right)^{1/2} \\
		& = (2n+1)^{-1/2} ||h||_{L^2(\mathbb{R})^N} \left( \int_{\tilde{\mathcal{X}}} \int_0^{2\pi} q_s^2(s_n,t)dt d\mu(q) \right)^{1/2},
	\end{align*}
	which by Lemma \ref{l:gradient} converges to zero. We thus have that for any $h \in H^1(\mathbb{R})^N$ with compact support, $\int_{\tilde{\mathcal{Y}}} | \partial L(q,q_t,t)h| d\nu(q)=0$. 
	By choosing such a countable, dense set of $h$, and by continuity, we conclude that $\nu$ is supported on the solutions of (\ref{r:EL}). By construction, $\nu$ is supported on $q$ such that $q_t \in L^2_{\text{ul}}(\mathbb{R})^N$, thus by Lemma \ref{l:Xcontain}, $\nu \in \mathcal{M}(\hat{\mathcal{E}})$.	
\end{proof}

\begin{proof}[Proof of Theorem \ref{t:variational}] Take $\nu$ constructed in Lemma \ref{l:limmeasure}. We apply results from subsection \ref{ss:shadowing} with $\Omega = \mathcal{Y}$, $\mathcal{Y}$ as in Lemma \ref{l:limmeasure}, thus compact and metrizable. As $\nu$ is a weak$^*$-limit of $\hat{S}$-invariant measures and $\hat{S}$ is continuous, $\nu$ is $\hat{S}$-invariant. It suffices to show that $\nu$ shadows $\mu$. We take as the factor the function $\hat{\theta}: \mathcal{M}_2 \rightarrow \mathcal{M}_1$ from the property (M3). Let $s_k >0$ be the sequence from Lemma \ref{l:limmeasure} such that $\nu$ is the weak$^*$ limit of $\mu(s_k)$.
	
	First we show that for each $\mathcal{D} \in \mathcal{G}$, $\nu(\mathcal{D})= \mu(\mathcal{D})$. It suffices to show it for the generator $\mathcal{D}_i$, $i \in \mathcal{I}$, from (M4). Choose $i_1 \in \mathcal{I}$, and find $i_n \in \mathcal{I}$, $n \in \mathbb{N}$ so that (M4) holds. Now by (M2) and (M3), we have $\mu(\mathcal{D}_{i_n} )=\mu(\mathcal{D}_{i_n} \cap \mathcal{M}_1) \leq \mu(s_k)(\mathcal{D}_{i_n})$ for all $n \in \mathbb{N}$. However, by (M4) and $\sigma$-aditivity of $\mu$ and $\mu(s_n)$, we get
	\begin{equation*}
		1 = \sum_{n=1}^{\infty} \mu(\mathcal{D}_{i_n}) \leq \sum_{n=1}^{\infty} \mu(s_k)(\mathcal{D}_{i_n})=\mu(s_k)(\cup_{n=1}^{\infty}\mathcal{D}_{i_n}) \leq 1,
	\end{equation*}
	thus equality must hold in all the terms. As $\nu$ is the weak$^*$-limit of $\mu(s_k)$ and $\mathcal{D}_{i_n}$ are closed, we have that for all $n\in \mathbb{N}$, 
	$\nu(\mathcal{D}_{i_n})\geq \limsup_{k \rightarrow \infty}\mu(s_k)(\mathcal{D}_{i_n}) = \mu(\mathcal{D}_{i_n})$. As $\mathcal{D}_{i_n}$, we analogously as above have
	$$
	1 = \sum_{n=1}^{\infty} \mu(\mathcal{D}_{i_n}) \leq \sum_{n=1}^{\infty} \nu(\mathcal{D}_{i_n})=\nu(\cup_{n=1}^{\infty}\mathcal{D}_{i_n}) \leq 1,
	$$
	thus again equality must hold in all the terms. 
	
	By (M1) and the definition of $\hat{\theta}$, $\hat{\theta}$ and $\hat{S}$ commute. As $\nu(\mathcal{D})=\mu(\mathcal{D})=\mu(\mathcal{D} \cap \mathcal{M}_1)$, and $\mathcal{D} \cap \mathcal{M}_1$ generate all Borel-measurable sets on $\mathcal{M}_1$, to show that $\hat{\theta}$ is measure-preserving, it suffices to show that for all $i \in \mathcal{I}$, 
	\begin{equation}
		\nu (\mathcal{D}_i)  = \nu(\hat{\theta}^{-1}(\mathcal{D}_i \cap \mathcal{M}_1)).  \label{r:musthold}
	\end{equation} Choose $i_1 \in \mathcal{I}$, and find a sequence $i_n \in \mathcal{I}$ so that (M4) holds. By definition of $\hat{\theta}$ in (M3), we have that $\mathcal{D}_{i_n} \subset \hat{\theta}^{-1}(\mathcal{D}_{i_n} \cap \mathcal{M}_1)$, thus $\nu(\mathcal{D}_{i_n}) \leq \nu( \hat{\theta}^{-1}(\mathcal{D}_{i_n} \cap \mathcal{M}_1))$.
	By (M1) and (M3), the sets $\hat{\theta}^{-1}(\mathcal{D}_{i_n}\cap \mathcal{M}_1)$, $n\in \mathbb{N}$ are pairwise disjoint. Analogously as above, from all of this and $\sum_{n=1}^{\infty}\nu(\mathcal{D}_{i_n})=1$ we conclude that (\ref{r:musthold}) must hold.
	
	If $\mu$ is $\hat{S}$-ergodic, we can find a $\hat{S}$-ergodic $\nu$ by Lemma \ref{l:muergodic}.	
\end{proof}

\vspace{2ex}

\centerline{II: INVARIANT SETS IN THE A-PRIORI UNSTABLE CASE}

\vspace{2ex}

\section{The homoclinic orbits} \label{s:five}

As of this section, we focus on the a-priori unstable case with the Lagrangian (\ref{d:unstable}) and $N=2$. In this section we recall the key properties of the Peierl's barrier function and stable and unstable manifolds of the invariant tori $\mathbb{T}_{\omega}$. The results of this section are standard (see \cite{Fathi,Sorrentino10} and references therein). As we were unable to find in the literature the a-priori bounds we require later, we give self-contained proofs.

For a fixed $\omega \in \mathbb{R}$, let $S^-_{\omega},S^+_{\omega}: \mathbb{R}^2 \rightarrow \infty$ be the Peierl's barrier functions defined with
\begin{align*}
S^-_{\omega}(t_0,v_0) & = \inf \left\lbrace \int_{-\infty}^{t_0} L_{\omega}(q(t),q_t(t),t)dt, \: q=(u,v)\in H^1_{\text{loc}}((-\infty,t_0])^2, q(t_0)=(\pi,v_0), \lim_{t\rightarrow -\infty }u(t)=0\right\rbrace,  \\
S^+_{\omega}(t_0,v_0) & = \inf \left\lbrace \int_{t_0}^{\infty} L_{\omega}(q(t),q_t(t),t) dt, \: q=(u,v)\in H^1_
{\text{loc}}([t_0,\infty))^2, q(t_0)=(\pi,v_0), \lim_{t\rightarrow \infty }u(t)= 2\pi  \right\rbrace.  
\end{align*}
The functions for which the minima $S^-_{\omega}(t_0,v_0)$, $S^+_{\omega}(t_0,v_0)$ are attained are the solutions of (\ref{r:EL}) and lie on unstable, respectively stable manifolds of $\mathbb{T}_{\omega}$ (see Proposition \ref{p:stable}, (i) below). We call them one-sided (left-, respectively right-hand) sided minimizers at $(\omega,t_0,v_0)$.

We first obtain a-priori bounds on $S^-_{\omega}$, $S^+_{\omega}$. We then introduce the notion and construct specific super- and sub-solutions of (\ref{r:grad}) required in this section and later, and finally construct one-sided minimizers and prove explicit a-priori bounds. In particular, we prove that there exists an absolute constant $c_2 > 0$ so that, if $q^-=(u^-,v^-):(-\infty,t_0] \rightarrow \mathbb{R}^2$, $q^+=(u^+,v^+):[t_0,\infty) \rightarrow \mathbb{R}^2$ are any one-sided minimizers, then
\begin{align}
	|u^-(t)|  \leq c_2 e^{ -\frac{1}{2}\sqrt{\ve}  |t-t_0|}, \: \text{for all } t \leq t_0, \label{r:thirdq} \hspace{30pt} 
	|u^+(t) - 2\pi|  \leq c_2 e^{-\frac{1}{2}\sqrt{\ve}  |t-t_0|}, \: \text{for all } t \geq t_0. 
\end{align}
Lastly, we prove continuity of $S^-$, $S^+$ in $\omega,t_0,v_0$ and estimate the Lipschitz constant in $\omega$. 

\begin{lemma} \label{l:Sbounds}
	For all $\omega,t_0,v_0 \in \mathbb{R}$, we have
	\begin{equation}
		4\sqrt{\ve (1 - \mu)} \leq S^-_{\omega}(t_0,v_0),S^+_{\omega}(t_0,v_0) \leq 4\sqrt{\ve (1 + \mu)}. \label{s:bound}
	\end{equation}
\end{lemma}

\begin{proof}
	By definition,
	\begin{equation}
		\int_{t_0}^{\infty}L_{\omega}(q,q_t,t)dt \leq \int_{t_0}^{\infty} \left( \frac{1}{2}u_t^2 + \frac{1}{2}(v_t-c)^2 + (\ve(1+\mu))(1-\cos u(t)) \right) dt. \label{r:updef}
	\end{equation}
	It is well-known \cite{Arnold:64,Bessi:96} that the right-hand side of (\ref{r:updef}) attains minimum for the separatrix solution of the pendulum equation $u^0(t)= 4 \arctg e^{\sqrt{\ve(1+\mu)} \: (t-t_0)}$, $v^0(t)=\omega(t-t_0)+v_0$, and that the value of the integral on the right-hand side of (\ref{r:updef}) is then by direct calculation $4 \sqrt{\ve(1 + \mu)}$. Analogously we deduce that for any $q \in H^1_{\text{loc}}((t_0,\infty))^2$,
	$$
	4\sqrt{ \ve(1 - \mu)} \leq \int_{t_0}^{\infty} \left( \frac{1}{2}u_t^2 + \frac{1}{2}(v_t-c)^2 + (\ve(1-\mu))(1-\cos u(t)) \right) dt \leq \int_{t_0}^{\infty}L_{\omega}(q,q_t,t)dt,
	$$
	which completes the proof for $S^+_{\omega}(t_0,v_0)$. The bounds for $S^-_{\omega}(t_0,v_0)$ are analogous.
\end{proof}

In order to obtain a-priori bounds on one-sided minimizers, we require the notion of super-, respectively sub-solutions of (\ref{r:gradu}) or (\ref{r:gradv}). We say that $q=(u,v)$ is a super-solution of (\ref{r:gradu}) on $U = (t_0,t_1) \times (s_0,s_1]$, where $-\infty \leq t_0 <t_1 \leq \infty$, $-\infty \leq s_0 <s_1 \leq \infty$, if it is continuous on $\bar U$ and for any $(s,t) \in U$, $u_s-u_{tt}+\partial_uV(u,v,t)\leq 0$. Analogously we say that $q=(u,v)$ is a sub-solution of (\ref{r:gradu}) on $U$, if $u_s-u_{tt}+\partial_uV(u,v,t)\geq 0$. We say that $q$ is a strict super-, respectively sub-solution, if strict inequalities hold. 

We say that a $z : \mathbb{R} \rightarrow \mathbb{R}$ is a stationary super-solution on $I = (t_0,t_1)$, $-\infty \leq t_0 <t_1 \leq \infty$, if it is continuous on $[t_0,t_1]$ and $C^2$ on $(t_0,t_1)$, and such that for any $v \in C^2(\mathbb{R}^2)$, and for any $t \in (t_0,t_1)$,
$$ z_{tt} - \partial_uV(z,v,t) \leq 0.$$
We see that then for any $v \in C^2(\mathbb{R})$, $(z,v)$ is a super-solution of (\ref{r:gradu}) on $(t_0,t_1) \times \mathbb{R}$, where $z$ is considered as a fixed function in $s$. Analogously we define the notions of strict stationary super-solutions, sub-solutions, and analogous notions for (\ref{r:gradv}).

The following Lemma is a special case of the parabolic maximum principle \cite{Evans:10}.

\begin{lemma} \label{l:stationary}
	Assume $z$ is a strict stationary super-solution of (\ref{r:gradu}) on $(t_0,t_1)$, $-\infty \leq t_0 <t_1 \leq \infty$, assume $q=(u,v) \in \mathcal{E}$, and let $u(t) \leq z(t)$ for all $t \in (t_0,t_1)$. Then for all $t \in (t_0,t_1)$, $u(t) < z(t)$.
	
Analogous statements hold for strict stationary sub-solutions.
\end{lemma}

\begin{proof}
	Assume the contrary and find $t_2 \in (t_0,t_1)$ such that $u(t_2)=z(t_2)$. Direct calculation yields that 
	$$
	u_{tt}(t_2)-\partial_u V(u(t_2),v(t_2),t_2) \leq z_{tt}(t_2) - \partial_u V(z(t_2),v(t_2),t_2) < 0,
	$$
	which is in contradiction to $q \in \mathcal{E}$.
\end{proof}

\begin{lemma} \label{l:subsuper}
	There exist $z^- : (-\infty, 3/(4\sqrt{\ve})] \rightarrow \mathbb{R}$ and $z^+ : [-3/(4\sqrt{\ve}),\infty) \rightarrow \mathbb{R}$, depending only on $\ve$, $\mu$, satisfying for all $t$ in the domain of definition:

	(i) $0 < z^-(t) < 3 \pi /2 $, $ \pi/2 < z^+(t) < 2\pi$, both are continuous and $C^2$ in the interior of the domain,
	
	(ii) $z^-(0)=z^+(0)=\pi$,
	
	(iii) $z^-$ is a strict stationary super-solution on $(-\infty, 3/(4\sqrt{\ve}))$ of (\ref{r:gradu}), and $z^+$ is a strict stationary sub-solution (\ref{r:gradu}) on $(-3/(4\sqrt{\ve}),\infty)$. Furthermore, for any constant $T \geq 0$, $z^-(t+T)$ and $z^+(t-T)$ are strict stationary super-, resp. sub-solutions in the interior of their domain of definition.
	
	(iv) $z^-$, $z^+$ are strictly increasing and we have
	\begin{equation}
	 \sqrt{\ve}/2 < z^-_t(t),z^+_t(t) \quad \text{for all } t \in (-1/(4\sqrt{\ve}),1/(4\sqrt{\ve})).
	\end{equation}
	
	(v) For all $t \in [1/(4\sqrt{\ve}), 3/(4\sqrt{\ve})]$, $z^-(t) \leq \pi + 1/4$ and $z^+(-t) \geq \pi- 1/4$,
	
	(vi) There exists an absolute constant $c_2 > 0$ such that for all $t$ in the domains of definition,
	\begin{align}
	|z^-(t)|  \leq c_2 e^{ -\frac{1}{2}\sqrt{\ve} |t|},   \hspace{30pt} 
	|z^+(t) - 2\pi|  \leq c_2 e^{-\frac{1}{2}\sqrt{\ve}  |t|}.  \label{r:firstz}
	\end{align}
	
	(vii) There exists an absolute constant $c_3 > 0$ such that
	\begin{align}
	|z^-(t)-u^{(\ve)}(t)| \leq c_3 \sqrt{\ve \mu}, \: t \leq 0, \hspace{30pt} |z^+(t)-u^{(\ve)}(t)| \leq c_3 \sqrt{\ve \mu}, \: t \geq 0, \label{r:z-prec}
	\end{align}
    where $u^{(\ve)}(t) = 4 \arctg e^{\sqrt{\ve}t}$ is the separatrix solution in the case $\mu= 0$.
\end{lemma}

An explicit construction of $z^-,z^+$ and the proof of Lemma \ref{l:subsuper} is given in the Appendix B.

\begin{proposition} \label{p:stable}
	Let $(\omega,t_0,v_0)\in \mathbb{R}^3$. Then there exist one-sided minimizers $q^- : (\infty,t_0] \rightarrow \mathbb{R}$, $q^+: [t_0,\infty) \rightarrow \mathbb{R}$, for which $S^-_{\omega}(t_0,v_0)$, $S^+_{\omega}(t_0,v_0)$  
	attain their minimal value.
	
	Furthermore, any such one-sided minimizers $q^-=(u^-,v^-)$, $q^+=(u^+,v^+)$ at $(\omega,t_0,v_0)$ satisfy the following:
	
	(i) They are $C^4$ and solutions of the Euler-Lagrange equations on $(-\infty,t_0)$, $(t_0,\infty)$ respectively, 
	
	(ii) For all $0 \leq T \leq 3/(4\sqrt{\ve})$,
	\begin{align}
	0 & < u^-(t) \leq z^-(t-t_0+T)  & \text{for all } t \leq t_0, \label{r:firstq} \\
	z^+(t-t_0-T) & \leq u^+(t) < 2 \pi & \text{for all } t \geq t_0. \label{r:secondq}
	 \end{align}
\end{proposition}

The proof is in the Appendix B. (Existence and (i) are a consequence of the Tonelli theorem \cite[Appendix 1]{Mather:91}, and the a-priori bounds follow from Lemma \ref{l:stationary} applied to $z^-$, $z^+$ constructed in Lemma \ref{l:subsuper}.)

Combining (\ref{r:firstz}), (\ref{r:firstq}) and (\ref{r:secondq}) with $T=0$ we get:

\begin{corollary}
	Any one-sided minimizers $q^-$, $q^+$ at $(\omega,t_0,v_0)$ satisfy (\ref{r:thirdq}).
\end{corollary}

We finally deduce the Lipschitz constant for $S^-$ and $S^+$ in the variable $\omega$.
	
\begin{corollary} \label{l:continuity}
The functions $S^-$, $S^+$ are continuous in $t_0,v_0,\omega$. Furthermore, there exists an absolute constant $c_4 \geq 1$ such that for any $(t_0,v_0)\in \mathbb{R}^2$,
\begin{equation}
|S_{\omega}^-(t_0,v_0)-S_{\tilde{\omega}}^-(t_0,v_0)| \leq c_4 \mu |\tilde{\omega}-\omega|,   \hspace{5ex}
|S_{\omega}^+(t_0,v_0)-S_{\tilde{\omega}}^+(t_0,v_0)| \leq c_4 \mu |\tilde{\omega}-\omega|.  \label{r:sminus}
\end{equation}
\end{corollary}

\begin{proof}
We fix first $(t_0,v_0)$ and show that the Lipschitz constant of $S^+_{\omega}(t_0,v_0)$ in $\omega$. Choose $\omega,\tilde{\omega} \in \mathbb{R}$, and let $q=(u,v)$ be a right-hand sided minimizer constructed in Proposition \ref{p:stable} at $(\omega,t_0,v_0)$ respectively. Define
$$
\tilde{q}(t)=(\tilde{u}(t),\tilde{v}(t)):=(u(t),v(t) + (\tilde{\omega}-\omega)(t-t_0)),
$$
defined for $t \in [t_0,\infty).$ By definition of $q,\tilde{q}$, by applying $1-\cos u \leq (u-2\pi)^2/2$, (\ref{r:thirdq}) and the standing assumption (A2), we get
\begin{align*}
	S^+_{\tilde{\omega}}(t_0,v_0) & \leq 
	\int_{t_0}^{\infty} L_{\tilde{\omega}}(\tilde{q}(t),\tilde{q}_t(t),t)dt = \int_{t_0}^{\infty} \left( \frac{1}{2} u^2_t + \frac{1}{2}(v_t-\omega)^2 + V(\tilde{u},\tilde{v},t) \right) dt \\
	& = S^+_{\omega}(t_0,v_0) + \int_{t_0}^{\infty} (V(\tilde{u},\tilde{v},t)-V(u,v,t))dt \\
	& \leq S^+_{\omega}(t_0,v_0) + \ve \mu \int_{t_0}^{\infty} \left\lbrace 1-\cos(u(t))\right\rbrace \left\lbrace |\sup_{a \in [v(t),\tilde{v}(t)]} f_v(u(t),a,t)| |\tilde{\omega}-\omega|(t-t_0) \right\rbrace dt  \\
	& \leq  S^+_{\omega}(t_0,v_0) + c_4\ve \mu \int_{t_0}^{\infty} e^{-\sqrt{\ve}(t-t_0)}|\tilde{\omega}-\omega|(t-t_0)dt  \leq  S^+_{\omega}(t_0,v_0) + c_4 \mu |\tilde{\omega}-\omega|, 
\end{align*}
where $c_4$ (chosen to be $\geq 1$) is an absolute constant. The other inequalities in (\ref{r:sminus}) are proved analogously. Continuity in $t_0,v_0$ is follows similarly from the definitions of $S^-$, $S^+$.
\end{proof}

\section{The heteroclinic orbits and the region of instability} \label{s:hetero}

We discuss first the notion of a region of instability defined in the Introduction. We then recall the fact that, if $\omega,\tilde{\omega}$ are sufficiently close and in the same region of instability, then there exists a heteroclinic orbit connecting the tori $\mathbb{T}_{\omega}$ and $\mathbb{T}_{\tilde{\omega}}$. We also establish a-priori bounds on heteroclinic orbits, and show that the set of heteroclinic orbits is compact in $H^2_{\text{loc}}(\mathbb{R})^2$. As we were unable to find in the literature the a-priori bounds we need later, we give self-contained proofs.

We can write the function $S_{\omega}$ defined in the Introduction as $S_{\omega}(t,v)=S_{\omega}^-(t,v)+S_{\omega}^+(t,v)$.
\begin{defn} \label{d:instability}
	The region of instability is a connected component of the set of non-degenerate $\omega \in \mathbb{R}$, where $\omega$ is non-degenerate if every connected component of the set of global minima of $S_{\omega}$ in $\mathbb{R}^2$ is bounded. 
\end{defn}
One can easily show as a consequence of Corollary \ref{l:continuity} that a region of instability is open. We do not require it here, as we take (S1) as the standing assumption. 
Thus for every global minimum $(t_0,v_0)$ of $S_{\omega}$ we can find a closed, bounded set $\mathcal{N}(t_0,v_0) \subset \mathbb{R}^2$, $(t_0,v_0)\in \mathcal{N}(t_0,v_0)$, such that for each $(t_1,v_1) \in \partial \mathcal{N}(t_0,v_0)$, (\ref{r:DDelta0}) holds, and such that there is a constant $R \geq \sup \lbrace |x-y|, \: x,y \in \mathcal{N}(t_0,v_0) \rbrace$ satisfying (\ref{r:Rcond}). 

\begin{remark} \label{rem:instability} By continuity and periodicity of $S_{\omega}$, if $[{\omega}^-,{\omega}^+]$ is a segment in a region of instability, we can always find $\Delta_0>0$, $R>0$ satisfying (\ref{r:DDelta0}), uniform over $\omega \in [{\omega}^-,{\omega}^+]$, and uniform over $(t_0,v_0)$ which are global minimizers of $S_{\omega}$. The proof is analogous to the argument used in Lemma \ref{l:compact}, and omitted as not needed in the following.
\end{remark}	 

Let $[\omega^-,\omega^+]$ be an interval in the same region of instability satisfying (S1), and let $\varpi=\max \lbrace \omega^-,\omega^+, 1 \rbrace$. We fix $\Delta_0$, $R$ associated to $[\omega^-,\omega^+]$ from now on. 
We define the action $\mathcal{L}_{\omega,\tilde{\omega}}: H^1_{\text{loc}}(\mathbb{R})^2 \rightarrow \mathbb{R} \cup \lbrace \infty \rbrace$ and the minimal action $\Sigma_{\omega,\tilde{\omega}}: \mathbb{R}^2 \rightarrow \mathbb{R}$ along a trajectory of the heteroclinic orbit as:
\begin{align}
 \mathcal{L}_{\omega,\tilde{\omega}}(q) & =\int_{-\infty}^0 L_{\omega}(q,q_t,t)dt + \int_0^{\infty}L_{\tilde{\omega}}(q,q_t,t)dt + (\tilde{\omega}-\omega)v(0), \label{d:Lhetero} \\
\Sigma_{\omega,\tilde{\omega}}(t_0,v_0)& :=S^-_{\omega}(t_0,v_0)+S^+_{\tilde{\omega}}(t_0,v_0)+(\tilde{\omega}-\omega)v_0+\frac{1}{2}(\omega^2-\tilde{\omega}^2)t_0. \label{d:sigma}
\end{align}
The main result of the section is:
\begin{proposition} \label{p:hetero} Assume that $\omega,\tilde{\omega} \in [\omega^-,\omega^+]$ satisfy
	\begin{equation}
	|\omega-\tilde{\omega}| \leq \frac{\Delta_0}{4 c_4 (R \vee \mu) \cdot \varpi}. \label{r:assumec}
	\end{equation}
\noindent (i) There exist $q=(u,v) \in \mathcal{E}$ and $(t_0,v_0)\in [0,2\pi)^2$ such that $q(t_0)= (\pi,v_0)$, and such that $q|_{t\leq t_0}$ and $q|_{t\geq t_0}$ are one-sided minimizers at $(\omega,t_0,v_0)$, respectively $(\tilde{\omega},t_0,v_0)$.
\vspace{1ex}

\noindent (ii) Furthermore, there exists a closed, bounded set $\mathcal{N}_q \subset \mathbb{R}^2$ containing $(t_0,v_0)$, of radius at most $R$, such that for any $(t_1,v_1) \in \partial \mathcal{N}_q$, 
\begin{equation}
\Sigma_{\omega,\tilde{\omega}}(t_1,v_1)-\Sigma_{\omega,\tilde{\omega}}(t_0,v_0) \geq 2\Delta_0 >0. \label{r:sigma}
\end{equation}
\end{proposition}

We denote by  $\mathcal{H}$ the set of all $q \in \mathcal{E}$ satisfying (i), (ii) in Proposition \ref{p:hetero} for some $\omega,\tilde{\omega} \in [\omega^-,\omega^+]$ satisfying (\ref{r:assumec}). We say that such $q \in \mathcal{H}$ is a heteroclinic minimizer connecting $\omega,\tilde{\omega}$. Within this section, denote by $(t_{\omega},v_{\omega}) \in [0,2\pi)^2$ a minimizer of $S_{\omega}$, fixed if non-unique, and by $\mathcal{N}_{\omega}=\mathcal{N}(t_{\omega},v_{\omega})$.

\begin{lemma} \label{l:sigmamin}
	If (\ref{r:assumec}) holds, then $\Sigma_{\omega,\tilde{\omega}}$ attains a local minimum $(t_0,v_0)$ in the interior of $\mathcal{N}_{\omega}$, such that for any $(t_1,v_1) \in \partial \mathcal{N}_{\omega}$, (\ref{r:sigma}) holds.
\end{lemma}

\begin{proof} Note first that $\Sigma_{\omega,\tilde{\omega}}$ is continuous by the definition and Corollary \ref{l:continuity}. Thus by the definition and compactness of $\mathcal{N}_{\omega}$, it suffices to show that for some $(t_2,v_2)$ in the interior of $\mathcal{N}_{\omega}$, and any $(t_1,v_1) \in \partial \mathcal{N}_{\omega}$,  $\Sigma_{\omega,\tilde{\omega}}(t_1,v_1)-\Sigma_{\omega,\tilde{\omega}}(t_2,v_2) \geq 2\Delta_0 >0$. Let $(t_2,v_2) = (t_{\omega},v_{\omega})$. Then by definition, because of $|t_1-t_{\omega}| \leq R$, $|v_1-v_{\omega}| \leq R$, $c_4 \geq 1$ and (\ref{r:assumec}) we obtain
	\begin{align}
	\Sigma_{\omega,\tilde{\omega}}(t_1,v_1)-\Sigma_{\omega,\tilde{\omega}}(t_{\omega},v_{\omega}) & \geq S^-_{\omega}(t_1,v_1)-S^-_{\omega}(t_{\omega},v_{\omega}) + S^+_{\tilde{\omega}}(t_1,v_1)-S^+_{\tilde{\omega}}(t_{\omega},v_{\omega}) \notag \\
	& \quad - |\tilde{\omega}-\omega||v_1-v_{\omega}| - \varpi|\tilde{\omega}-\omega||t_1-t_{\omega}| \notag \\
	& \geq S^-_{\omega}(t_1,v_1)-S^-_{\omega}(t_{\omega},v_{\omega}) + S^+_{\tilde{\omega}}(t_1,v_1)-S^+_{\tilde{\omega}}(t_{\omega},v_{\omega})  - \Delta_0 / 2. \label{r:sigmahelp}
	\end{align}
	From (\ref{r:sminus}) and $\varpi \geq 1$ we deduce that
	$$
	S^+_{\tilde{\omega}}(t_1,v_1)-S^+_{\tilde{\omega}}(t_{\omega},v_{\omega}) \geq 	S^+_{\omega}(t_1,v_1)-S^+_{\omega}(t_{\omega},v_{\omega}) - \Delta_0 / 2,
	$$
	which combined with (\ref{r:sigmahelp}) and (\ref{r:DDelta0}) gives
	$$
	\Sigma_{\omega,\tilde{\omega}}(t_1,v_1)-\Sigma_{\omega,\tilde{\omega}}(t_{\omega},v_{\omega}) \geq S_{\omega}(t_1,v_1)-S_{\omega}(t_{\omega},v_{\omega}) - \Delta_0 \geq 2 \Delta_0.
	$$
\end{proof}

\begin{lemma} \label{l:sigmahetero}
	Assume $(t_0,v_0)$ is a local minimum of $\Sigma_{\omega,\tilde{\omega}}$, and let
	\begin{equation}
	q(t)= 
	\begin{cases} q^-(t) & t \leq t_0 , \\
	q^+(t) & t \geq t_0,
	\end{cases} \label{r:defhetero}
	\end{equation}
	where $q^-$ is the left-hand sided minimizer at $(\omega,t_0,v_0)$, and $q^+$ the right-sided minimizers at $(\tilde{\omega},t_0,v_0)$. We then have that $q$ is a solution of (\ref{r:EL}). 
\end{lemma} 

\begin{proof} Let $q$ and $(t_0,v_0)$ be as in the statement of the Lemma. We first show that, if $\tilde{q}=(\tilde{u},\tilde{v}) \in H^1(\mathbb{R})^2$ such that for some $(t_1,v_1) \in \mathbb{R}^2$, $\tilde{q}(t_1)=(\pi,v_1)$, then 
\begin{equation}	
\mathcal{L}_{\omega,\tilde{\omega}}(\tilde{q}) - \mathcal{L}_{\omega,\tilde{\omega}}(q) \geq \Sigma_{\omega,\tilde{\omega}}(t_1,v_1)-\Sigma_{\omega,\tilde{\omega}}(t_0,v_0).	 \label{r:Lmain}
\end{equation}
Indeed, by definitions and the partial integration, 	
\begin{align*}
	\mathcal{L}_{\omega,\tilde{\omega}}(\tilde{q}) & =
	 \int_{-\infty}^{t_1} L_{\omega}(\tilde{q}(t),\tilde{q}_t(t),t)dt  + \int_{t_1}^{\infty} L_{\tilde{\omega}}(\tilde{q}(t),\tilde{q}_t(t),t)dt + (\tilde{\omega}-\omega)v_1 + \frac{1}{2}(\omega^2-\tilde{\omega}^2)t_1  \\
	& \geq S_{\omega}^-(t_1,v_1)+S_{\tilde{\omega}}^+(t_1,v_1) + (\tilde{\omega}-\omega)v_1 + \frac{1}{2}(\omega^2-\tilde{\omega}^2)t_1 = \Sigma_{\omega,\tilde{\omega}}(t_1,v_1).
\end{align*}
By the definition of $q$, we obtain an equality in an analogous calculation for $q$, thus $\mathcal{L}_{\omega,\tilde{\omega}}(q) = \Sigma_{\omega,\omega*}(t_0,v_0)$. This gives (\ref{r:Lmain}).
	
We now claim that for any $h \in H^1(\mathbb{R})^2$,
	$
	\partial \mathcal{L}(q)h = \lim_{\delta \rightarrow 0}\frac{1}{\delta}\left( \mathcal{L}(q+ \delta h)-\mathcal{L}(q) \right)
	$
	is equal to $0$.
	It suffices to show that for any $h=(u^h,v^h) \in H^1(\mathbb{R})^2$ and sufficiently small $\delta >0$, $\mathcal{L}(q+ \delta h) \geq \mathcal{L}(q)$, and that $\partial\mathcal{L}(q)h$ exists. 
	Consider an open neighborhood $U$ of $(t_0,v_0)$ in $\mathbb{R}^2$ such that $\Sigma_{\omega,\tilde{\omega}}|_U \geq \Sigma_{\omega,\tilde{\omega}}(t_0,v_0)$. 
	
	We show first that for any $h \in H^1(\mathbb{R})^2$ we can find $\delta_0>0$ small enough, such that for any $0 \leq \delta \leq \delta_0$, there exists $(t_1,v_1) \in U$ such that $(q+\delta h)(t_1)=(\pi,v_1)$. Let $\tilde{q}=(\tilde{u},\tilde{v})=q + \delta h$. Indeed, we can find $t_1$ sufficiently close to $t_0$ such that $\tilde{u}(t_1)=\pi$ for $\delta$ small enough, because of (\ref{r:firstq}), (\ref{r:secondq}) with $T=0$ and the fact that $z^-$, $z^+$ are strictly increasing at $t=0$. We find $v_1=\tilde{v}(t_1)$ sufficiently close to $v_0$ by finding $\delta_0$ small enough so that all the terms in
	\begin{align*}
	|\tilde{v}(t_1)-v_0|  \leq \int_{t_0}^{t_1}|\tilde{v}_t|dt + |\tilde{v}(t_0)-v_0|   \leq \sqrt{2}|t_1-t_0|^{1/2} \left( \int_{t_0}^{t_1}(v^2_t(t)+ \delta^2 (v^h)^2_t) dt  \right) ^{1/2} + \delta |v^h(t_0)|
	\end{align*}
	are small enough. Now combining it with (\ref{r:Lmain}) and the fact that the right-hand side in (\ref{r:Lmain}) is $\geq 0$ on $U$, we obtain $\mathcal{L}(q+ \delta h) \geq \mathcal{L}(q)$ for $\delta \leq \delta_0$.
	Now, a straightforward calculation and the fact from Proposition \ref{p:stable} that $q$ solves (\ref{r:EL}) for all $t$ except perhaps $t=t_0$ yields
	$$ \partial{L}(q)h=(u^-_t(t_0)-u^+_t(t_0),v^-_t(t_0)-v^+_t(t_0))\cdot h(t_0),$$
	and $\partial{L}(q)h$ exists, thus $\partial{L}(q)h=0$. 
	As $h$ is arbitrary, we get that $q$ is $C^1$ at $t_0$. By the uniqueness of solutions of the Euler-Lagrange equations, $q$ is a solution of (\ref{r:EL}) also at $t=t_1$.
\end{proof}

Proposition \ref{p:hetero} now follows from Lemmas \ref{l:sigmamin} and \ref{l:sigmahetero}, with $\mathcal{N}_q:=\mathcal{N}_{\omega}$.

\begin{lemma} \label{l:t0v0}
	If $q \in \mathcal{H}$, then there exists a unique $(t_0,v_0) \in [0,2\pi]$ such that $q(t_0)=(\pi,v_0)$.
\end{lemma}

\begin{proof}
	Existence follows from the definition of $\mathcal{H}$, and uniqueness from (\ref{r:secondq}) and the properties of $z^-$, $z^+$ proved in Lemma \ref{l:subsuper}.
\end{proof}

\begin{lemma} \label{l:heteroprop}
	Let $q =(u,v)\in \mathcal{H}$ connecting $\omega,\tilde{\omega}$, such that $q(t_0)=(\pi,v_0)$. Then there exists an absolute constant $c_5>0$ such that for all $t \in \mathbb{R}$,
\begin{align}
 | u(t)- 2\pi \mathbf{1}_{[t_0,\infty)}(t)| & \leq c_5 e^{-\frac{1}{2}\sqrt{\ve} |t-t_0|}, \label{r:uhzero} \\
 |u_t(t)| & \leq c_5 \sqrt{\ve} e^{-\frac{1}{2}\sqrt{\ve} |t-t_0|}, \label{r:uhone}
\end{align}
 \begin{align}
 |v(t)-v_0-\omega(t-t_0) | & \leq c_5 \mu,  &  |v_t(t)-\omega | & \leq c_5 \sqrt{\ve} \mu \: e^{-\sqrt{\ve} |t-t_0|}, & t\leq t_0,  \label{r:vhleft} \\
 |v(t)-v_0-\tilde{\omega}(t-t_0) | & \leq c_5 \mu, 
 &  |v_t(t)-\tilde{\omega} | & \leq c_5 \sqrt{\ve}  \mu \:  e^{-\sqrt{\ve} |t-t_0|}, & t\geq t_0, \label{r:vhright}
\end{align}
\begin{align}
|u_{tt}(t)| & \leq c_5  \ve \: e^{-\frac{1}{2}\sqrt{\ve} |t-t_0|}, &  |v_{tt}(t)| &\leq c_5  \ve \mu \: e^{-\sqrt{\ve} |t-t_0|}, \label{r:two} \\
|u_{ttt}(t)| & \leq c_5  \ve \varpi \: e^{-\frac{1}{2}\sqrt{\ve} |t-t_0|}, &  |v_{ttt}(t)| & \leq c_5  \ve \mu \varpi  \: e^{-\frac{1}{2}\sqrt{\ve} |t-t_0|}. \label{r:three}
\end{align}
\end{lemma}

\begin{proof} The absolute constant $c_5$ may change from line to line in the proof.
The relation (\ref{r:uhzero}) follows from the definition of $q$ and (\ref{r:thirdq}). 
Now, by using $\sin x \leq |x - 2 k \pi|$ for $k=0,1$, the fact that $u$ is a solution of (\ref{r:EL}), and finally using (\ref{r:uhzero}), we see that 
\begin{equation*}
|u_{tt}(t) | \leq \ve (1-\cos u(t)+|\sin u(t)|) \ll \ve \: e^{-\frac{1}{2}\sqrt{\ve} |t-t_0|},
\end{equation*}
which is the left-hand side of (\ref{r:two}). By integrating it over $[t,\infty)$ for $t \geq t_0$, alternatively over $(-\infty,t]$ for $t \leq t_0$, and using $ \lim_{|t| \rightarrow \infty } u_t = 0$, we get (\ref{r:uhone}). 
Analogously, as $v(t)$ is a solution of (\ref{r:EL}), by using $\cos x \leq (x - 2 k \pi)^2/2$ for $k=0,1$ and (\ref{r:uhzero}), we obtain
\begin{equation*}
|v_{tt}(t) | \leq \ve \mu (1-\cos u(t)) \ll \ve \mu \: e^{-\sqrt{\ve} |t-t_0|}, \label{r:vktwo}
\end{equation*}
which is the right-hand side of of (\ref{r:two}). As  $ \lim_{t \rightarrow -\infty } v_t(t) = \omega$ and $ \lim_{t \rightarrow \infty } v_t(t) = \tilde{\omega}$, by integrating it over $(-\infty,t]$ for $t\leq t_0$, respectively over $[t,\infty)$ for $t \geq t_0$, we obtain the right-hand sides of (\ref{r:vhleft}) and (\ref{r:vhright}). We use $v(t_0)=v_0$, integrate the right-hand side of (\ref{r:vhleft}) over $[t,t_0]$, respectively the right-hand side of (\ref{r:vhright}) over $[t_0,t]$, and obtain the left-hand sides of (\ref{r:vhleft}) and (\ref{r:vhright}). Finally, to bound the third derivatives, by careful differentiation, while using uniform bounds on $f$ and its derivatives, and as $\mu \leq 1$ and $\varpi \geq 1$, we obtain
\begin{align*}
|u_{ttt}| & = |D_t V_u(u(t),v(t),t)| \ll \ve\mu |u_t| + \ve |u_t| + \ve \mu (1- \cos u + |\sin u|)|(v_t| + \ve \mu (1 - \cos u + |\sin u|) \\
& \ll \ve |u_t| + \ve |v_t - \omega \mathbf{1}_{(-\infty,t_0)}(t)-\tilde{\omega} \mathbf{1}_{[t_0,\infty)}(t)| + \ve |(1 - \cos u + |\sin u|)| \varpi.
\end{align*}
By inserting the bounds (\ref{r:uhzero}), (\ref{r:uhone}) and the right-hand sides of (\ref{r:vhleft}), (\ref{r:vhright}), we obtain the left-hand side of (\ref{r:three}).
Similarly we get
\begin{align*}
|v_{ttt}| & =  |D_t V_v(u(t),v(t),t)| 
\ll \ve \mu |u_t| + \ve \mu |v_t - \omega \mathbf{1}_{(-\infty,t_0)}(t)-\tilde{\omega} \mathbf{1}_{[t_0,\infty)}(t)| + \ve \mu \varpi ( 1 - \cos u),
\end{align*}
which analogously as above implies the right-hand side of (\ref{r:three}).
\end{proof}

\begin{lemma} \label{l:compact}
	The set $\mathcal{H}$ is compact in $H^2_{\text{loc}}(\mathbb{R})^2$. Furthermore, for each $q \in \mathcal{H}$ we have that $q \in \mathcal{E}$ and $q_t \in H^2_{\text{ul}}(\mathbb{R})^2$.
\end{lemma}

\begin{proof} It is straightforward to observe that the closure in $H^2_{\text{loc}}(\mathbb{R})^2$ of all $q$ satisfying (\ref{r:uhzero})-(\ref{r:three}) and $(t_0,v_0) \in [0,2\pi]^2$ is compact. Thus it suffices to show that $\mathcal{H}$ is closed in $H^2_{\text{loc}}(\mathbb{R})^2$. Assume $q_n \in \mathcal{H}$ connecting $\omega_n$ and $\tilde{\omega}_n$, $q_n(t_n)=(\pi,v_n)$, is a sequence converging to $q \in H^2_{\text{loc}}(\mathbb{R})^2$. We first show that $q \in \mathcal{E}$. By construction, $q$ is a solution of (\ref{r:EL}), and by Lemma \ref{l:heteroprop} and the construction we easily show that $q_t \in L^{\infty}(\mathbb{R})^2$. Now $q \in \mathcal{E}$ follows from Lemma \ref{l:Xcontain}. We see that $q$ must satisfy the condition (i) from the definition of $\mathcal{H}$, as the sequences $ \int_{-\infty}^{t_n}L_{\omega_n}(q_n,(q_n)_t,t)dt$, $\int_{t_n}^{\infty}L_{\tilde{\omega}_n}(q_n,(q_n)_t,t)dt$ are convergent by (\ref{r:uhzero})-(\ref{r:vhright}) and the Lebesgue dominated convergence theorem, and as $S^+$, $S^-$ are continuous. 
	
To show (ii), note that $\mathcal{N}_{q_n}$ is a family of compact sets with a bounded union, thus we can find a convergent subsequence converging to a set $\mathcal{N}_q$ in the Hausdorff topology. By the construction, the sequence $(t_n,v_n)$ converges to some $(t_0,v_0) \in \mathcal{N}_q$ such that $q(t_0)=(\pi,v_0)$. If $(\tilde{t}_0,\tilde{v}_0) \in \partial \mathcal{N}_q$, it is a limit of a subsequence of  $(\tilde{t}_{n_k},\tilde{v}_{n_k})$ lying on the boundaries of the convergent sub-sequence of $\partial \mathcal{N}_{q_n}$. The relation (\ref{r:sigma}) now follows by the continuity of $(\omega,\tilde{\omega},t,v) \mapsto \Sigma_{\omega,\tilde{\omega}}(t,v)$, established by the definition and Corollary \ref{l:continuity}.
\end{proof}

\section{An approximate shadowing orbit} \label{s:six}

In this section we define an approximate shadowing orbit $q^0$ which can be understood as a suitable initial condition for (\ref{r:grad}). Furthermore, we define the set $\mathcal{A}$ from Lemma \ref{l:combine}, and introduce the constants $L$, $L_k$, $k\in \mathbb{Z}$ (the time between the jumps) and $M$ (the magnitude of oscillations of $v$ with respect to (\ref{r:gradv})) to be optimized later, as a scaffolding for the proofs. Finally we show that $q^0 \in \mathcal{A}$, and evaluate
 bounds on $q^0$ needed later. We will eventually see that essentially the only role of $q^0$ in the proofs is to show that the constructed sets $\mathcal{A}$, $\mathcal{B}$ are not empty.

Fix a closed subset of a region of instability $[\omega^-,\omega^+]$, with the uniform constants $\Delta_0$, $R$ as in (S1) and $\varpi$ as in Introduction. Assume $\omega_k$, $k\in \mathbb{Z}$ is a sequence in $[\omega^-,\omega^+]$ such that for all $k \in \mathbb{Z}$, (\ref{r:assumec}) holds. The constant $4L$ will be the minimal time between two "jumps". Let $\tilde{L}_k$ be the approximate time of the "jumps", satisfying $\tilde{L}_k \equiv 0 \mod 2\pi$ and $\tilde{L}_{k+1}-\tilde{L}_k \geq 4L + 2\pi$. Let $q_k=(u_k,v_k) \in \mathcal{H}$ and $(T_k,V_k)$ be such that $q_k(T_k)=(\pi,V_k)$, and let $\mathcal{N}_{q_k} \subset \mathbb{R}^2$ be the sets associated to $q_k$ as in the definition of $\mathcal{H}$. Note that if $(t_1,v_1) \in \partial \mathcal{N}_{q_k}$, and if $\tilde{q}=(\tilde{u},\tilde{v}) \in H^1_{\text{loc}} (\mathbb{R})^2$ such that $\tilde{q}(t_1)=(\pi,v_1)$, $\lim_{t \rightarrow -\infty}\tilde{u}(t)=0$, $\lim_{t \rightarrow \infty}\tilde{u}(t)=2\pi$, then by (\ref{r:sigma}) and (\ref{r:Lmain}), we have
\begin{equation}
\mathcal{L}_{\omega,\omega^*}(\tilde{q}) -\mathcal{L}_{\omega,\omega^*}(q)  \geq 2\Delta_0 >0. \label{r:LL}
\end{equation}
Also, by the definition of $\Delta_0$ and (\ref{s:bound}), we can easily deduce (using $\mu \leq 1/16$ by (A2)) the useful bound
\begin{equation}
\Delta_0 \leq 9 \sqrt{\ve}\mu. \label{r:Delta0up}
\end{equation}
We construct the required parameters, functions and sets inductively in $|k|$ as follows: $\tilde{T}_0=T_0$, $\tilde{V}_0=V_0$, $\tilde{q}_0=(\tilde{u}_0,\tilde{v}_0):=q_0$, and
\begin{align*}
\tilde{T}_k & = T_k \mod 2\pi, \quad \text{so that } -\pi  < \tilde{T}_k-\tilde{L}_{k} \leq \pi, \\
\tilde{q}_k(t) & = (u_k(t-\tilde{T}_k+T_k)+2k\pi, \: v_k(t-\tilde{T}_k+T_k)+ \tilde{V}_k-V_k ), \\
\tilde{V}_k  & = V_k \mod 2\pi \quad \text{so that } -\pi < \tilde{v}_{k-1}(\tilde{T}_k)-\tilde{V}_k  \leq \pi \text{ for }k \geq 1, \\
\tilde{V}_k  & = V_k \mod 2\pi \quad \text{so that } -\pi <  \tilde{v}_{k+1}(\tilde{T}_k)-\tilde{V}_k \leq \pi \text{ for }k \leq -1, \\
\tilde{\mathcal{N}}_k & = \mathcal{N}_{q_k}+(\tilde{T}_k-T_k,\tilde{V}_k-V_k), \\
L_k &= \tilde{T}_{k+1} - \tilde{T}_k,
\end{align*}
where we always use the notation $\tilde{q}_k=(\tilde{u}_k,\tilde{v}_k)$. We now require "smoothening" functions $\varphi^-$, $\varphi^+$, defined over an arbitrary interval $[a,b]$, $a < b$:
\begin{equation}
\label{d:smooth}
\begin{split} 
\varphi^-_{a,b}(t) & = \begin{cases} 1 & t \leq a, \\
\frac{\exp(-(b-a)/(t-a))}{\exp(-(b-a)/(t-a))+\exp(-(b-a)/(b-t))} & t \in [a,b], \\
0 & t \geq b, 
\end{cases} \\
\varphi^+_{a,b}(t) & =1-\varphi^-_{a,b}(t).
\end{split}
\end{equation}
By definition $\varphi^-,\varphi^+$ are $C^{\infty}$, with values in $[0,1]$, and with uniformly bounded derivatives
\begin{equation}
 (\varphi^-_{a,b})^{(k)}(t),(\varphi^+_{a,b})^{(k)}(t) =O_k \left( \frac{1}{|b-a|^k} \right), \label{r:varphib}
\end{equation}
where the implicit constant depends only on $k$.
Let
\begin{equation}
q^0(t) = \varphi^-_{\tilde{T}_{k-1}+L,\tilde{T}_k-L}(t)\tilde{q}_{k-1}(t) + \varphi^+_{\tilde{T}_{k-1}+L,\tilde{T}_k-L}(t)\tilde{q}_k(t) \: \quad \text{for all } t \in [\tilde{T}_{k-1},\tilde{T}_k]. \label{r:defq*}
\end{equation}
\begin{remark} \label{r:notation}
Assume we fix a segment $[\omega^-,\omega^+]$ in a region of instability, and that for each $\omega,\tilde{\omega} \in [\omega^-,\omega^+]$ satisfying (\ref{r:assumec}) we chose a single $q \in \mathcal{H}$ (as such $q$ is not necessarily unique). Then $q^0$ is uniquely defined by the choice of $L$, $(\tilde{L}_k)_{k \in \mathbb{Z}}$, $(\omega_k)_{k \in \mathbb{Z}}$ (uniqueness of $T_k,V_k,\tilde{T}_k,\tilde{V}_k$ follows from Lemma \ref{l:t0v0}). In the proofs of the main theorems, we thus use the notation $q^0(L,(\tilde{L}_k)_{k \in \mathbb{Z}},(\omega_k)_{k \in \mathbb{Z}})$. We fix $q^0$ for now and do not use such notation until Section \ref{s:shadow}.
\end{remark}

Finally, let $M$ be a constant chosen later, so that
\begin{align}
 M \geq \sup_{k \in \mathbb{Z}} \lbrace |\tilde{v}_{k-1}(\tilde{T}_k)|-\tilde{V}_k |, |\tilde{v}_k(\tilde{T}_{k-1})|-\tilde{V}_{k-1} | \rbrace + ( \varpi+1)\mu . \label{r:mbound}
\end{align}
The set $\mathcal{A}$ is defined as the set of all $q = (u,v)\in H^3_{\text{loc}}(\mathbb{R})^2 \cap \mathcal{X}$ such that $q_t \in H^2_{\text{ul}}(\mathbb{R})^2$, and such that for all $k \in \mathbb{Z}$,
\begin{align}
|u(\tilde{T}_k) - (2k  + 1) \pi | \leq \frac{1}{3}, \label{d:Au} \\
|v(\tilde{T}_k) - \tilde{V}_k | \leq M. \label{d:Av}
\end{align}
\begin{lemma}
	We have that $q^0 \in \mathcal{A}$.
\end{lemma}
\begin{proof}
	The smoothness of $q^0$ and $q_t^0 \in H^2_{\text{ul}}(\mathbb{R})^2$ follow from the construction, Lemma \ref{l:heteroprop} and Remark \ref{r:UL}. By definition, $u(\tilde{T}_k)=(2k+1)\pi$ and $v(\tilde{T}_k)=\tilde{V}_k$, which trivially implies (\ref{d:Au}), (\ref{d:Av}). 
\end{proof}

Let $k(t) = j$ for $t \in (T_{j-1},T_j]$ and $\dt =\min \lbrace t-T_{k(t)-1},T_{k(t)}-t \rbrace = \min \lbrace |t-\tilde{T}_k|, \: k \in \mathbb{Z} \rbrace$. Furthermore, let
\[
\tilde{\omega}_k = \frac{\tilde{V}_k-\tilde{V}_{k-1}}{\tilde{T}_k-\tilde{T}_{k-1}}.
\]

\begin{lemma} \label{l:q*bounds}
	There exist an absolute constant $c_6 \geq 1$ so that any $q^0=(u^0,v^0)$ given by (\ref{r:defq*}) satisfies:
	\begin{align}
		|u^0(t)- 2k(t)\pi| & \leq c_6 e^{-\frac{1}{2}\sqrt{\ve}\dt}, \label{r:Uzero} \\
	|v^0(t)-V_{k(t)-1}-\tilde{\omega}_{k(t)}(t)(t-T_{k(t)-1}) |  & \leq c_6 (1 \wedge M), \label{r:Vzero} 
	\end{align}
	\begin{align}
	|u^0_t(t)| & \leq c_6   \left( \sqrt{\ve}\: e^{-\frac{1}{2}\sqrt{\ve}\dt} + \frac{1}{L_{k(t)}}\right), & |v^0_t(t) - \omega_{k(t)}| & \leq c_6 \left( \sqrt{\ve} \mu \: e^{-\frac{1}{2}\sqrt{\ve}\dt} + \frac{1}{L_{k(t)}}\right), \label{r:One} \\
	|u^0_{tt}(t)| & \leq c_6  \left( \ve\:  e^{-\frac{1}{2}\sqrt{\ve}\dt} + \frac{1}{L^2_{k(t)}}\right), & |v^0_{tt}(t)| & \leq c_6 \left(  \ve \mu \:  e^{-\frac{1}{2}\sqrt{\ve}\dt} + \frac{1}{L^2_{k(t)}}\right), \label{r:Two}
	\end{align}
	and finally
	\begin{subequations}
		\begin{align}
	|u^0_{ttt}(t)| & \leq c_6 \varpi  \left(\ve \: e^{-\frac{1}{2}\sqrt{\ve}\dt} + \frac{1}{L^3_{k(t)}}\right),  \\ |v^0_{ttt}(t)| & \leq c_6 \varpi  \left( \ve \mu \: e^{-\frac{1}{2} \sqrt{\ve}\dt} + \frac{1}{L^3_{k(t)}}\right). 
	\end{align}	\label{r:Three}
	\end{subequations}
\end{lemma}

\begin{proof} 
We write (\ref{r:defq*}) in an abbreviated form $q^0=\varphi^- \tilde{q}_{k-1} + \varphi^+ \tilde{q}_k$ for $t \in [\tilde{T}_{k-1},\tilde{T}_k]$. Then
\begin{align}
u^0 & = \varphi^- (\tilde{u}_{k-1}-2k\pi) + \varphi^+ (\tilde{u}_k-2k\pi) + 2k\pi, \label{r:u*final} \\
v^0 & = \varphi^- (\tilde{v}_{k-1}-\tilde{V}_{k-1}-\tilde{\omega}_k(t-\tilde{T}_{k-1})) + \varphi^+ (\tilde{u}_k-\tilde{V}_{k-1}-\tilde{\omega}_k(t-\tilde{T}_{k-1}) ) + \tilde{V}_{k-1}+\tilde{\omega}_k(t-\tilde{T}_{k-1}). \label{r:v*final}
\end{align}
To obtain (\ref{r:Uzero}), the left-hand side of (\ref{r:One}), and (\ref{r:Two}), (\ref{r:Three}), it suffices 
to differentiate (\ref{r:u*final}), (\ref{r:v*final}) and insert (\ref{r:uhzero})-(\ref{r:three}) as required and (\ref{r:varphib}).

From the left-hand sides of (\ref{r:vhleft}), (\ref{r:vhright}) and the definition of $\tilde{v}_k$, $\tilde{T}_k$, $\tilde{V}_k$, we easily obtain that for all $t \in [T_{k-1},T_k]$,
\begin{subequations}
\begin{align}
|\tilde{v}_{k-1}(t)-\tilde{V}_{k-1}-\omega_k(t-\tilde{T}_{k-1}) | & \ll \mu,  \label{r:vkzeroa} \\
|\tilde{v}_k(t)-\tilde{V}_k-\omega_k(t-\tilde{T}_k) | & \ll \mu. \label{r:vkzerob}
\end{align}
\label{r:vkzero}
\end{subequations}
By inserting $t=\tilde{T}_k$ in (\ref{r:vkzeroa}), we get 
\begin{equation*}
\left| \omega_k - \frac{\tilde{v}_{k-1}(\tilde{T}_k)-\tilde{V}_{k-1}}{\tilde{T}_k-\tilde{T}_{k-1}} \right| \ll \frac{\mu}{\tilde{T}_k-\tilde{T}_{k-1}}.
\end{equation*}
By definition and (\ref{r:mbound}) we know that $|\tilde{V}_k - \tilde{v}_{k-1}(\tilde{T}_k)| \ll 1 \wedge M$. As also $\mu \leq 1 \wedge M$, we have
\begin{equation}
 | \omega_k - \tilde{\omega}_k | \leq \left| \omega_k - \frac{\tilde{v}_{k-1}(\tilde{T}_k)-\tilde{V}_{k-1}}{\tilde{T}_k-\tilde{T}_{k-1}} \right| + 
 \left| \frac{\tilde{V}_k - \tilde{v}_{k-1}(\tilde{T}_k)}{\tilde{T}_k-\tilde{T}_{k-1}} \right| \ll \frac{1 \wedge M}{\tilde{T}_k-\tilde{T}_{k-1}}. \label{r:omega}
\end{equation}
Combining it with (\ref{r:vkzero}), and using $\tilde{V}_{k-1}+\tilde{\omega}_k(t-\tilde{T}_{k-1}) = \tilde{V}_k+\tilde{\omega}_k(t-\tilde{T}_k)$ and $\mu \leq 1 \wedge M$, we obtain
\begin{subequations}
	\begin{align}
	|\tilde{v}_{k-1}(t)-\tilde{V}_{k-1}-\tilde{\omega}_k(t-\tilde{T}_{k-1}) | & \ll 1 \wedge M,  \label{r:vkzerooa} \\
	|\tilde{v}_k(t)-\tilde{V}_{k-1}-\tilde{\omega}_k(t-\tilde{T}_{k-1}) | & \ll 1 \wedge M. \label{r:vkzeroob}
	\end{align}	\label{r:vkzeroo}
\end{subequations}
Now (\ref{r:Vzero}) follows from (\ref{r:v*final}) and (\ref{r:vkzeroo}). The right-hand side of (\ref{r:One}) is obtained easily by differentiating (\ref{r:v*final}), using the right-hand sides of (\ref{r:vhleft}), (\ref{r:vhright}) and finally (\ref{r:omega}).
\end{proof}

\section{Invariant sets with $L^{\infty}$ bounds} \label{s:seven}

We now construct $\mathcal{A}$-relatively $\xi$-invariant sets with respect to the dynamics (\ref{r:grad}) satisfying a-priori $L^{\infty}$ bounds. More specifically, we construct $\mathcal{B}_1$ such that any $q=(u,v) \in \mathcal{B}_1$ satisfies for all $t \in \mathbb{R}$
\begin{align}
  |u(t)-2k(t)\pi | & \leq c_7 e^{-\frac{1}{2}\sqrt{\ve}\: \dt  }, \label{r:main1} \\
  |v(t)-v^0(t)| & \leq c_8 M, \label{r:main2}
\end{align}	
where $c_7,c_8>0$ are absolute constants. We first define the set $\mathcal{B}_1$, then show that it is $\mathcal{A}$-relatively $\xi$-invariant, that $q^0 \in \mathcal{B}_1$, and finally we deduce (\ref{r:main1}) and (\ref{r:main2}).

Let $v^-_k,v^+_k : [\tilde{T}_k,\tilde{T}_{k+1}] \rightarrow{\mathbb{R}}$ be the unique $C^2$ functions satisfying
\begin{align*}
v^-_k(\tilde{T}_k)&= \tilde{V}_k+c_6 M,  &
v^+_k(\tilde{T}_k)&= \tilde{V}_k-c_6 M,  \\
v^-_k(\tilde{T}_{k+1})&= \tilde{V}_{k+1}+c_6 M, &
v^+_k(\tilde{T}_{k+1})&= \tilde{V}_{k+1}-c_6 M,
\end{align*}
and
\begin{equation*}
-(v^-_k)_{tt}(t)=(v^+_k)_{tt}(t)=c_2^2 e^2 \: \ve \mu \: e^{-\sqrt{\ve} \: \dt}.
\end{equation*}
We define $\mathcal{B}_1$ to be the set of all $q \in \mathcal{A}$ satisfying
\begin{align}
     z^+(t-\tilde{T}_k-3/(4 \sqrt{\ve})) + 2k\pi &\leq  u(t) \leq z^-(t-\tilde{T}_{k+1}+3/(4 \sqrt{\ve})) + 2(k+1)\pi, & t\in [\tilde{T}_k,\tilde{T}_{k+1}], \label{r:defb1u} \\
     v^+_k(t)& \leq  v(t) \leq v^-_k(t), & t\in [\tilde{T}_k,\tilde{T}_{k+1}]. \label{r:defb1v}
\end{align}

\begin{lemma} \label{l:B1}
	The set $\mathcal{B}_1$ is $\mathcal{A}$-relatively $\xi$-invariant.
\end{lemma}

\begin{proof} 
We apply twice the parabolic maximum principle \cite[Sec. 7, Theorem 12]{Evans:10}. Assume that $q(s_0) \in \mathcal{B}_1$, and that for all $s \in [s_0,s_1]$, $q(s)\in \mathcal{A}$. We have already shown in Lemma \ref{l:subsuper}, (iii), that  $z^+(.-\tilde{T}_k-3/(4 \sqrt{\ve}))$ is a strict stationary sub-solution, and $z^-(.-\tilde{T}_{k+1}+3/(4 \sqrt{\ve}))$ a strict stationary super-solution of (\ref{r:gradu}) on $(\tilde{T}_k,\tilde{T}_{k+1})$. The assumptions and Lemma \ref{l:stationary},(v) imply that 
$$
z^+(t-\tilde{T}_k-3/(4 \sqrt{\ve})) + 2(k-1)\pi \leq u(t,s) \leq z^-(t-\tilde{T}_{k+1}+3/(4 \sqrt{\ve})) + 2k\pi
$$
holds on the parabolic boundary 
\begin{equation}
(t,s) \in \lbrace [\tilde{T}_k,s], s\in [s_0,s_1] \rbrace \cup \lbrace [t,s_0], t\in [\tilde{T}_k,\tilde{T}_{k+1}] \rbrace \cup \lbrace [\tilde{T}_{k+1},s], s\in [s_0,s_1] \rbrace, \label{r:parboundary}
\end{equation}
thus by the parabolic maximum principle, (\ref{r:defb1u}) holds for all $s \in [s_0,s_1]$ and all $k \in \mathbb{Z}$.

Consider now the bounds on $z^-$, $z^+$. By Lemma \ref{l:stationary},(vi), we that for $t \in [\tilde{T}_k,\tilde{T}_{k+1}]$,
\begin{align*}
 |z^+(t-\tilde{T}_k-3/(4\sqrt{\ve}))-2\pi | & \leq c_2 e^{-\frac{1}{2}\sqrt{\ve} | t-\tilde{T}_k| + \frac{3}{8}} \leq c_2 e \:  e^{-\frac{1}{2}\sqrt{\ve} \dt }. 
\end{align*}
We obtain analogous bounds on $ |z^-(t-\tilde{T}_{k+1}+3/(4 \sqrt{\ve}))|$. We deduce that whenever $u(t)$ satisfies (\ref{r:defb1u}), then (\ref{r:main1}) holds for some absolute $c_7 = c_2 e$. 

We now show that whenever $u(t,s)$ satisfies (\ref{r:main1}), $v^-$ is a super-solution, and $v^+$ a sub-solution of (\ref{r:gradv}) on $(\tilde{T}_k,\tilde{T}_{k+1})$. Consider $v^-$. By using the definition of $(v_k^-)_{tt}$, the relation $(1-\cos x) \leq (x-2k(t)\pi)^2/2$, (\ref{r:main1}) with $c_7=c_2e$ and the standing assumption (A1), we get that for $t \in (\tilde{T}_k,\tilde{T}_{k+1})$,
\begin{align*}
(v^-_k)_{tt}-V_v(u,v^-,t) & = - c_2^2 e^2 \: \ve \mu e^{-\sqrt{\ve} \: \dt} + \ve \mu (1-\cos u(t))|f_v(u,v^-_k,t)| \\
& \leq - c_2^2 e^2 \: \ve \mu e^{-\sqrt{\ve} \: \dt} + c_2^2 e^2 \: \ve \mu  e^{-\sqrt{\ve} \: \dt} \leq 0.
\end{align*}
Analogously we get that $(v^+_k)_{tt}-V_v(u,v^+_k,t) \geq 0$. Furthermore, by the definition of $\mathcal{A}$, we have that for all $k \in \mathbb{Z}$, $v^+_k(\tilde{T}_k) \leq v(\tilde{T}_k) \leq v^-_k(\tilde{T}_k)$ and $v^+_k(\tilde{T}_{k+1}) \leq v(\tilde{T}_{k+1}) \leq v^-_k(\tilde{T}_{k+1})$.
It suffices now to apply the parabolic maximum principle to (\ref{r:gradv}) for all $k \in \mathbb{Z}$, with the same parabolic boundary (\ref{r:parboundary}).  
\end{proof}

\begin{lemma}
	We have that $q^0 \in \mathcal{B}_1$.
\end{lemma}

\begin{proof}
	Use the notation $\tilde{q}_k=(\tilde{u}_k,\tilde{v}_k)$ and $\tilde{q}_{k+1}=(\tilde{u}_{k+1},\tilde{v}_{k+1})$. From Proposition \ref{p:stable}, we see that for $t \in [\tilde{T}_k,\tilde{T}_{k+1}]$, 
	\begin{align*}
	 z^+(t-\tilde{T}_k-3/(4 \sqrt{\ve})) + 2k \pi &\leq  \tilde{u}_k(t) \leq 2(k+1)\pi \leq z^-(t-\tilde{T}_{k+1}+3/(4 \sqrt{\ve})) + 2(k+1)\pi, \\
	 z^+(t-\tilde{T}_k-3/(4 \sqrt{\ve})) + 2k \pi &\leq 2(k+1)\pi  \leq \tilde{u}_{k+1}(t) \leq z^-(t-\tilde{T}_{k+1}+3/(4 \sqrt{\ve})) + 2(k+1)\pi.
	\end{align*}
	As $u^0$ is a convex combination of $\tilde{u}_k(t)$, $\tilde{u}_{k+1}(t)$ on $[\tilde{T}_k,\tilde{T}_{k+1}]$, (\ref{r:defb1u}) holds for $u=u^0$.
		
The relation (\ref{r:defb1v}) for $v=v^0$ follows from (\ref{r:Vzero}) and the definitions of $v_k^-$, $v_k^+$ (as $v_k^-$ is concave and $v_k^+$ is convex).	
\end{proof}

\begin{lemma}
	The relations (\ref{r:main1}), (\ref{r:main2}) hold for all $q \in \mathcal{B}_1$.
\end{lemma}

\begin{proof}
	The relation (\ref{r:main1}) has already been established in the proof of Lemma \ref{l:B1}. As $v(t),v^0(t)$ satisfy (\ref{r:defb1v}), we have that 
	\begin{equation}
	|v(t)-v^0(t)| \leq \sup \lbrace |v^+_k(t)-v^-_k(t)|, k \in \mathbb{Z}, t \in [\tilde{T}_k,\tilde{T}_{k+1}] \rbrace . \label{r:v+v-}
	\end{equation}
	To establish a bound on $|v^+_k(t)-v^-_k(t)|$, we introduce $w(t)=v^-_k(t)-\tilde{V}_k-(\tilde{V}_{k+1}-\tilde{V}_k)(t-\tilde{T}_k)-M$. As $w(\tilde{T}_k)=w(\tilde{T}_{k+1})=0$, by symmetry $w_t(\tilde{T})=0$, where $\tilde{T}=(\tilde{T}_k+\tilde{T}_{k+1})/2$. Consider $T \in [\tilde{T},\tilde{T}_{k+1}]$.
	\begin{align*}
	|w_t(T)| & = c_2^2 e^2 \: \ve \mu \int_{\tilde{T}}^T \exp( -\sqrt{\ve} \: (\tilde{T}_k-t))dt \leq 8 e^2 \sqrt{\ve} \: \mu \exp(-\sqrt{\ve}(\tilde{T}_k-T)), \\
	|w(T)-w(\tilde{T})| & \leq c_2^2 e^2 \: \sqrt{\ve} \: \mu \int_{\tilde{T}}^T \exp(-\sqrt{\ve}(\tilde{T}_k-t)) dt \leq c_2^2 e^2 \mu.
	\end{align*}
	As $w(T)$ is decreasing on $[\tilde{T},\tilde{T}_{k+1}]$, we get that $|w(t)| \leq c_2^2 e^2 \mu$. By analogy, the same holds on $[\tilde{T}_k,\tilde{T}]$. We see that $|v^-_k(t)-\tilde{V}_k-(\tilde{V}_{k+1}-\tilde{V}_k)(t-\tilde{T}_k)| \leq M + c_2^2 e^2 \mu$. Analogously,
	$|v^+_k(t)-\tilde{V}_k-(\tilde{V}_{k+1}-\tilde{V}_k)(t-\tilde{T}_k)| \leq M + c_2^2 e^2 \mu$, thus
	$$ 
	| v^+_k(t)-v^-_k(t)| \leq 2M + 2c_2^2 e^2 \: \mu, \text{ for all } t \in [\tilde{T}_k,\tilde{T}_{k+1}],
	$$
	which is by (\ref{r:mbound}) $\leq c_8 M$, with $c_8 = 2c_2^2 e^2 +2$. It suffices to insert this in (\ref{r:v+v-}).
\end{proof}

\section{Bounds on the derivatives} \label{s:eight}

Here we show that there is a $\mathcal{A}$-relatively $\xi$-invariant set such that the norms of the first, second and third order derivatives of $u$, $v$ all behave as $O_{\ve}(\log(\dt )/\dt )$. Let
\begin{equation}
\lambda(\tau) = \frac{\sqrt{\ve}}{4} \wedge  \frac{8\log \dtau}{\dtau}, \label{r:lambda}
\end{equation} 
where $\wedge$ is the minimum. In this section we use the weighted $ L^2$ norm
\begin{equation*}
||w||^2_{\Ltau}  := \int_{\mathbb{R}} e^{-\lambda(\tau)|t-\tau|} w^2(t)dt.
\end{equation*}
Let $\mathcal{B}_2$ be the set of all $q \in \mathcal{B}_1$ such that
\begin{subequations} \label{r:defB2}
\begin{align}
||u_t-u^0_t||^2_{\Ltau^2} & \leq c_9 \: \lambda(\tau), \label{r:defB2u} \\
||v_t-v^0_t||^2 _{\Ltau^2} & \leq c_9(M^2+1) \: \lambda(\tau), \label{r:defB2v}
\end{align} 
\end{subequations}
for all $\tau \in \mathbb{R}$. Let $\mathcal{B}_3$ be the set of all $q \in \mathcal{B}_2$ such that for all $\tau \in \mathbb{R}$,
\begin{subequations} \label{r:defB3}
\begin{align}
\ve||u_t-u^0_t||^2_{\Ltau^2} + ||u_{tt}||^2_{\Ltau^2}  & \leq c_{10}  (M^2 + \varpi^2) \ve \: \lambda(\tau), \label{r:defB3u}   \\
\ve||v_t-v^0_t||^2_{\Ltau^2} + ||v_{tt}||^2_{\Ltau^2}  & \leq c_{10}  (M^2 + \varpi^2) \ve \: \lambda(\tau). \label{r:defB3v}
\end{align}
\end{subequations}
Finally, let 
$\mathcal{B}_4$ be the set of all $q \in \mathcal{B}_3$ such that for all $\tau \in \mathbb{R}$,
\begin{subequations} \label{r:defB4}
\begin{align}
\ve||u_t-u^0_t||^2_{\Ltau^2} + ||u_{tt}||^2_{\Ltau^2} + ||u_{ttt}||^2_{\Ltau^2}   & \leq c_{11} (M^4 + \varpi^4)\ve \: \lambda(\tau), \label{r:defB4u}   \\
\ve||v_t-v^0_t||^2_{\Ltau^2} + ||v_{tt}||^2_{\Ltau^2} + ||v_{ttt}||^2_{\Ltau^2} & \leq c_{11} (M^4 + \varpi^4)\ve \: \lambda(\tau). \label{r:defB4v}
\end{align}
\end{subequations}
\begin{proposition} \label{p:Bthree} There exist absolute constants $c_9$, $c_{10}$ and $c_{11}$ so that the sets ${\mathcal B}_2$, ${\mathcal B}_3$ and ${\mathcal B}_4$ are $\mathcal{A}$-relatively $\xi$-invariant, and such that $q^0 \in  {\mathcal B}_4$. 
\end{proposition}

The proof of Proposition is routine but technical, and as such postponed to the Appendix C. In essence, by differentiating the weighted integral versions of (\ref{r:gradu}), (\ref{r:gradv}), we obtain a differential inequality which by the Gronwall's lemma implies invariance of the sets as required. An important step is use of a variant of the Poincar\'{e} inequality (Lemma \ref{l:poincare1}) which relies on the $L^{\infty}$ bounds obtained in the previous section. We do the procedure iteratively for the three sets.

The main implication needed in the following is that we can for each $k$ approximate $q-\tilde{q}_k$ close to $\tilde{T}_k$ with a "well-behaved" $h$ vanishing at $\pm \infty$.

\begin{lemma} \label{p:B4bounds}
Assume that $q \in \mathcal{B}_4$. Then there exists an absolute constant $c_{12}>0$ such that for each $k \in \mathbb{Z}$ there exist $\tilde{h} = (\tilde{u}^h,\tilde{v}^h) \in H^3_{\text{loc}}(\mathbb{R})^2$ satisfying the following:
	
\begin{itemize}
	\item[(i)] For all $t \in [\tilde{T}_k-L, \tilde{T}_k+L]$, $\tilde{h}(t)=q(t)-\tilde{q}_k(t)$,
	\item[(ii)] For $t \geq \tilde{T}_k + L(1+1/\log L)$ and for $t\leq \tilde{T}_k - L(1+1/\log L)$ we have $\tilde{h}(t)=0$,
	\item[(iii)] For all $t \in \mathbb{R}$, $|\tilde{u}^h(t)| \leq c_{12} e^{-\frac{1}{2}\sqrt{\ve}|t - \tilde{T}_k|}$,
	\item[(iv)] For all $t \in [\tilde{T}_k-2\pi,\tilde{T}_k+2\pi] $, we have that $|\tilde{v}^h(t)|\leq c_8 M$,
	\item[(v)] For all $T \geq 0$,
	\begin{align*}
	 ||\tilde{h}_t||^2_{H^2((-\infty,\tilde{T_k}-T])^2}  +||\tilde{h}_t||^2_{H^2([\tilde{T_k}+T, \infty))^2} & \leq c_{12}(M^4 + \varpi^4) \left( \frac{\log^2 T}{T} \wedge \frac{\sqrt{\ve} |\log \ve|}{8} \right).
	\end{align*}
	\item[(vi)] Specifically, for $T = L$, we have
	\begin{align*}
	||\tilde{h}_t||^2_{H^2((-\infty,\tilde{T_k}-L])^2} + ||\tilde{h}_t||^2_{H^2([\tilde{T_k}+L, \infty))^2} & \leq c_{12} (M^4 + \varpi^4)  \frac{\log L}{L}.
	\end{align*}
\end{itemize}
\end{lemma}

We will use the following simple Lemma:
\begin{lemma} \label{l:basic}
	Assume $y_0 \geq 4$ and let 
	$ y_{j+1} = y_j\left( 1 + 1/ \log y_j \right)$.
	Then for some absolute implicit constant,
	$$ \sum_{j=0}^{\infty} \frac{\log y_j}{y_j} \ll  \frac{\log^2 y_0}{y_0}.$$
\end{lemma}

\begin{proof}
We first show inductively in $j=0,1,...$ that $y_j \geq e^{\frac{1}{2}\sqrt{x+j}}$, where $x$ is chosen so that $y_0=e^{\frac{1}{2}\sqrt{x}}$, i.e. $x=4\log^2 y_0$. Indeed, by the Mean Value Theorem, there is a real number $z$, $j\leq z \leq j+1$, so that
\begin{align*}
  e^{\frac{1}{2}\sqrt{x+j+1}}& =  e^{\frac{1}{2}\sqrt{x+j}}+\frac{e^{\frac{1}{2}\sqrt{x+z}}}{4 \sqrt{x+z} } \leq e^{\frac{1}{2}\sqrt{x+j}}+\frac{e^{1/2}}{8} \cdot \frac{e^{\frac{1}{2}\sqrt{x+j}}}{ \frac{1}{2}\sqrt{x+j}  } \leq  y_j + \frac{y_j}{\log y_j} = y_{j+1}. 
\end{align*}
Now as $\log y/y$ is decreasing for $y \geq 4$,
\begin{align*}
\sum_{j=0}^{\infty} \frac{\log y_j}{y_j} & \leq \frac{\log y_0}{y_0}+\int_x^{\infty} \frac{1}{2}\sqrt{z}e^{-\frac{1}{2}\sqrt{z}} dz  \ll \frac{\log y_0}{y_0}+ xe^{-\frac{1}{2}\sqrt{x}}= \frac{\log y_0}{y_0} + 4 \frac{\log^2 y_0}{y_0} \ll \frac{\log^2 y_0}{y_0}.
\end{align*}
\end{proof}

\begin{proof}[Proof of Lemma \ref{p:B4bounds}] Let
\begin{align*} 
  \tilde{h}(t) = 
  \begin{cases}
  \varphi^+_{\tilde{T}_k-L(1+1/\log L),\tilde{T}_k-L}(t) \: \cdot (q(t) - \tilde{q}_k(t)), & t \leq \tilde{T}_k, \\
  \varphi^-_{\tilde{T}_k+L, \tilde{T}_k+L(1+1/\log L)}(t) \: \cdot (q(t) - \tilde{q}_k(t)), & t \geq \tilde{T}_k.		 
  \end{cases}
\end{align*}
Now, (i) and (ii) follow from the definition of the smoothening functions (\ref{d:smooth}), and (iii) from (\ref{r:uhzero}), the definition of $\tilde{q}_k$ and (\ref{r:main1}). We claim that for each $T \geq 4 / \sqrt{\ve}$,
\begin{equation} \label{r:boundstemp}
 ||\tilde{h}_t||^2_{H^2([\tilde{T}_k+T,\tilde{T}_k+ T (1 + 1/\log T)])^2} \ll (M^4 + \varpi^4)  \frac{\log T}{T}.
\end{equation}
For $T \leq L$, this follows directly from (\ref{r:defB4}) with $\tau=\tilde{T}_k+T$, where the bounds on $\tilde{q}_k$ follow from its definition (as it is a translate of $q_k \in \mathcal{H}$), the bounds on $q_k$ from (\ref{r:vhleft})-(\ref{r:three}), and finally by using (\ref{r:varphib}). To show (iv), we use (i), and the fact that for all  $t \in [\tilde{T}_k-2\pi,\tilde{T}_k+2\pi]$, $q^0(t)=\tilde{q}^k(t)$, thus by (\ref{r:main2})
\begin{align*}
|\tilde{v}^h(t)| = |v(t)-\tilde{v}_k(t)| = |v(t)-v^0(t)| \leq  c_8 M.
\end{align*}
As by (ii), $h(t)$ vanishes for $t \geq L(1+1/\log L)$, the claim holds for $T \geq L(1+1/\log L)$. For $T \in [L, L(1+1/\log L)]$, (\ref{r:boundstemp}) similarly follows from the case $T = L$. Analogously we obtain for such $T$,
\begin{align}
 ||\tilde{h}_t||^2_{H^2([\tilde{T}_k- T (1 + 1/\log T),\tilde{T}_k-T] )^2} & \ll (M^4 + \varpi^4)  \frac{\log T}{T}, \label{r:boundstemp2} 
\end{align}
and
\begin{equation} \label{r:anotherterm}
||\tilde{h}_t||^2_{H^2([-4/\sqrt{\ve},0])^2}  \ll (M^4 + \varpi^4)  \: \sqrt{\ve}, \quad ||\tilde{h}_t||^2_{H^2([0,4/\sqrt{\ve}] )^2}  \ll (M^4 + \varpi^4)  \: \sqrt{\ve}.
\end{equation}
Now (vi) follows from (\ref{r:boundstemp}) and (\ref{r:boundstemp2}) with $T=L$, and again by noting that by (ii), $\tilde{h}(t)$ vanishes for $t \leq \tilde{T}_k-L(1+1/\log L)$ and $t \geq \tilde{T}_k+L(1+1/\log L)$. We obtain (v) as follows: in the case $T \geq 4 /\sqrt{\ve}$, we combine (\ref{r:boundstemp}) and (\ref{r:boundstemp2}) while inserting a sequence of $y_0=T$, $y_j=y_{j-1}(1+1/\log y_j)$ instead of $T$, and applying Lemma \ref{l:basic}. If $T \leq 4 /\sqrt{\ve}$, we add another term in the that estimate by using (\ref{r:anotherterm}).
\end{proof}

\section{Lower bound on the action dissipation} \label{s:dissipation}

In this section we develop a lower bound for the dissipation of the action with respect to the dynamics (\ref{r:grad}).

We now fix the constant $M$ with
\begin{equation} \label{r:Mcond}
M = 2\pi + 2( \varpi+1)(R+\mu) + 6R^{1/2}\ve^{1/4}.
\end{equation}
(clearly $M$ satisfies (\ref{r:mbound}) as required). 
Let $\mathcal{C}$ be the closure in $H^2_{\text{loc}}(\mathbb{R})^2$ of the set of all $h = (u^h,v^h) \in H^3_{\text{loc}}(\mathbb{R})^2$ satisfiying for all $t \in \mathbb{R}$ and all $T  \geq 0$,
\begin{align}
|u^h(t)| & \leq c_{12}\: e^{-\frac{1}{2}\sqrt{\ve}|t|+2\pi}, \label{d:Cone} \\
|v^h(0)| & \leq c_8M, \label{d:Ctwo} \\
||h_t||^2_{H^2((-\infty,-T])^2}  +||h_t||^2_{H^2([T, \infty))^2} & \leq 2 c_{12}(M^4 + \varpi^4) \left( \frac{\log^2 T}{T} \wedge \frac{\sqrt{\ve} |\log \ve|}{8} \right). \label{d:Cthree}
\end{align}
Consider for $(q,h) \in \mathcal{H} \times \mathcal{C}$, $q=(u,v)$, $h=(u^h,v^h)$,
\begin{align*}
E_q(h) & = \int_{-\infty}^0 L_{\omega^-(q)}(q+h,q_t+h_t,t) dt + \int_0^{\infty}L_{\omega^+(q)}(q+h,q_t+h_t,t)
 + (\omega^+(q)-\omega^-(q))(v(0)+v^h(0)), \\
D_q(h) & = \int_{-\infty}^{\infty}  (q+h)_s^2 dt,
\end{align*}
where we take $\omega^-(q)=\lim_{t \rightarrow -\infty} v_t $ and $\omega^+(q)=\lim_{t \rightarrow\infty} v_t $, and $q_s$ is evaluated by inserting (\ref{r:grad}).
We establish an uniform lower bound on the action dissipation $D_q$ on a certain level of action:
\begin{proposition} \label{p:dissipation}
	There exists a constant $\Delta_1>0$, $0 \leq \Delta_1 \leq \Delta_0/2$, depending on the region of instability $[\omega^-,\omega^+]$, $R$, $\Delta_0$ and $f$, and constants $0 \leq \Delta_0(q) \leq \Delta_0/2$ defined for all $q \in \mathcal{H}$, so that for all $(q,h) \in \mathcal{H} \times \mathcal{C}$, if
	$$ |E_q(h)-E_q(0)-\Delta_0(q)| \leq \Delta_1,$$
	then
	$$ D_q(h) \geq \Delta_1. $$
\end{proposition}
To obtain $\Delta_1$ we introduce for any $q \in \mathcal{H}$:
\begin{align}
\Delta_1(q,e) & = \inf \left\lbrace D_q(h), \: \: h \in \mathcal{C}, \: E_q(h)=E_q(0)+e \right\rbrace,  \\
\Delta_1(q) & = \sup_{e \in [0,\Delta_0]} \Delta_1(q,e).
\end{align}
We prove the Proposition in several steps. First we recall an infinite-dimensional version of the Morse-Sard theorem, which will enable us to deduce that $\Delta_1(q)>0$ for all $q \in \mathcal{H}$. We then in several lemmas establish various continuity and lower semi-continuity properties, which combined with compactness of $\mathcal{H}$, $\mathcal{C}$ enables us to complete the proof.

Recall first the Poho\v{z}aev infinite-dimensional version of the Morse-Sard theorem. Consider a real functional $E$ on a real, separable, reflexive Banach space $\mathcal{Y}$. We say that $E$ is Fredholm, if it is $C^2$ (in the sense of Fr\'{e}chet derivatives), and the dimension of $\text{Ker } D^2 E (h)$, $D^2E(h) : \mathcal{Y} \rightarrow \mathcal{Y}^*$ is finite dimensional for any $h \in \mathcal{Y}$. (Equivalently, $D^2E$ is Fredholm, as in this case it suffices to check finite dimensionality of the kernel \cite{Pohozaev:68}.) A critical value of $E$ is any value $e \in \mathbb{R}$ for which there exists $h\in \mathcal{Y}$ so that $E(h)=e$ and $D E(h) = 0$.

\begin{lemma} \label{l:sard} {\bf Morse-Sard-Poho\v{z}aev \cite{Pohozaev:68}.}
	Assume that $E : \mathcal{Y} \rightarrow \mathbb{R}$ is a real, $C^k$ functional defined on a real, separable, reflexive Banach space $\mathcal{Y}$. Assume that 
	$\dim (\text{Ker } D^2 E (h) )\leq m < \infty$
	for any $h\in \mathcal{Y}$, and let $k \geq \max \left\lbrace m,2 \right\rbrace$. Then the set of critical values of $E$ has Lebesgue measure 0.
\end{lemma}

We will apply Lemma \ref{l:sard} to the functionals $h \mapsto E_q(h)$ for $q \in \mathcal{H}$. Let $\mathcal{Y}$ be the set of all $q=(u,v) \in H^2_{\text{loc}}(\mathbb{R})^2$ such that $||q||_{\mathcal{Y}} < \infty$, where
\begin{equation}
||q||_{\mathcal{Y}} = \left(  \int_{-\infty}^{\infty}e^{\frac{1}{4}\sqrt{\ve}\: |t|}u(t)^2 dt + |v(0)|^2 + ||u_t||^2_{H^1(\mathbb{R})} +  ||v_t||^2_{H^1(\mathbb{R})}  \right)^{1/2}. \label{r:norm}
\end{equation}
The space $\mathcal{Y}$ is a Hilbert space, as the norm (\ref{r:norm}) is induced by a scalar product defined in a straight-forward way. Thus $\mathcal{Y}$ is separable and reflexive.

Let us establish compactness of $\mathcal{H}$ and $\mathcal{C}$ and continuity of $D_q$, $E_q$. Recall that assumed topology on $\mathcal{H}$ and $\mathcal{C}$ is induced by $H^2_{\text{loc}}(\mathbb{R})^2$. It is straightforward to verify that on $\mathcal{C}$ it coincides with the topology induced by the $\mathcal{Y}$-norm.

\begin{lemma} \label{l:comcon}
	(i) The sets $\mathcal{H}$ and $\mathcal{C}$ are compact,
	
	(ii) The functions $(q,h) \mapsto E_q(h),D_q(h)$ are well-defined and continuous on $\mathcal{H} \times \mathcal{C}$.
\end{lemma}

\begin{proof} We have showed compactness of $\mathcal{H}$ in Lemma \ref{l:compact}. Compactness of $\mathcal{C}$ follows directly from the definition, the compact embedding theorem applied to $h$ restricted to any bounded closed interval, and a diagonalization argument. The claim (ii) follows easily from the definitions of $E_q$, $D_q$, the uniform bounds on $q \in \mathcal{H}$ in Lemma \ref{l:heteroprop}, the definition of $\mathcal{C}$ and the assumed localized topologies on the sets $\mathcal{H}$ and $\mathcal{C}$.
\end{proof}

\begin{lemma} \label{r:Eproperties}
For any $q \in \mathcal{H}$,

(i) $E_q : \mathcal{Y} \rightarrow \mathbb{R}$ is $C^4$,

(ii) $E_q$ is Fredholm and the dimension of $\text{Ker }D^2E_q$ is at most $4$,

(iii) If $DE_q(h) \neq 0$, then $D_q(h) > 0$.
\end{lemma}

\begin{proof} Let $q=(u,v)$, $h=(u^h,v^h) \in \mathcal{Y}$, and $g^{(1)}, g^{(2)}, g^{(3)}, g^{(4)} \in \mathcal{Y}$, $g^{(j)}=(u^{(j)},v^{(j)})$. We will show that the Fr\'{e}chet derivatives of $E_q$ are for  given with
\begin{align}
DE_q(h)g^{(1)} & =  \int_{-\infty}^{0} (v_t -\omega+v^h_t)v^{(1)}_t dt + \int_0^{\infty}(v_t -\tilde{\omega}+v^h_t)v^{(1)}_t dt + (\tilde{\omega}-\omega)v^{(1)}(0) \notag \\
& \quad + \int_{-\infty}^{\infty} \left\lbrace (u_t+u^h_t)u^{(1)}_t+D_{u,v}V(q(t)+h(t),t)g^{(1)}(t) \right\rbrace dt, \label{r:der1} \\
D^2E_q(h)(g^{(1)},g^{(2)}) & =  \int_{-\infty}^{\infty} \left\lbrace g^{(1)}_t g^{(2)}_t+D^2_{u,v} V(q(t)+h(t),t)(g^{(1)}(t),g^{(2)}(t))  \right\rbrace dt, \label{r:der2} \\
D^k E_q(h)(g^{(1)},...,g^{(k)})   & =  \int_{-\infty}^{\infty}  \left\lbrace D^k_{u,v} V(q(t)+h(t),t) (g^{(1)}(t),...,g^{(k)}(t)) \right\rbrace dt, \label{r:higher}
\end{align}
where $k=3,4$. We first show that the integrals on the right-hand sides are finite. Consider the terms containing $u_t,u^h_t,v_t,v^h_t$ in (\ref{r:der1}). They are absolutely integrable by Cauchy-Schwartz, as $||u^h_t||^2_{L^2(\mathbb{R})}<\infty$, $||v^h_t||^2_{L^2(\mathbb{R})}<\infty$ by the definition of $\mathcal{Y}$, and by Lemma \ref{l:heteroprop} applied to $u_t,|v_t-\omega|,|v_t-\tilde{\omega}|$.	Analogously the term $g^{(1)}_t g^{(2)}_t$ in (\ref{r:der2}) is absolutely integrable.
Thus it suffices to show that for any integers $i,j \geq 0$, $i+j=k$, $k=1,2,3,4$, the integral
\begin{align}
X:=\int_{-\infty}^{\infty} \left\lbrace \partial^{i}_{u}\partial^{j}_{v} V(q(t)+h(t),t)u^{1}(t)...u^{i}(t)v^{i+1}(t)...v^k(t) \right\rbrace dt
\end{align}
is absolutely integrable. It is straightforward to check that for $g^{(j)}=(u^{(j)},v^{(j)}) \in \mathcal{Y}$,
\begin{align*}
|u^{(j)}(t)| & \leq ||u^{(j)}||_{L^{\infty }(\mathbb{R})}\ll ||u^{(j)}||_{H^1(\mathbb{R})} \leq ||g^{(j)}||_{\mathcal{Y}}, \\  
|v^{(j)}(t)| & \leq |v^{(j)}(0)| + \int_0^t |v^{(j)}_t(\tau)|d\tau \ll (1+|t|^{1/2})||g^{(j)}||_{\mathcal{Y}}.
\end{align*}
If $i \geq 1$, we thus have by applying uniform bounds on derivatives of $V$,
\begin{align}
X & \ll_f ||g^{(2)}||_{\mathcal{Y}}...||g^{(k)}||_{\mathcal{Y}} \int_{-\infty}^{\infty} |u^{(1)}(t)|(1+|t|^{1/2})^{k-1}dt \notag \\
& \ll_f   ||g^{(2)}||_{\mathcal{Y}}...||g^{(k)}||_{\mathcal{Y}} \left( \int_{-\infty}^{\infty} e^{-\frac{1}{4}\sqrt{\ve}|t|} (1+|t|^{k-1}) dt \right)^{1/2} \left( e^{\frac{1}{4}\sqrt{\ve}|t|} |u^{(1)}(t)|^2 dt 
\right)^{1/2} \notag \\
& \ll_{f,\ve}  ||g^{(1)}||_{\mathcal{Y}}||g^{(2)}||_{\mathcal{Y}}...||g^{(k)}||_{\mathcal{Y}}. \label{r:higherhelp}
\end{align}
Analogously, in the case $i=0$,
\begin{align*}
X & \ll ||g^{(1)}||_{\mathcal{Y}}...||g^{(k)}||_{\mathcal{Y}} \int_{-\infty}^{\infty}  |\partial^{k}_{v}V(q(t)+h(t),t)| (1+|t|^{1/2})^k dt \\ 
& \ll_f ||g^{(1)}||_{\mathcal{Y}}...||g^{(k)}||_{\mathcal{Y}} \int_{-\infty}^{\infty} \left( 1-\cos (u(t)+u^h(t)) \right) (1+|t|^{1/2})^k dt \\
& \ll_f ||g^{(1)}||_{\mathcal{Y}}...||g^{(k)}||_{\mathcal{Y}}  \int_{-\infty}^{\infty} \left(e^{-\frac{1}{2}\sqrt{\ve} (t-t_0)} + u^h(t)^2 \right)(1+|t|^{1/2})^k dt,
\end{align*}
where in the last row we applied 
$1 - \cos (u(t)+u^h(t)) \ll (u(t)-2\pi \mathbf{1}_{[t_0,\infty)}(t))^2 + u^h(t)^2$
and then (\ref{r:uhzero}). Analogously to (\ref{r:higherhelp}) we establish bound on the remaining terms containing $u^h$ and get
\begin{equation}
 X \ll_{f,\ve} ||g^{(1)}||_{\mathcal{Y}}...||g^{(k)}||_{\mathcal{Y}} ( 1 + ||h||^2_{\mathcal{Y}}). \label{r:higherhelp2}
\end{equation}
We now show that (\ref{r:der1}) is indeed the Fr\'{e}chet derivative of $E_q$; the proof for higher order derivatives is analogous (we use H\"older continuity of fourth derivatives assumed in (A1) for $D^4E_q$). By the Mean Value theorem and analogously as when evaluating (\ref{r:higherhelp}), (\ref{r:higherhelp2}), we obtain
\begin{align*}
|E_q(h+g^{(1)})-E_q(h)-DE_q(h)g^{(1)}| & \leq  \int_{-\infty}^{\infty} \left\lbrace \frac{1}{2} (g^{(1)}_t)^2 + |D^2_{u,v}V(q(t)+h(t),t)||g^{(1)}|^2 \right\rbrace dt \\
&\ll_{f,\ve} (1+ ||h||^2_{\mathcal{Y}}) ||g^{(1)}||^2_{\mathcal{Y}},
\end{align*}
which by the definition of the Fr\'{e}chet derivative gives the claim.

To show (ii), we observe that the kernel of $D^2E_q(h)$ is by partial integration the set of all $g^{(1)} \in \mathcal{Y}$ such that $-g^{(1)}_{tt}+	D^2_{u,v} V(q(t)+h(t),t)g^{(1)}(t)=0$ for all $t \in \mathbb{R}$. This is a system of four linear ordinary differential equations (a linearization of (\ref{r:EL})), so the space of its solutions in $\mathcal{Y}$ is at most four-dimensional.

By partial integration and (\ref{r:grad}), we can write (\ref{r:der1}) as
$ DE_q(h)g^{(1)} =  \int_{-\infty}^{\infty} -(q+h)_sg^{(1)} dt$, which is clearly $\equiv 0$ if and only if $(q+h)_s \equiv 0$, which implies (iii). 
\end{proof}

We are now ready to apply the Morse-Sard-Poho\v{z}aev Lemma.

\begin{lemma} \label{l:notzero}
	For any $q \in \mathcal{H}$, there exists $e$, $0 \leq e \leq \Delta_0$ such that $\Delta_1(q,e) > 0$.
\end{lemma}

\begin{proof}
Because of Lemma \ref{r:Eproperties}, (i) and (ii), we can apply the Morse-Sard-Poho\v{z}aev Lemma \ref{l:sard} to the functional $E_q$, and find any level set $E_q(h)=e$, $0 \leq e \leq \Delta_0$, so that for any $h \in \mathcal{Y}$, $E_q(h)=E_q(0)+e$ implies that $DE_q(h) \neq 0$. By  Lemma \ref{r:Eproperties}, (iii), for any $h \in \mathcal{Y}$, $D_q(h) > 0$. 

By Lemma \ref{l:comcon}, (i) and (ii), the level set $\lbrace E_q(h)=E_q(0)+ e \rbrace \cap \mathcal{C}$ is a compact subset of $\mathcal{Y}$. By continuity of $D_q$, we can bound $D_q(h)$ away from zero on that level set, which by definition of $\Delta_1(q,e)$ completes the proof.
\end{proof}

\begin{lemma}
	The function $(q,e) \mapsto \Delta_1(q,e)$ is lower semi-continuous on $\mathcal{H} \times [0,\Delta_0]$.  
\end{lemma}

\begin{proof}
	Choose a sequence $(q_n,e_n) \in \mathcal{H} \times [0,\Delta_0]$ converging to $(q,e)$. As the level sets $\lbrace E_q(h)=E_q(0)+ e \rbrace \cap \mathcal{C}$ are compact, we can find $h_n \in \mathcal{C}$, $n \in \mathbb{N}$ so that $E_{q_n}(h_n)=E_{q_n}(0)+e_n$ and $D_{q_n}(h_n)=\Delta_1(q_n,e_n)$. 
	
	Find a subsequence $k_n$ so that $\liminf_{n \rightarrow \infty} \Delta_1(q_n,e_n)=\lim_{n \rightarrow \infty}\Delta_1(q_{k_n},e_{k_n})$, and a further subsequence (denoted again by $k_n$) so that $h_{k_n}$ is convergent in $\mathcal{C}$. 	Let $h = \lim_{n \rightarrow \infty} h_{k_n}$. Now because of continuity of $(q,h)\mapsto D_q(h)$ by Lemma \ref{l:comcon}, (ii), we have
	\begin{align*}
	\liminf_{n \rightarrow \infty} \Delta_1(q_n,e_n)& =\lim_{n \rightarrow \infty}\Delta_1(q_{k_n},e_{k_n}) = \lim_{n \rightarrow \infty} D_{q_n}(h_n) = D_q(h). 
	\end{align*}
	Again by Lemma \ref{l:comcon}, (ii), $E_q(h)=E_q(0)+e$, thus by definition $D_q(h) \geq \Delta_1(h,e)$, which completes the proof.	
\end{proof}

\begin{proof}[Proof of Proposition \ref{p:dissipation}]
We prove it by using several times the well-known properties of lower semi-continuous functions. As $q \mapsto \Delta_1(q)$ is by definition a supremum of a family of lower semi-continuous functions, it is lower-semi continuous. By Lemma \ref{l:notzero}, $\Delta_1(q)>0$ for every $q \in \mathcal{H}$. As lower semi-continuous functions on a compact set attain a minimum, and $\mathcal{H}$ is compact, there exists $\tilde{\Delta}_1 >0$ so that  $\Delta_1(q) \geq \tilde{\Delta}_1$ for all $q \in \mathcal{H}$.

Now consider the set $\lbrace \Delta_1(q,e) > \tilde{\Delta}_1 / 2 \rbrace \subset \mathcal{H} \times [0,\Delta_0/2]$. By lower semi-continuity of $(q,e)\mapsto \Delta_1(q,e)$, it is an open subset of $\mathcal{H} \times [0,\Delta_0/2]$. We can find its open cover consisting of sets $U_i \times B(e_i,r_i)$, $i \in \mathcal{I}$, where $U_i \subset \mathcal{H}$ is open in $\mathcal{H}$, and $B(e_i,r_i)$ are open balls in $\mathbb{R}$. By choice of $\tilde{\Delta}_1$, the set $\lbrace \Delta_1(q,e) > \tilde{\Delta}_1 / 2 \rbrace$ projects in the first coordinate to the entire $\mathcal{H}$, so $U_i$, $i \in \mathcal{I}$ is an open cover of $\mathcal{H}$. By compactness, we find its finite subcover. Denote it by $U_1,...,U_n$, and its associated open balls by $B(e_1,r_1),...,B(e_n,r_n)$.

We now set $\Delta_1 = \min \lbrace \tilde{\Delta}_1/2,r_1,...,r_n \rbrace$, and set for $q \in U_j$, $\Delta_0(q)=e_j$ (we choose any $j$ if $q$ is in more than one $U_j$). It is straightforward to check that this completes the proof.
\end{proof}

\section{Local upper bounds on the action} \label{s:upper}

This section contains the core of the argument, as we show that the action within two "intersections", or more precisely in the segment $[\tilde{T}_k-L,\tilde{T}_k+L]$, can not increase more than an arbitrarily small constant, proportional to 
$\log L/L$. This will complete our method of control of the dynamics. The proof relies on an action-balance law, stating that the change of action with respect to (\ref{r:grad}) is equal to the action dissipation and action flux. Let $\tilde{E}_k, \tilde{D}_k, \tilde{F}_k : \mathcal{X} \rightarrow \mathbb{R}$ be the truncated action, action dissipation and action flux near $\tilde{T}_k$, defined for $q=(u,v)$ as 
\begin{align*}
\tilde{E}_k(q) & =\int_{\tilde{T}_k-L}^{\tilde{T}_k} L_{\omega_k}(q,q_t,t)dt + \int_{\tilde{T}_k}^{\tilde{T}_k+L} L_{\omega_{k+1}}(q,q_t,t)dt  + (\omega_{k+1}-\omega_k)v(\tilde{T}_k), \\
\tilde{D}_k(q)&  = \int_{\tilde{T}_k-L}^{\tilde{T}_k+L} \left\lbrace u_s^2 + v_s^2 \right\rbrace dt, \\
\tilde{F}_k(q)& = u_t(\tilde{T}_k+L)u_s(\tilde{T}_k+L)+(v_t(\tilde{T}_k+L)- \omega_{k+1})v_s(\tilde{T}_k+L) \\ &\quad - u_t(\tilde{T}_k-L)u_s(\tilde{T}_k-L)-(v_t(\tilde{T}_k-L)- \omega_k)v_s(\tilde{T}_k-L).
\end{align*}
Let $\mathcal{B}_5$ be the set of all $q \in \mathcal{B}_4$ so that for all $k \in \mathbb{Z}$,
\begin{equation}
\tilde{E}_k(q) \leq \tilde{E}_k(\tilde{q}_k) + \Delta_0(q_k), \label{r:Econd}
\end{equation}
where $\Delta_0(\tilde{q}_k)$, $\Delta_1$ are as constructed in Proposition \ref{p:dissipation}. 

\begin{proposition} \label{p:b6}
	There exists an absolute constant $c_{13}>0$ such that, if 
	\begin{equation}
	  L \geq c_{13}(\varpi^5+M^5) \frac{| \log \Delta_1 |}{\Delta_1},	
	\label{r:Lcond}
	\end{equation}
	then $\mathcal{B}_5$ is $\mathcal{A}$-relatively $\xi$-invariant. Furthermore, $q^0 \in \mathcal{B}_5$.
\end{proposition}
\noindent (Recall that $M$ in (\ref{r:Lcond}) is given by (\ref{r:Mcond}).) To prove it, we first establish the action-balance law (\ref{r:balance}), then find upper bounds on the action flux, lower bounds on the action dissipation on the energy level $\tilde{E}_k(q) = \tilde{E}_k(\tilde{q}_k) + \Delta_0(\tilde{q}_k)$, and then complete the proof. The constant $c_{13}$ may change throughout the section.

\begin{lemma} \label{l:balance}
	For any $q \in \mathcal{X}$,
	\begin{equation}
	\frac{d}{ds}\tilde{E}_k(q)=-\tilde{D}_k(q) + \tilde{F}_k(q). \label{r:balance}
	\end{equation}
\end{lemma}
\begin{proof} The proof is straightforward, by differentiating $\tilde{E}_k(q)$, partial integration and inserting (\ref{r:grad}).
\end{proof}

\begin{lemma} \label{l:flux}
	There exists an absolute constant $c_{13}>0$ such that if (\ref{r:Lcond}) holds, then for any $q \in \mathcal{B}_4$ and any $k \in \mathbb{Z}$,
	$|\tilde{F}_k(q)| \leq \Delta_1 / 4.$
\end{lemma}

\begin{proof} By (\ref{r:One}), (\ref{r:defB3u}), the definition of $\lambda(\tau)$, using that $\dtau \leq L_{k(\tau)}$  and $\ve \leq 1$, we obtain
	\begin{align}
	|u_t(\tau)| & \ll ||u_t||_{H^1([\tau,\tau+1])} \ll 
	(M + \varpi) \lambda(\tau)^{1/2} + ||u^0_t||_{L^2[\tau,\tau+1]} \notag \\ & \ll (M + \varpi) \left( \frac{\log ||\tau ||}{\dtau} \right)^{1/2}, \label{r:uder}
	\end{align}
	and analogously
	\begin{align}
	|v_t(\tau) - \omega_{k(\tau)}| & \ll (M + \varpi)\left( \frac{\log ||\tau ||}{\dtau} \right)^{1/2}.  \label{r:vder}
	\end{align}
	Because of (\ref{r:Vubound}) and (\ref{r:Vvbound}) in the Appendix C, we get for all $q \in \mathcal{B}_4 \subset \mathcal{B}_1$, $|V_u| \ll e^{-\sqrt{\ve} \dt / 2} \ll (\log \dt / \dt)^{1/2}$, $|V_v| \ll e^{-\sqrt{\ve} \dt} \ll (\log \dt / \dt)^{1/2}$. 
	Using this and (\ref{r:defB4u}), (\ref{r:defB4v}) applied to $|u_{tt}|$, $|v_{tt}|$ analogously as above, we get
	\begin{align}
	|u_s| & \leq |u_{tt}|+|V_u| \ll (M^2 + \varpi^2)\left( \frac{\log \dtau}{\dtau} \right)^{1/2}, \\
	|v_s| & \leq |v_{tt}|+|V_v| \ll (M^2 + \varpi^2)\left( \frac{\log \dtau}{\dtau} \right)^{1/2}.	\label{r:sbound}	
	\end{align}	
 As by definition, for $\tau = \tilde{T_k} \pm L$, $\dtau = L$, we deduce that for any $q \in \mathcal{B}_4$,
 \begin{equation}
 \tilde{F}_k(q) \leq c_{13} (M^3+\varpi^3)\frac{\log L}{L} \label{r:Ffinal}
 \end{equation}
 for some absolute constant $c_{13}>0$. Now it is straightforward to check that 
 $$
 L \gg (\varpi^5+M^5) \frac{| \log \Delta_1 |}{\Delta_1} \gg (\varpi^3+M^3) \log (\varpi^3+M^3) \frac{| \log \Delta_1 |}{\Delta_1}  
 $$
 suffices for the right-hand side of (\ref{r:Ffinal}) to be $\ll \Delta_1$, which completes the proof for a large enough absolute constant $c_{13}$.
\end{proof}

\begin{lemma} \label{l:dissip}
There exists an absolute constant $c_{13}>0$ such that if (\ref{r:Lcond}) holds, then for any $q=(u,v) \in \mathcal{B}_4$ and any $k \in \mathbb{Z}$, there exists $h \in \mathcal{C}$, so that
	\begin{gather}
	|\tilde{E}_k(q)-E_{q_k}(h) - T_k(\omega_k^2-\omega_{k+1}^2)/2-(\omega_{k+1}-\omega_k)(\tilde{V}_k-V_k)|  \leq \Delta_1 / 2, \label{r:Dis1} \\
	|\tilde{D}_k(q)-D_{q_k}(h)|  \leq \Delta_1 / 2. \label{r:Dis2}
	\end{gather}
Furthermore, for all $k \in \mathbb{Z}$,
   	\begin{align}
   	|\tilde{E}_k(\tilde{q}_k)-E_{q_k}(0)- T_k(\omega_k^2-\omega_{k+1}^2)/2-(\omega_{k+1}-\omega_k)(\tilde{V}_k-V_k)| \leq \Delta_1 / 2. \label{r:Dis3}
   	\end{align}
\end{lemma}

\begin{proof} Fix $k \in \mathbb{Z}$. We define $\tilde{h}$ as in Lemma \ref{p:B4bounds}, and let $h(t)=\tilde{h}(t+\tilde{T}_k-T_k)$. We first show that $h \in \mathcal{C}$. As by definition $|T_k|\leq 2\pi$, by inserting the definition of $h$ in the bounds in Lemma \ref{p:B4bounds}, we easily see that
(\ref{d:Cone}) and (\ref{d:Ctwo}) follow from Lemma \ref{p:B4bounds}, (iii) and (iv). Analogously we show that (\ref{d:Cthree}) follows from Lemma \ref{p:B4bounds}, (vi). 

Denote by $X_k=T_k(\omega_k^2-\omega_{k+1}^2)/2 + (\omega_{k+1}-\omega_k)(\tilde{V}_k-V_k)$ (a constant independent of $q$).
By definitions, the partial integration to change the range of integration in the second line, and substitution $t \rightarrow t + \tilde{T}_k-T_k$ in the third line, and finally by using Lemma \ref{p:B4bounds}, (i), we obtain
\begin{align}
E_{q_k}(h) & = \int_{-\infty}^0 L_{\omega_k}(q_k+h,(q_k)_t+h_t,t) dt + \int_0^{\infty}L_{\omega_{k+1}}(q_k+h,(q_k)_t+h_t,t)
+ (\omega_{k+1}-\omega_k)(v_k(0)+v^h(0)) \notag \\
& = \int_{-\infty}^{T_k} L_{\omega_k}(q_k+h,(q_k)_t+h_t,t) dt + \int_{T_k}^{\infty}L_{\omega_{k+1}}(q_k+h,(q_k)_t+h_t,t) \notag
\\ & \hspace{150pt} + (\omega_{k+1}-\omega_k)(V_k+v^h(T_k))  + (\omega^2_k-\omega_{k+1}^2)T_k/2 \notag \\
& = \int_{-\infty}^{\tilde{T}_k} L_{\omega_k}(\tilde{q}_k+\tilde{h},(\tilde{q}_k)_t+\tilde{h}_t,t) dt + \int_{\tilde{T}_k}^{\infty}L_{\omega_{k+1}}(\tilde{q}_k+\tilde{h},(\tilde{q}_k)_t+\tilde{h}_t,t)dt
+ (\omega_{k+1}-\omega_k)v(\tilde{T}_k)+X_k\notag \\
& = \tilde{E}_k(q) + \int_{-\infty}^{\tilde{T}_k-L} L_{\omega_k}(\tilde{q}_k+\tilde{h},(\tilde{q}_k)_t+\tilde{h}_t,t) dt 
+ \int_{\tilde{T}_k+L}^{\infty}L_{\omega_{k+1}}(\tilde{q}_k+\tilde{h},(\tilde{q}_k)_t+\tilde{h}_t,t)dt
+ X_k. \label{r:secondintegral}
\end{align} 
By the definition of $V$, $\tilde{q}_k$, and then (\ref{r:uhzero}) and Lemma (\ref{p:B4bounds}), (iii), we obtain
\begin{align*}
\int_{-\infty}^{\tilde{T}_k-L} V(\tilde{q}_k+\tilde{h},t)dt & \leq \ve \int_{-\infty}^{\tilde{T}_k-L} (1-\cos (\tilde{q}_k+\tilde{h}))\dt \ll \ve \int_{-\infty}^{T_k-L}q^2_k(t)dt + \ve \int_{-\infty}^{\tilde{T}_k-L}\tilde{h}^2_k(t)dt \\
& \ll \ve \int_{-\infty}^{T_k-L} e^{-\frac{1}{2}\sqrt{\ve}|t-T_k|} + \ve \int_{-\infty}^{\tilde{T}_k-L}  e^{-\frac{1}{2}\sqrt{\ve}|t-\tilde{T}_k|} \ll \sqrt{\ve}e^{-\frac{1}{2}\sqrt{\ve}L} \ll \frac{\log L}{L}.
\end{align*}
Inserting it in the definition of $L_{\omega_k}$ and combining with (\ref{r:uhone}), (\ref{r:vhleft}) and Lemma (\ref{p:B4bounds}), (vi), we get
\begin{align*}
\int_{-\infty}^{\tilde{T}_k-L} L_{\omega_k}(\tilde{q}_k+\tilde{h},(\tilde{q}_k)_t+\tilde{h}_t,t) dt & \ll ||(\tilde{u}_k)_t||^2_{L^2((-\infty,\tilde{T}_k-L])}+||(\tilde{v}_k)_t-\omega_k||^2_{L^2((-\infty,\tilde{T}_k-L])} \\ & \quad +||\tilde{h}_t||^2_{L^2([\tilde{T}_k-L(1+1/L),\tilde{T}_k-L])^2} + \frac{\log L}{L} \\
& \ll (M^4+\varpi^4)\frac{\log L}{L}.
\end{align*}
By proving an analogous statement for the second integral in (\ref{r:secondintegral}) and combining all the relations above, we see that
$$
|\tilde{E}_k(q)-E_{q_k}(h) - X_k|  \ll (M^4+\varpi^4)\frac{\log L}{L}.
$$
We complete the claim analogously as in Lemma \ref{l:flux}. The proof for (\ref{r:Dis2}) is analogous, as we also by Lemmas \ref{l:heteroprop} and \ref{p:B4bounds} control the second derivatives of $\tilde{q}_k$ and $\tilde{h}$. We get (\ref{r:Dis3}) by inserting $h=0$ in (\ref{r:Dis1}).
\end{proof}

\begin{proof}[Proof of Proposition \ref{p:b6}]
Let $c_{13}$ be the larger of the constants in Lemmas \ref{l:balance} and \ref{l:dissip}. Let for some $k \in \mathbb{Z}$, $s_1$ be the supremum of all the times $s$ such that (\ref{r:Econd}) holds. Then by continuity, 
\begin{equation}
\tilde{E}_k(q) = \tilde{E}_k(\tilde{q}_k) + \Delta_0(q_k). \label{r:Dis4}
\end{equation} Combining (\ref{r:Dis1}), (\ref{r:Dis3}) and (\ref{r:Dis4}),  we obtain $|E_{q_k}(h)-E_{q_k}(0)-\Delta_0(q_k)| \leq \Delta_1$. By Proposition \ref{p:dissipation}, we thus obtain $D_{q_k}(h) \geq \Delta_1$, so by (\ref{r:Dis2}), $\tilde{D}_k(q) \geq \Delta_1 / 2$. Inserting this and the bound from Lemma \ref{l:flux} in the action balance law (\ref{r:balance}), we get
\begin{equation*}
\frac{d}{ds}\tilde{E}_k(q) \leq - \Delta_1 / 2 + \Delta_1/4 \leq - \Delta_1 / 4 < 0,
\end{equation*}
which is in contradiction with the assumption. To show $q^0 \in \mathcal{B}_5$, we recall that $q^0$ and $\tilde{q}_k$ coincide on $[\tilde{T}_k-L,\tilde{T}_k+L]$. As we always have $\Delta_0(q_k) > 0$,  we conclude that $\tilde{E}_k(q^0) = \tilde{E}_k(\tilde{q}_k) < \tilde{E}_k(\tilde{q}_k)+\Delta_0(q_k)$, so $q^0 \in \mathcal{B}_5$ by definition.
\end{proof}

\section{Completion of construction of an invariant set} \label{s:invariant}

In this section we complete the construction of an invariant set $\mathcal{B}$ with respect to (\ref{r:grad}) by finally applying Lemma \ref{l:combine}.

\begin{proposition} \label{p:invariant}
Assume $\omega_k$ is a sequence in a closed subset $[\omega^-,\omega^+]$ satisfying (\ref{r:assumec}), and that $L$ satisfies
\begin{equation}
L \geq c_{14} \varpi^5 \frac{|\log \Delta_1|}{\Delta_1} \label{r:LLcond}
\end{equation}
for some large enough absolute constant $c_{14}$.
% (\ref{r:Lcond}) and (\ref{r:Lcond2}) below, and $M$ satisfies (\ref{r:Mcond}). 
Then there exists a $\xi$-invariant set $\mathcal{B} \subset \mathcal{B}_5  \subset \mathcal{X}$, such that $q^0 \in \mathcal{B}$.
\end{proposition}

\begin{remark} \label{r:notation2} Assume that as in the Remark \ref{r:notation} we fix a segment $[\omega^-,\omega^+]$ in a region of instability, and that for each $\omega,\tilde{\omega} \in [\omega^-,\omega^+]$ satisfying (\ref{r:assumec}) we chose a single $q \in \mathcal{H}$ (as such $q$ is not necessarily unique). We also fix $\mathcal{N}_q$ associated to such $q$ as in the definition of $\mathcal{H}$. We will see that the set $\mathcal{B}$ is then uniquely defined by the choice of $L$, $(\tilde{L}_k)_{k \in \mathbb{Z}}$ and $(\omega_k)_{k \in \mathbb{Z}}$, and satisfies all the relations in the definitions of $\mathcal{B}_1$-$\mathcal{B}_5$. In the proofs of the main theorems, we thus use the notation $\mathcal{B}(L,(\tilde{L}_k)_{k \in \mathbb{Z}},(\omega_k)_{k \in \mathbb{Z}})$. 
\end{remark}

We say that $q$ intersects the set $\mathcal{N}_k \subset \mathbb{R}^2$ at $k$, if there exists $(t,v) \in \mathcal{N}_k$ such that $q(t)=((2k+1)\pi,v)$. Analogously we define the notion of $q$ intersecting $\partial \mathcal{N}_k$ at $k$. We first in two lemmas establish that $q \in \mathcal{B}_5$ can not intersect $\partial \mathcal{N}_k$ at $k$, then define $\mathcal{B}_6$ and prove its $\xi$-relative invariance, and finally complete the proof of Proposition \ref{p:invariant}.

\begin{lemma} \label{l:B6one} There exists an absolute constant $c_{15}$ such that, if
	\begin{equation} 
	L \geq c_{15} \frac{|\log \mu|}{\sqrt{\ve}}, \label{r:Lcond2}
	\end{equation}
	then the following holds: for each $q \in \mathcal{B}_5$ intersecting $\partial \mathcal{N}_k$ at $k$, we have
\begin{equation}
\tilde{E}_k(q) \geq \tilde{E}_k(\tilde{q}_k) + \Delta_0. \label{r:morethan}
\end{equation}
\end{lemma}

\begin{proof}
	To prove (\ref{r:morethan}), we will approximate $q(t-T_k+\tilde{T}_k)-(2k \pi,\tilde{V}_k-V_k)$ with a $\tilde{q}=(\tilde{u},\tilde{v}) \in H^1_{\text{loc}}$ so that we can apply (\ref{r:LL}). Indeed, let $\tilde{q}(t)=q(t-T_k+\tilde{T}_k)-(2\pi k,\tilde{V}_k-V_k)$ for $t \in [T_k-L,T_k+L]$. We define $\tilde{v}(t)$ uniquely for all $t \in \mathbb{R}$ by $\tilde{v}_t = \omega_k$ for $t < T_k-L$, $\tilde{v}_t = \omega_{k+1}$ for $t > T_k+L$, $\tilde{v}$ continuous. Let
	\begin{equation*}
    \tilde{u}(t)=a \cdot e^{-\frac{1}{2}\sqrt{\ve} |t-T_k|}, \: t \leq T_k-L, \hspace{50pt} \tilde{u}(t)=2\pi - b \cdot e^{-\frac{1}{2}\sqrt{\ve} |t-T_k|}, \: t \geq T_k+L,
	\end{equation*}
	where the constants $a,b$ are uniquely chosen so that $\tilde{u}$ is continuous. By construction, (\ref{r:main1}) and the definition of $V$, we have
	\begin{align}
	|\tilde{E}_k(q) - \mathcal{L}_{\omega_k,\omega_{k+1}}(\tilde{q}) | & = \int_{-\infty}^{T_k-L}L_{\omega_k}(\tilde{q},\tilde{q}_t,t) dt + \int_{T_k+L}^{\infty}L_{\omega_{k+1}}(\tilde{q},\tilde{q}_t,t) dt \notag \\
	& \ll \ve \int_{-\infty}^{T_k-L}e^{-\frac{1}{2}\sqrt{\ve} |t-T_k|}dt + \ve \int_{T_k+L}^{\infty}e^{-\frac{1}{2}\sqrt{\ve} |t-T_k|}dt \ll \sqrt{\ve}e^{-\frac{1}{2}\sqrt{\ve}L} \leq e^{-\frac{1}{2}\sqrt{\ve}L}.  \label{r:EminLa}
	\end{align}
	Similarly, we obtain 
	\begin{equation}
	|\tilde{E}_k(\tilde{q}_k)-\mathcal{L}_{\omega_k,\omega_{k+1}}(q_k)| \ll  e^{-\frac{1}{2}\sqrt{\ve}L}. \label{r:EminLb}
	\end{equation} 
	 By $(\ref{s:bound})$ and the definition of $\Delta_0$, we have $\Delta_0 \ll \sqrt{\ve} \: \mu \leq \mu$. Thus we can choose $L \gg  |\log \mu| / \sqrt{\ve} \gg |\log \Delta_0| / \sqrt{\ve} $ for a large enough absolute constant, so that the left-hand sides od (\ref{r:EminLa}) and (\ref{r:EminLb}) are bounded by $\Delta_0 / 2$. Combining it with 
	(\ref{r:LL}), we complete the proof.
\end{proof}

\begin{lemma} \label{l:B6two}
If $q \in \mathcal{B}_5$, then for all $k \in \mathbb{Z}$, $q$ can not intersect  $\partial \mathcal{N}_k$ at $k$.
\end{lemma}

\begin{proof}
As $q \in \mathcal{B}_5$, (\ref{r:Econd}) holds. If $q$ would intersect $\partial \mathcal{N}_k$ at $k$, this would contradict (\ref{r:morethan}) and $\Delta_0(\tilde{q}_k) \leq \Delta_0/2$ established in Proposition \ref{p:dissipation}.
\end{proof}

We now define the number of times $q$ intersects of $\mathcal{N}_k$ at $k$. Let $\varUpsilon_k$ (as a function of $q=(u,v) \in \mathcal{X}$) be the set of all $t \in \mathbb{R}$ satifying 
$$
\varUpsilon_k = \lbrace \tilde{T}_{k-1} \cup \tilde{T}_{k+1} \rbrace \cup \lbrace t \in [\tilde{T}_{k-1},\tilde{T}_{k+1}], (t,v(t)) \in \partial B_k \rbrace.
$$
By the definition of $\mathcal{A}$, $u(\tilde{T}_{k-1}) < (2k+1)\pi$ and $u(\tilde{T}_{k+1}) > (2k+1)\pi$.  Now by Lemma (\ref{l:B6two}), for any $t \in \varUpsilon_k$, $u(t) \neq (2k+1)\pi$. Let $\sim_k$ be a relation of equivalence on $\varUpsilon_k$ defined with $t_1 \sim_k t_2$ whenever for all $t_3 \in \varUpsilon_k$ such that $t_1 \leq t_3 \leq t_2$, we have that $u(t_1)-(2k+1)\pi$, $u(t_2)-(2k+1)\pi$ and $u(t_3)-(2k+1)\pi$ have the same sign. Let $\tilde{\varUpsilon}_k = \varUpsilon_k / \sim_k$ with the induced topology. As by assumptions, $\varUpsilon_k$ is a closed subset of a compact set, $\varUpsilon_k$ is compact. By definition, continuity of $q$ and compactness of $\partial \mathcal{N}_k$, we see that $\tilde{\varUpsilon}_k$ is totally disconnected and compact, thus finite, and $|\tilde{\varUpsilon}_k| \geq 2$. Consider $|\tilde{\varUpsilon}_k|-1$ intervals $(t_j,t_{j+1}) \in  \varUpsilon_k^c$,  where $t_j,t_{j+1} \in \varUpsilon_k$ and $t_j \not\sim_k t_{j+1}$. We say that $q$ intersects $\mathcal{N}_k$ at $k$ exactly $m$-number of times, if $m$ is the number of such intervals, for which $q$ intersects $\mathcal{N}_k$ at $k$ for some $(t,v) \in \mathcal{N}_k$ such that $t \in (t_j,t_{j+1})$. 

Let $\mathcal{B}_6$ be the set of all $q \in \mathcal{B}_5$, such that for each $k \in \mathbb{Z}$, $q$ intersects $\mathcal{N}_k$ at $k$ odd number of times (i.e. that $m$ in the definition of the number of intersections is odd).

\begin{lemma} \label{l:intersect}
	The set $\mathcal{B}_6$ is $\mathcal{A}$-relatively $\xi$-invariant, and $q^0 \in \mathcal{B}_6$.
\end{lemma}

The proof relies on somewhat subtle topological considerations, and is postponed to the Appendix D. We note that the proof uses in a fundamental way the existence of a continuous semiflow which solves (\ref{r:grad}). 

\begin{proof}[Proof of Proposition \ref{p:invariant}] Assume that $L$ satisfies (\ref{r:Lcond}) and (\ref{r:Lcond2}), and that $M$ is given as in (\ref{r:Mcond}). Let $\tilde{\mathcal{B}}=\mathcal{B}_6$ and $\mathcal{B} = \tilde{\mathcal{B}} \cap \mathcal{A}$. It suffices to show that $q^0 \in \mathcal{B}$, and that the conditions (B1), (B2) hold, as the claim will then hold by Lemma \ref{l:combine}. We have already shown (B1) and $q^0 \in \mathcal{B}$ in Lemma \ref{l:intersect}. We now show (B2). The smoothness requirement in the definition of $\mathcal{A}$ follows from Theorem \ref{t:exist}, (iv). We now show that (\ref{d:Au}) and (\ref{d:Av}) hold.
	
Let $q \in \mathcal{B}_6$, fix $k \in \mathcal{B}_6	$, and find $(t,v(t)) \in \tilde{\mathcal{N}}_k$ such that $u(t)=(2k+1)\pi$ (this exists by the definition of $\mathcal{B}_6$). By definition of $\tilde{\mathcal{N}}_k$, $\tilde{T}_k$ and $\tilde{V}_k$ we have that $|t-\tilde{T}_k| \leq R$, $|v(t)-\tilde{V}_k| \leq R$. Without loss of generality, let $t \leq \tilde{T}_k$, $v(t) \leq \tilde{T}_k$ (the other cases are analogous). Then by the definition of $\Delta_0$ and (\ref{s:bound}),
\begin{equation*}
 \int_{t}^{\tilde{T}_k} \left( u_t^2 + (v_t-\omega_k)^2 \right) dt \leq \max_{(t_0,v_0)\in \mathbb{R}^2}S_{\omega}(t_0,v_0)
 \leq 8 \sqrt{ \ve(1 + \mu)} \leq 9 \sqrt{\ve}.
\end{equation*}
We thus have
\begin{align*}
 |u(\tilde{T}_k)-(2k+1)\pi| & \leq \int_t^{\tilde{T}_k}|u_t|dt  \leq R^{1/2} \left( \int_{t}^{\tilde{T}_k} u_t^2 dt \right)^{1/2} \leq 3 R^{1/2} \ve^{1/4}, \\
 |v(\tilde{T}_k)-\tilde{V}_k | & \leq R + \int_t^{\tilde{T}_k}|v_t|dt \leq R (1 + |\varpi|) +  R^{1/2} \left( \int_{t}^{\tilde{T}_k} (v_t-\omega_k)^2 dt \right)^{1/2} \\
 & \leq R(1+|\varpi|) + 3 R^{1/2} \ve^{1/4}.
\end{align*}
By the assumptions (\ref{r:Rcond}) and (\ref{r:Mcond}), we thus have $|u(\tilde{T}_k)-(2k+1)\pi| \leq 1/4$, $|v(\tilde{T}_k)-\tilde{V}_k | \leq M/2$. 
Let $q(s_0)=q$, and consider the solution $q(s)$ of (\ref{r:grad}) with the initial condition $q(s_0)$ at $s=s_0$. By (\ref{r:sbound}) and the definition of $M$ in (\ref{r:Mcond}),
we have $|u_s| \ll_{\ve,\varpi} 1$, $|v_s| \ll_{\ve,\varpi} 1$, thus the upper bound on $|u_s(t)|$, $|v_s(t)|$ is independent of the choice of $t=\tilde{T}_k$ and $q=(u,v) \in \mathcal{B}_6$. We conclude that for any $\tilde{\lambda}>0$, there exists $\lambda >0$, independent of the choice of $q(s_0) \in \mathcal{B}_6$, such that for all $s \in [s_0,s_0+\lambda]$, we have $|u(\tilde{T}_k)-(2k+1)\pi| \leq 1/4 + \tilde{\lambda}$, $|v(\tilde{T}_k)-\tilde{V}_k | \leq M/2 + \tilde{\lambda}$. By the definition of $\mathcal{A}$, this proves (B2). We conclude that $\mathcal{B}:=\mathcal{B}_6 \cap \mathcal{A}$ is indeed $\xi$-invariant.

It remains to show that (\ref{r:LLcond}) suffices for (\ref{r:Lcond}) and (\ref{r:Lcond2}) to hold. By definition of $M$ in (\ref{r:Mcond}), $c(\varpi) \ll \varpi^5$, thus (\ref{r:Lcond}) is satisfied. As by definition of $\Delta_1$ in Proposition \ref{p:dissipation}, (\ref{r:Delta0up}) and (A2), we know that $\Delta_1 \leq \Delta_0/2 \leq 4\sqrt{\ve}\mu \leq \sqrt{\ve}$, we deduce  $|\log \mu| / \sqrt{\ve} \ll |\log \sqrt{\ve}| / \sqrt{\ve} \leq |\log \Delta_1|/\Delta_1$, which was required. 
\end{proof}

\vspace{2ex}

\centerline{III: PROOFS OF THE MAIN THEOREMS}

\vspace{2ex}

\section{Proofs of the shadowing Theorem, Theorems \ref{t:main1} and \ref{t:main2}} \label{s:shadow}

We first prove a "classical" shadowing theorem, showing existence of a solution of (\ref{r:EL}) shadowing an arbitrary sequence of tori $\mathbb{T}_{\omega}$, and then Theorems \ref{t:main1} and \ref{t:main2}. Prior to all of it, we establish useful a-priori bounds on the derivatives of $q \in \mathcal{B}$. The constant $c_0$ may change from line to line within the section.

\begin{lemma} Assume $q \in \mathcal{B}$. Then for all $t \in \mathbb{R}$ and some absolute $c_0 >0$,
	\begin{equation}
	|u_t(t)| \leq c_0 \varpi \left(\sqrt{\ve} \wedge \left( \frac{\log \dt}{\dt} \right)^{1/2} \right), \hspace{20pt} |v_t(t) - \omega_{k(t)}| \leq c_0 \varpi \left(\sqrt{\ve} \wedge \left( \frac{\log \dt}{\dt} \right)^{1/2} \right). \label{r:derbounds}
	\end{equation}
	Specifically, (\ref{r:smallset}) holds.
\end{lemma}

\begin{proof}
	We first establish the following fact: assume $w \in H^1([\tau,\tau+4/\sqrt{\ve}])$ and $A >0$ a constant, such that
	\begin{equation}
	||w||^2_{L^2([\tau,\tau+4/\sqrt{\ve}])} \leq A \sqrt{\ve}, \hspace{20pt} ||w_t||^2_{L^2([\tau,\tau+4/\sqrt{\ve}])} \leq A \ve^{3/2}. \label{r:condw}
	\end{equation}
	Then we have $||w||_{L^{\infty}([\tau,\tau+4/\sqrt{\ve}])} \ll A^{1/2} \sqrt{\ve}$ (for some absolute implicit constant). Indeed, by substitution $\tilde{w}(t)=w(\tau+4t /\sqrt{\ve})$, we obtain by direct calculation $||\tilde{w}||_{H^1([0,1])} \ll A^{1/2} \sqrt{\ve}$, which implies  $||\tilde{w}||_{L^{\infty}([0,1])} \ll A^{1/2} \sqrt{\ve}$, thus the claim.
	
	By (\ref{r:One}), for any $\tau \in \mathbb{R}$ we have that $||u_t^0||^2_{{L^2([\tau,\tau+4/\sqrt{\ve}])}} \ll \sqrt{\ve}$ and $||v_t^0-\omega_{k(.)}||^2_{{L^2([\tau,\tau+4/\sqrt{\ve}])}} \ll \sqrt{\ve}$. Inserting that in (\ref{r:defB3u}) and (\ref{r:defB3v}), and in the view of the definition (\ref{r:Mcond}) of $M$ and that $\lambda(\tau)\leq \sqrt{\ve}/4$, we get that for any $\tau \in \mathbb{R}$ the relations (\ref{r:condw}) hold with $ w = u_t$, and also with $w = v_t-\omega_{k(.)}$, with $A = c_0 \varpi^2$, for some absolute $c_0$. This yields the $c_0\varpi \sqrt{\ve}$ upper bound in (\ref{r:derbounds}). The other bound has already been established (by an analogous argument) in (\ref{r:uder}) and (\ref{r:vder}).	
\end{proof}	

\begin{thm} \label{t:shadowing} Let $[\omega^-,\omega^+]$ be such that (S1) holds.
Let $\omega_k$, $k \in \mathbb{Z}$ be a sequence in $[\omega^-,\omega^+]$, and choose arbitrarily small $\delta_k>0$, $k \in \mathbb{Z}$. Then there exists $q=(u,v) \in \mathcal{E}$ and a sequence of times $t_k$ such that for all $k \in \mathbb{Z}$
	\begin{equation}
	|u(t_k) - 2k \pi| < \delta_k, \quad |u_t| < \delta_k, \quad	 |v_t(t_k) - \omega_k|  < \delta_k. \label{r:target}
	\end{equation}
\end{thm}

\begin{proof} 
	We set $L$ large enough so that (\ref{r:LLcond}) holds. Choose $\tilde{L}_k$ so that $\tilde{L}_{k+1}-\tilde{L}_k \geq 4L \: \vee \:  c_{16} \: \varpi^2 |\log \delta_k|^2 / \delta_k^2 + 2\pi$ for some $c_{16}>0$ to be determined later, and then $L_k \geq 4L \: \vee \:  c_{16} \: \varpi^2 |\log \delta_k|^2 / \delta_k^2$.
	Let $q^0(L,(\tilde{L}_k)_{k \in \mathbb{Z}},(\omega_k)_{k \in \mathbb{Z}})$ as in Remark \ref{r:notation}, and let  $\mathcal{B}=\mathcal{B}(L,(\tilde{L}_k)_{k \in \mathbb{Z}},(\omega_k)_{k \in \mathbb{Z}})$ as in Remark \ref{r:notation2}. Then by Proposition \ref{p:invariant}, $\mathcal{B}$ is $\xi$-invariant, non-empty, and by Theorem \ref{t:contains}, there is a $q \in \mathcal{E}$ which is also in the closure of $\mathcal{B}$ in $\mathcal{X}_{\text{loc}}$. Let $t_k = (\tilde{T}_{k-1}+\tilde{T}_k)/2)=\tilde{T}_{k-1}+L_k/2$. By (\ref{r:main1}) (the bound on $u$) and (\ref{r:derbounds}) (the bounds on $u_t$, $v_t$), inserting $\tau=t_k$, thus $||\tau|| = L_k/2$, it follows that we can choose an absolute constant $c_{16}$ large enough so that (\ref{r:target}) holds.
\end{proof}

We prove first Theorem \ref{t:main2} as the construction is simpler and illustrative, and then Theorem \ref{t:main1}. 

\begin{proof}[Proof of Theorem \ref{t:main2}, (i)] 
Let $L \equiv 0 \mod 2 \pi$ be such that $ c_{17} \frac{|\log \mu|}{\sqrt{\ve}} \leq L \leq c_{17} \frac{|\log \mu|}{\sqrt{\ve}} + 2\pi $ for some $c_{17}\geq c_{15}$, i.e. such that (\ref{r:LLcond}) holds, for some absolute $c_{17}$ to be determined later.
	
To prove (i), we will embed the standard Bernoulli shift in $\hat{\mathcal{X}}$ in such a way that the Theorem \ref{t:variational} can be applied by using the construction of $\xi$-invariant sets in Proposition \ref{p:invariant}. Let $(\Omega_0,\mathcal{F}_0,\mu_0,s)$ be the standard Bernoulli shift, where $\Omega_0$ is the set of all $\chi = (\chi_j)_{j \in \mathbb{Z}}$, $\chi_j \in \lbrace 0,1 \rbrace$, $\mathcal{F}_0$ is the $\sigma$-algebra on $\Omega_0$ induced by finite cylinders, $\mu_0$ is the product measure, where $0$ and $1$ have the same probability $1/2$, and $s : \Omega_0 \rightarrow \Omega_0 $ is the right shift. The mapping $\iota_0 : \Omega_0 \rightarrow \mathcal{X}$ and the induced map $\hat{\iota}_0 = \iota \circ \iota_0: \Omega_0 \rightarrow \hat{\mathcal{X}}$ is defined as follows: we associate to each $\chi \in \Omega_0$ a sequence $\omega_k(\chi)$ and $\tilde{L}_k(\chi)$, $k \in \mathbb{Z}$. We then define $q_{\chi}=q^0(L,(\tilde{L}_k(\chi))_{k \in \mathbb{Z}},(\omega_k(\chi))_{k \in \mathbb{Z}})$ as in Remark \ref{r:notation}, and the $\xi$-invariant sets $\mathcal{B}_{\chi}$ containing $q_{\chi}$, $\mathcal{B}_{\chi}=\mathcal{B}(L,(\tilde{L}_k(\chi))_{k \in \mathbb{Z}},(\omega_k(\chi))_{k \in \mathbb{Z}})$ as in Remark \ref{r:notation2}. Let $\hat{q}_{\chi}$, $\hat{\mathcal{B}}_{\chi}$ be their embeddings in the quotient set $\hat{\mathcal{X}}$.

Choose arbitrary $\omega_k = \omega \in [\omega^-,\omega^+]$ fixed for all $k \in \mathbb{Z}$. We set $\tilde{L}_k$ so that $q_{\chi}$ has a "jump" at $T=4nL$ if and only if $\chi_n = 0$. More precisely, let $\tilde{\Omega}_0$ be the subset of $\Omega_0$ of $\chi$ with infinitely many $0$ (clearly measurable and $\mu_0(\tilde{\Omega}_0)=1$), and for $\chi \in \tilde{\Omega}_0$, let $k_j(\chi)$ be the increasing sequence of integers of "positions" of zeros in $\chi$, uniquely defined by the requirement $k_0(\chi) \leq 0$, $k_1(\chi) \geq 1$. We now set $\tilde{L}_j(\chi) = 4L \cdot k_j(\chi)$, and complete the definition of $\hat{q}_{\chi}$, $\mathcal{B}_{\chi}$. By the definition of the induced localized topology on $\hat{\mathcal{X}}$, $\hat{\iota_0} : \tilde{\Omega}_0 \rightarrow \hat{\mathcal{X}}$ is continuous (assuming the product topology on $\Omega_0$, $\tilde{\Omega}_0$), thus measurable. Also note that by construction, $\hat{q}_{s(\chi)}=\hat{S}^{4L/2\pi}(\hat{q}_{\chi})$, $\hat{\mathcal{B}}_{s(\chi)}=\hat{S}^{4L/2\pi}(\hat{\mathcal{B}}_{\chi})$. Let $\tilde{\mu}=(\hat{\iota}_0)^*\mu_0$ be the pulled measure, and let $\tilde{\mathcal{M}_1} = \hat{\iota}_0(\tilde{\Omega}_0)$. 

As the entropy $h_{\mu_0}(s)=\log 2$ \cite{Walters}, by construction we have $h_{\tilde{\mu}}(\hat{S}^{4L/2\pi}) = \log 2$. Finally we define a measure $\mu$ "to be shadowed" by 
\begin{equation*}
\mu = \sum_{n=0}^{4L/2\pi -1} (\hat{S}^n)^*\tilde{\mu}, \hspace{30pt} \mathcal{M}_1=\bigcup_{n=0}^{4L/2\pi -1} \hat{S}^n(\tilde{\mathcal{M}}_1).
\end{equation*}
By construction, $\mu$ is $\hat{S}$-invariant. By \cite{Walters}, Theorem 4.13, (i), $h_{\mu}(\hat{S})=2\pi \log 2 / (4L)$.

To apply Theorem \ref{t:variational} and establish existence of a shadowing measure $\nu$, we need to construct a $\sigma$-subalgebra $\mathcal{G}$. First we define the sets $\mathcal{D}_i \subset \hat{\mathcal{X}}$, $i \in \mathcal{I}$, where $\mathcal{I}=\lbrace (j,k,n), \: j \in \lbrace 0,1 \rbrace, \: k \in \lbrace 0,1,...,4L/2\pi-1 \rbrace, \: n \in \mathbb{Z}\rbrace$,  as the set of all $q = (u,v) \in \hat{\mathcal{X}}$ satisfying the following conditions
\begin{align}
u(4nL+2k\pi) & \in \begin{cases}
[\pi -1, \pi+1] \mod 2\pi & j=0,\\
 [ -1, 1] \mod 2\pi \:  & j=1.
\end{cases}  \label{d:j0j1}
\end{align}
Now let $\mathcal{G}$ be the $\sigma$-subalgebra generated by $\mathcal{D}_i$ $i\in \mathcal{I}$.
As each $q \in \mathcal{M}_1$ can be represented as $\hat{S}^{k_0}q_{\chi}$ for some $k_0=0,...,4L/2\pi-1$, $\chi \in \tilde{\Omega}_0$, we define
\begin{equation*}
\mathcal{D}_q = \bigcap_{n=-\infty}^{\infty} \left\lbrace \mathcal{D}_{\chi_n,k_0,n} \bigcap_{\substack{ k = 0 \\ k \neq k_0}}^{4L/2\pi-1} \mathcal{D}_{0,k,n}   \right\rbrace,
\hspace{30pt} \mathcal{B}_q = \hat{S}^k \mathcal{B}_{\chi}.
\end{equation*}
By construction we can choose an absolute constant $c_{17}\geq c_{15}$ large enough so that $\mathcal{B}_q \subset \mathcal{D}_q$ for all $q \in \mathcal{M}_1$. The reason is as follows: (i) in the case $j=0$ in (\ref{d:j0j1}), because of the relations (\ref{d:Au}), (\ref{r:defb1u}) and the properties of $z^+$, $z^-$ (as there is some $\tilde{T}_m$ such that $|\tilde{T}_m-(4nL+2k\pi)| \leq 2\pi$), and (ii) in the case $j=1$, because of (\ref{r:defb1u}) and the exponentially fast decay of $z^+$, $z^-$ towards $0 \mod 2\pi$. Thus by Proposition \ref{p:invariant}, for all $q \in \mathcal{M}_1$ and all $s \geq 0$, $\hat{\xi}^s(q) \in  \mathcal{D}_q$. By construction we have that if $q \in \mathcal{D}\cap \mathcal{M}_1$ and $\mathcal{D} \in \mathcal{G}$, then $\mathcal{D}_q \subset \mathcal{D}$. This completes the proof of the condition (M2). By construction, (M1), (M3) hold. To verify (M4), it suffices to note that for fixed $k,n$, $\mathcal{D}_{0,k,n}$ and $\mathcal{D}_{1,k,n}$ are disjoint and cover the entire $\mathcal{M}_1$.
As this completes the proof of conditions (M1)-(M4), we can apply Theorem \ref{t:variational} and obtain a $\nu \in \mathcal{M}(\hat{\mathcal{E}})$ which shadows $\mu$. By Lemma \ref{l:entropy} we get $h_{\nu}(\hat{S})\geq h_{\mu}(\hat{S}) \sim 1/L$. The claim now follows from (\ref{r:LLcond}) and the variational principle for the topological and metric entropy \cite{Walters}.
\end{proof}

\begin{proof}[Proof of Theorem \ref{t:main2}, (ii)] 
	Choose $L \equiv 0 \mod 2\pi$ satisfying $ c_{15} \frac{|\log \mu|}{\sqrt{\ve}} \leq L \leq c_{15} \frac{|\log \mu|}{\sqrt{\ve}} + 2\pi $, i.e. such that (\ref{r:LLcond}) holds.	
	Choose an increasing sequence $\omega_1 = \omega^-$, ..., $\omega_n=\omega^+$, such that $\omega_k-\omega_{k-1} \leq \Delta_0 / (4 c_4 (R \vee \mu) \varpi)$ for $k=2,...,n$ (as required by (\ref{r:assumec})), with equality for all except possibly $k=n$. Choose $\omega_k$ for $k \notin \lbrace 1,...,n \rbrace$ in an arbitrary way as long as $\omega_k \in [\omega^-,\omega^+]$ and as long as (\ref{r:assumec}) holds. Let $\tilde{L}_{k+1}-\tilde{L}_k = 4L + 2\pi$ for all $k$ except $k=1$ and $k=n$ for which we set
	\begin{equation}
	 \tilde{L}_{k+1}-\tilde{L}_k = c_0 \varpi^2 \frac{|\log \delta |^2}{\delta^2} + 4L + 2\pi, \label{r:smallLk}
	\end{equation} 
	for $c_0$ large enough to be chosen later. By Proposition \ref{p:invariant}, $\mathcal{B}=\mathcal{B}(L,(\tilde{L}_k(\chi))_{k \in \mathbb{Z}},(\omega_k(\chi))_{k \in \mathbb{Z}})$ is $\xi$-invariant, and by Theorem \ref{t:contains}, there is a $q \in \mathcal{E}$ which lies in the closure of $\mathcal{B}$ in $\mathcal{X}_{\text{loc}}$.
	
	Let $t^-= (\tilde{T}_1+\tilde{T}_2)/2 = \tilde{T}_1 + L_1/2$, and $t^+ = (\tilde{T}_n+\tilde{T}_{n+1})/2 = \tilde{T}_n + L_n/2$. Now by (\ref{r:derbounds}), as $||t^-||=L_1/2$, $||t^+||=L_n/2$, as $L_1,L_n \geq c_0 \varpi^2 \frac{|\log \delta |^2}{\delta^2} $, and finally as $\omega_{k(t^-)}=\omega_1=\omega^-$ and $\omega_{k(t^+)}=\omega_n=\omega^+$; we can choose $c_0$ in (\ref{r:smallLk}) to be a a large enough absolute constant such that $|v_t(t^-)-\omega^-| \leq \delta$, $|v_t(t^+)-\omega^+| \leq \delta$. Now as $n \sim \varpi (R \vee \mu) (\omega^+-\omega^-) / \Delta_0$, and by (\ref{r:LLcond}) and (\ref{r:smallLk}),
	\begin{align*}
	|t^+-t^-| & \leq \sum_{k=1}^n L_k \ll  \varpi^2 \frac{|\log \delta |^2}{\delta^2} + n L 
	 \ll \varpi^2 \frac{|\log \delta |^2}{\delta^2} + \varpi^6 \frac{|\log \Delta_1| (R \vee \mu)}{\Delta_0 \Delta_1} (\omega^+-\omega^-) \\
	& \ll \varpi^6 \tilde{c}(\delta) \frac{|\log \Delta_1| (R \vee \mu)}{\Delta_0 \Delta_1} (\omega^+-\omega^-),
	\end{align*}
	where we can set $\tilde{c}(\delta)=|\log \delta |^2/\delta^2$ as $nL \gg \varpi^2$. The proof is completed.
\end{proof}

\begin{proof}[Proof of Theorem \ref{t:main1}] We first complete the proof in the case when the entire $\mathcal{O} \subseteq [\tilde{\omega}^-,\tilde{\omega}^+] \subseteq [\omega^-,\omega^+]$, such that $\tilde{\omega}^+-\tilde{\omega}^- \leq \Delta_0 / (4 c_4 (R \vee \mu) \cdot \varpi)$ (i.e. we can by Proposition \ref{p:hetero} connect any two $\omega$, $\tilde{\omega}$ in $\mathcal{O}$ with a heteroclinic orbit with only one "jump"). The only modification as compared to the proof of Theorem \ref{t:main2} is now the choice of the sequence $\omega_k$. Let $\lbrace \tilde{\omega}_j, \: j \in \mathcal{P} \rbrace$, $\mathcal{P}$ finite or countable, be a dense subset of $\mathcal{O}$. Let $p : \mathbb{N} \rightarrow \mathcal{P}$ be any function such that for each $j \in \mathcal{P}$, $p^{-1}(j)$ is infinite (constructed e.g. by a diagonalization procedure). By keeping the notation as in the proof of Theorem \ref{t:main2}, (i), for a given $\chi$ we define $\omega_j = \tilde{\omega}_{p(k_{j+1}-k_j)}$ (where by definition $k_{j+1}-k_j-1$ is the number of consecutive "ones" between the zeroes at positions $j$ and $j+1$ in a chosen $\chi \in \tilde{\Omega}_0$). The rest of the construction is analogous to the proof of Theorem \ref{t:main1}, by which we obtain $\nu \in \mathcal{M}(\hat{\mathcal{E}})$, which is supported on the closure of the union of $\xi$-invariant sets $\mathcal{B}(L,(\tilde{L}_k)_{k \in \mathbb{Z}},(\omega_k(\chi))_{k \in \mathbb{Z}})$ (considered as subsets of $\hat{\mathcal{X}}$), for $\chi \in \Omega_0$. It is easy to check by construction and (\ref{r:derbounds}) that $\text{supp }\hat{\pi}^*\nu$ intersects $\mathbb{T}_{\omega}$ for every $\omega \in \mathcal{O}$ (which implies the first part of (\ref{r:shadow})), and that for every $\omega \notin \mathcal{O}$, $\text{supp }\hat{\pi}^*\nu \cap \mathbb{T}_{\omega} = \emptyset$.

Consider now the general case, and let $n$ be an integer such that 
$$ n \geq \frac{4 c_4 (R \vee \mu) \cdot \varpi}{\Delta_0} (\omega^+-\omega^-)
$$
(i.e. $n$ is the minimal number of "jumps" required by the Proposition \ref{p:invariant} and (\ref{r:assumec}) to cross the entire $[\omega^-,\omega^+]$). Consider the subshift of finite type $(\Omega_1,\mathcal{F}_1,\mu_1,s)$ ($(\Omega_1,\mathcal{F}_1)$ a subspace of $(\Omega_0,\mathcal{F}_0)$), $\mu_1$ the $s$-invariant probability measure, where we "allow" only sequences with no less than $n$ consecutive zeros. The only change in the construction above is that we associate to each consecutive sequence of $m \geq n$ zeros in a chosen $\chi \in \Omega_1$ a sequence $\omega_1$, $\omega_2$, ..., $\omega_m$ in an arbitrary, $s$-invariant way, such that (\ref{r:assumec}) holds for any $\omega=\omega_k$, $\tilde{\omega}=\omega_{k+1}$, $k \in \mathbb{Z}$ (this is possible by the choice of $n$). The rest of the proof is analogous.	
\end{proof}

\begin{remark} We could have sharpened the definition of the subalgebra $\mathcal{G}$ in the proof of Theorem \ref{t:main1}, to be able to use the notion of the conditional support from Section \ref{s:measure}. More specifically, we could obtain that $\hat{\pi}(\text{supp}(\nu | \mathcal{G}))$ intersects all $\mathbb{T}_{\omega}$, $\omega \in \mathcal{O}$, and that $\cup_{\omega \in \mathbb{R}-\mathcal{O}} \subset \hat{\pi}(\text{supp}^c(\nu | \mathcal{G}))$. This would, however, unnecessarily complicate the proof.	
\end{remark}

\section{Proof of Theorem \ref{t:main3}} \label{s:Theorem3}

In this section we first state an analogue of Lemma \ref{l:subsuper} which will enable us more precise control for small $\mu$, then evaluate the error term when approximating $S_{\omega}$ with $M_{\omega}$, and finally we prove Theorem \ref{t:main3}.

\begin{lemma} \label{l:subsuper2}
	There exist a constant $\tilde{T}>0$ and functions $\tilde{z}^- : [-\tilde{T}, 0] \rightarrow \mathbb{R}$ and $\tilde{z}^+ : [0,\tilde{T}] \rightarrow \mathbb{R}$, depending only on $\ve$, $\mu$, satisfying for all $t$ in the domain of definition:
	
	(i) $\tilde{z}^-$, $\tilde{z}^+$ are continuous and $C^2$ in the interior of the domain,
	
	(ii) $\tilde{z}^-(0)=\tilde{z}^+(0)=\pi$, $\tilde{z}^-(-\tilde{T})=-\sqrt{\ve \mu}$, $\tilde{z}^+(\tilde{T})=2\pi+\sqrt{\ve \mu}$,
	
	(iii) $\tilde{z}^-$ is a strict stationary sub-solution on $(-\tilde{T}, 0)$ of (\ref{r:gradu}), and $\tilde{z}^+$ is a strict stationary super-solution (\ref{r:gradu}) on $(0,\tilde{T})$. Furthermore, for any constant $T \in \mathbb{R}$, $\tilde{z}^-(t+T)$ and $\tilde{z}^+(t+T)$ are strict stationary sub-, resp. super-solutions in the interior of their domain,
	
	(iv) $\tilde{z}^-$, $\tilde{z}^+$ are strictly increasing,
		
	(v) There exists an absolute constant $c_{18} > 0$ such that
	\begin{align}
	|\tilde{z}^-(t)-u^{(\ve)}(t)| \leq c_{18} \sqrt{\ve \mu}, \: t \in [-\tilde{T}, 0], \hspace{10pt} |\tilde{z}^+(t)-u^{(\ve)}(t)| \leq c_{18} \sqrt{\ve \mu}, \: t \in [0,\tilde{T}],   \label{r:z--prec}
	\end{align}
	where $u^{(\ve)}(t) = 4 \arctg e^{\sqrt{\ve}t}$ is the separatrix solution in the case $\mu= 0$.
\end{lemma}

The proof is a relatively straightforward modification of the proof of Lemma \ref{l:subsuper}, thus also done in the Appendix B.

\begin{lemma} \label{l:approximate}
 (i) There exists an absolute constant $c_{19}>0$ such that if $q=(u,v)$ is a two-sided minimizer at $(\omega,t_0,v_0) \in \mathbb{R}^3$, then for all $t \in \mathbb{R}$,
  \begin{equation}
     |u(t)-u^{(\ve)}(t-t_0)| \leq c_{19} \sqrt{\ve \mu}, \label{r:whereu}
  \end{equation}  
   
  (ii) For all $(t_0,v_0), (t_1,v_1) \in \mathbb{R}^2$, and all $\omega \in \mathbb{R}$,
   \begin{equation}
   |S_{\omega}(t_1,v_1)-S_{\omega}(t_0,v_0)-\mu(M_{\omega}(t_1,v_1)-M_{\omega}(t_0,v_0))| = O_f(\ve \mu^{3/2}). \label{r:3over2}
	\end{equation}
\end{lemma}

\begin{proof}
We first note that analogously to the proof of Proposition \ref{p:stable}, (ii) we can show that for all $t \in [t_0-\tilde{T},t_0]$, we have $u(t) \geq \tilde{z}^-(t-t_0)$, thus by Proposition \ref{p:stable}, (ii), $\tilde{z}^-(t-t_0) \leq u(t) \leq z^-(t-t_0)$. We deduce that 
$$|u(t)-u^{(\ve)}(t-t_0)| \leq |\tilde{z}^-(t-t_0)-u^{(\ve)}(t-t_0)| + |z^-(t-t_0)-u^{(\ve)}(t-t_0)|
$$
which implies (\ref{r:whereu}) by and (\ref{r:z-prec}) and (\ref{r:z--prec}). The case $t < -\tilde{T}+t_0$ follows from $\tilde{z}^-(\tilde{T})=-\sqrt{\ve \mu}$ and $0 \leq u(t) \leq z^-(t-t_0)$, $\tilde{z}^-$ increasing. The proof for $t \geq t_0$ is analogous, which completes (i).

Let $q=(u,v) \in H^1_{\text{loc}}(\mathbb{R})^2$ be a two-sided minimizer at $(\omega,t_0,v_0)$, and let 
$\tilde{q}(t)=(u(t+t_0-t_1),v(t+t_0-t_1)+v_1-v_0)$ (i.e. we translate $q$ in $t,v$ so that $\tilde{q}(t_1)=(\pi,v_1)$). As by definition, $\int_{-\infty}^{\infty}L_{\omega}(\tilde{q},\tilde{q}_t,t)dt \geq S_{\omega}(t_1,v_1)$, we by straightforward calculation have
\begin{align}
S_{\omega}(t_1,v_1)-S_{\omega}(t_0,v_0)-\mu(M_{\omega}(t_1,v_1)-M_{\omega}(t_0,v_0)) & 
\leq 2 \ve \mu \left| \int_{-\infty}^{\infty}(1-\cos u)f(u,v,t)  dt \right. \notag \\ & \hspace{-20pt} \left. - \int_{-\infty}^{\infty} \left( 1-\cos (u^{\ve}(t-t_0))\right) f(u^{\ve}(t-t_0),v_0+\omega(t-t_0),t)dt                                  \right|. \label{r:melcalc}
\end{align}
By the Mean Value Theorem, (\ref{r:uhone}), by the left-sides of (\ref{r:vhleft}), (\ref{r:vhright}) (which also hold for two-sided minimizers which are not necessarily homoclinics) and (\ref{r:whereu}), the absolute value of the difference of the integrands in (\ref{r:melcalc}) is $\ll_f e^{-\frac{1}{2}\sqrt{\ve}|t-t_0|}\sqrt{\ve \mu}$, thus (\ref{r:melcalc}) is $\ll_f \ve \mu^{3/2}$. The other inequality in (\ref{r:melcalc}) is obtained analogously, by starting with a two-sided minimizer at $(\omega,t_1,v_1)$.
\end{proof}

\begin{proof}[Proof of Theorem \ref{t:main3}] This follows from (\ref{r:3over2}), as we choose $\mu_0$ to be small enough, so that for $\mu \leq \mu_0$, the $O_f(\ve \mu^{3/2})$ term is $\leq \mu \tilde{\Delta}_0$. Then for all $0 < \mu \leq \mu_0$, (S1) holds with $\Delta_0 = \mu \tilde{\Delta}_0$.
\end{proof}

\vspace{2ex}

\centerline{IV: APPENDICES}

\vspace{2ex}

\section{Appendix A: The function spaces and existence of solutions}

This section is dedicated to the proof of Theorem \ref{t:exist}. We first recall definition of the required function spaces (see \cite{Gallay:01} and references therein for details), then prove the theorem in four separate lemmas: on local existence and uniqueness, global existence, continuous dependence on initial conditions and regularity of solutions.

Let $\varphi^y(q)(t)=q(t+y)$ be the translation, $y \in \mathbb{R}$. The uniformly local norms and spaces are given with:
\begin{eqnarray}
||q||_{L^2_{\text{ul}}(\mathbb{R})^N}&=& \sup_{y \in \mathbb{R}} \left( \int_{\mathbb{R}} e^{-|t+y|}q(t)^2 \right) ^{1/2}, \nonumber \\
L^2_{\text{ul}}(\mathbb{R})^N &=& \left\lbrace  q \in L^2_{\text{loc}}(\mathbb{R})^N , 
\: ||q||_{L^2_{\text{ul}}(\mathbb{R})^N} < \infty, \: 
\lim_{y \rightarrow 0}||\varphi^yq-q||_{L^2_{\text{ul}}(\mathbb{R})^N}=0  \right\rbrace , \nonumber \\
H^k_{\text{ul}}(\mathbb{R})^N &=& \left\lbrace u\in L^2_{\text{ul}}(\mathbb{R})^N \: | \: \: \partial^j_tu \in L^2_{\text{ul}}(\mathbb{R})^N \text{ for all } j\leq k \right\rbrace ,\nonumber \\
||q||_{H^k_{\text{ul}}(\mathbb{R})^N} &=&\left( \sum_{j=0}^k || \partial^j_t q||^2_{L^2_{\text{ul}}(\mathbb{R})^N} \right)^{1/2}. \nonumber
\end{eqnarray} 
\begin{remark} \label{r:UL}
For our purposes it suffices to note that if $q : \mathbb{R}^N \rightarrow \mathbb{R}$ is Lipschitz continuous and bounded in $L^{\infty}(\mathbb{R})^N$, then it is in $L^2_{\text{ul}}(\mathbb{R})^N$. Specifically, the Lipschitz continuity implies that $ \lim_{y \rightarrow 0}||\varphi^yq-q||_{L^2_{\text{ul}}(\mathbb{R})^N}=0$ holds, i.e. that $\varphi^y(q)$ is continuous in $y \in \mathbb{R}$.
\end{remark}

Denote by $A:H^{2}_{\text{loc}}(\mathbb{R})^N \rightarrow L^{2}_{\text{loc}}(\mathbb{R})^N$ the linear operator $Aq=-q_{tt}$. The system (\ref{r:Grad}) can then be written in a compact form 
\begin{equation}
 q_s=-Aq+F(q), \label{r:compactF0}
\end{equation} 
$F(q)(t)=\partial V(q(t),t)/\partial q$. Fix an initial condition $q^0 \in \mathcal{X}$. We substitute $\tilde{q}=q-q^0$, and consider
	\begin{equation} \label{r:compactF}
	\begin{split}
	\tilde{q}_s&=-\tilde{A}\tilde{q}+\tilde{F}(\tilde{q}), \\
	\tilde{q}(0)&=0, 
	\end{split} 
	\end{equation}
where $\tilde{F}(\tilde{q})=F(\tilde{q}+q^0)-Aq^0$, and $\tilde{A}$ is the restriction of $A$ to $H^2_{\text{ul}}(\mathbb{R})^N$, $\tilde{A}: H^2_{\text{ul}}(\mathbb{R})^N \rightarrow L^2_{\text{ul}}(\mathbb{R})^N$. It is straightforward to check that $\tilde{F}$ is well-defined as $\tilde{F}:H^1_{\text{ul}}(\mathbb{R})^N \rightarrow L^2_{\text{ul}}(\mathbb{R})^N$, and uniformly Lipschitz (for a fixed $q^0$) on the entire domain. Without loss of generality, we assume that the initial condition is given at $s_0=0$, and fix $q^0$ throughout the proofs. 

\begin{lemma} \label{l:local} 
For some $S>0$ small enough, there exists an unique solution $\tilde{q}$ of (\ref{r:compactF}) on $(0,S)$,
$$\tilde{q} \in C^0(\left[ 0,S \right),H^1_{\text{ul}}(\mathbb{R})^N) \cap C^1(\left( 0,S \right) ,H^1_{\text{ul}}(\mathbb{R})^N) \cap C^0(\left( 0,S \right) ,H^2_{\text{ul}}(\mathbb{R})^N).$$ 			
\end{lemma}

\begin{proof} We follow \cite[Chap. 3]{Henry:81}. First note that $\tilde{A}$ generates an analytic semigroup $S(s)=\exp(-s\tilde{A})$ of bounded linear operators in $L^2_{\text{ul}}(\mathbb{R})^N$, which can for example be verified by using the explicit expression of the heat kernel (see e.g. \cite{Gallay:01} for details). As $\tilde{F}$ is uniformly, thus locally Lipschitz in $q$, and constant in $s$, the claim follows from \cite[Theorem 3.3.3]{Henry:81} with (using the notation from \cite{Henry:81}) $X=L^2_{\text{ul}}(\mathbb{R})^N$, $D(\tilde{A})=H^2_{\text{ul}}(\mathbb{R})^N$, $\alpha=1/2$,  $X^{1/2}=H^1_{\text{ul}}(\mathbb{R})^N$.
\end{proof}

We now require the well-known fact \cite{Henry:81} that $\tilde{q}$ is a ("classical") solution of (\ref{r:compactF}) on $(0,S)$ if and only if it is a mild solution, i.e. if for any $0<s_1 \leq S$, $\tilde{q}$ satisfies the integral equation
\begin{equation}
\tilde{q}(s_1)  =  \int_0^{s_1} e^{-\tilde{A}(s_1-s)}\tilde{F}(\tilde{q}(s)) ds. \label{r:milder}
\end{equation}

\begin{lemma} \label{l:global}
	The solution $\tilde{q}$ of (\ref{r:compactF}) exists on $(0,\infty)$.
\end{lemma}

\begin{proof}
	As $\tilde{F}(\tilde{q})$ is uniformly bounded in $L^2_{\text{ul}}(\mathbb{R})^N$ by some constant $A$, we have that if the solution of (\ref{r:compactF}) exists on $(0,S)$, then for all $s \in (0,S)$, $||\tilde{q}(s)||_{H^2_{\text{ul}}(\mathbb{R})^N}\leq e^S \cdot A$. The claim now follows from \cite[Corollary 3.3.5]{Henry:81}, as "blow-up" is not possible.
\end{proof}

\begin{lemma} \label{l:continuous} The solution of (\ref{r:compactF0}) is continuous with respect to initial conditions in both $\mathcal{X}_{\text{ul}}$ and $\mathcal{X}_{\text{loc}}$.
\end{lemma}

\begin{proof}
 We substitute back $q$ instead of $\tilde{q}$ in (\ref{r:milder}), and obtain
 \begin{equation}
 q(s_1)  = q^0+ \int_0^{s_1} e^{-\tilde{A}(s_1-s)}\left( F(q(s))-Aq^0\right) ds. \label{r:milder2}
 \end{equation}
Consider a sequence of initial conditions $q^{0,(n)}$ converging in either  $\mathcal{X}_{\text{ul}}$ or $\mathcal{X}_{\text{loc}}$ norm to $q^0$, and consider associated solutions $q^{(n)}(s)$, $q(s)$. Continuity is then shown by bounding the difference of the right-hand sides of (\ref{r:milder2}), for $q^{(n)}(s_1)$ and $q(s_1)$ for $n$ large enough and $s_1>0$ small enough, in either $\mathcal{X}_{\text{ul}}$- or $\mathcal{X}_{\text{loc}}$-norm (see also \cite[Corollary 3.4.1]{Henry:81} for details).
\end{proof}

\begin{lemma} \label{l:smooth}
If $V \in H^k(\mathbb{T}^{N+1})$, $k \geq 2$, then for all $s > 0$, if $q$ is the solution of (\ref{r:compactF0}), then $q(s) \in H^k_{\text{ul}}(\mathbb{R})^N$.
\end{lemma}

\begin{proof} We prove it inductively. In the proofs of Lemmas \ref{l:local}, \ref{l:global}, we already established the case $k=2$. Consider the case $k\geq 3$, and assume the claim holds for a given $k-1$. Let $r = d^{k-2}q/dt^{k-2}$. By the inductive assumption, $r(s)\in H^1_{\text{ul}}(\mathbb{R})^N$ for all $s>0$. For arbitrarily small $\delta > 0$, consider the system of equations
\begin{equation} \label{r:regular}
\begin{split}
r_s &=-\tilde{A}r+F^{(k-2)}(s),  \\
r(\delta) &= \frac{d^{k-2}}{dt^{k-2}}q(\delta) 
\end{split}
\end{equation}	
where 
$
F^{(k-2)}(s)= \frac{d^{k-2}}{dt^{k-2}} \frac{\partial}{\partial q}V(q(s,.),.)
$
is a fixed function. One can verify by using the inductive assumption $q \in H^{k-1}_{\text{ul}}(\mathbb{R})^N$, the assumed regularity of $V$ and the embedding properties of the uniformly local spaces \cite{Gallay:01}, that $F^{(k-2)}(s) \in L^2_{\text{ul}}(\mathbb{R})^N$ for all $s \in (0,\infty)$, and that it is uniformly bounded in $L^2_{\text{ul}}(\mathbb{R})^N$ on $(\delta,S]$ for any $S > \delta$. Now by repeating the argument of existence and uniqueness of solutions as in Lemma \ref{l:local}, we deduce that for any $s \in (\delta,\infty)$, $r(s)$ is a solution of (\ref{r:regular}), thus $r(s) \in H^2_{\text{ul}}(\mathbb{R})^N$ and $q(s) \in H^k_{\text{ul}}(\mathbb{R})^N$. As $\delta >0$ is arbitrarily small, the claim is proved. 
\end{proof}

Theorem \ref{t:exist} now follows from Lemmas \ref{l:local}, \ref{l:global}, \ref{l:continuous} and \ref{l:smooth}. 

We close the section with a frequently required result that the solutions of (\ref{r:EL}) we construct are indeed in $\mathcal{X}$.

\begin{lemma} \label{l:Xcontain}
   If $q$ is a solution of (\ref{r:ELN}) such that either $q_t \in L^{\infty}(\mathbb{R})^N$ or $q_t \in L^2_{\text{ul}}(\mathbb{R})^N$, then $q \in \mathcal{X}$ (and by definition, $q \in \mathcal{E}$).
\end{lemma}

\begin{proof}
First we note that as $\partial V(q,t) / \partial q$ is $C^1$ and periodic, $q_{tt}$ is continuous and in $L^{\infty}(\mathbb{R})^N$. We easily deduce that also in the case $q_t \in L^2_{\text{ul}}(\mathbb{R})^N$, we have $q_t \in L^{\infty}(\mathbb{R})^N$. The Mean Value Theorem shows that the Lipschitz constant for $q_{tt}$ is 
bounded with 
$||\partial^2 V / \partial q^2 ||_{L^{\infty}(\mathbb{R})^N} ||q_t||_{L^{\infty}(\mathbb{R})^N} + ||\partial^2 V / \partial q \partial t ||_{L^{\infty}(\mathbb{R})^N}$, which is finite as $V$ is $C^2$ and periodic in all the variables. By Remark \ref{r:UL}, we now have $q_t \in  L^2_{\text{ul}}(\mathbb{R})^N$, $q_{tt} \in  L^2_{\text{ul}}(\mathbb{R})^N$, thus $q \in \mathcal{X}$.
\end{proof}

\section{Appendix B: A-priori bounds on one-sided minimizers}

This Appendix is dedicated to the proofs of Lemma \ref{l:subsuper} and Proposition \ref{p:stable} in Section \ref{s:five} and Lemma \ref{l:subsuper2} in Section \ref{s:Theorem3}, i.e. the construction of one-sided minimizers in Section \ref{s:five} and calculation of a-priori bounds.

\subsection{Proofs from Section \ref{s:five}}

\begin{lemma} \label{l:construction}
	Define $\tilde{w}(t) = 4 \arctg \exp (\sqrt {\ve(1-2\mu^{1/2})} \: t)$, and
	\begin{equation*}
	w(t) =
	\begin{cases}
	\tilde{w}(t) & t \geq t_1 \\
	\tilde{w}(t) + \ve^{3/2}\mu^{1/2}(t_1-t)^3 & t \leq t_1,
	\end{cases}
	\end{equation*}
	where $t_1$ is chosen so that $\tilde{w}(t)=\pi+2\mu^{1/2}$. Then
	
	(i) $0 < t_1 \leq 2 \sqrt{ \mu / \ve}$,
	
	(ii) $w$ is is $C^2$, for all  $t \in [-1/(2\sqrt{\ve}),\infty)$, $0 < w_t < 2 \sqrt{\ve}$, and for all $t \in [-1/(2\sqrt{\ve}),1/(2\sqrt{\ve})]$,
	\begin{equation}
	\sqrt{\ve}/2  \leq w_t, \label{r:boundder}
	\end{equation}
	
	(iii) There is a unique $t_0$ satisfying $w(t_0)=\pi$ on $[-1/\sqrt{\ve},\infty)$, and it satisfies $0 > t_0 \geq - \mu^{1/2} $,
	
	(iv) $w$ is a strict stationary sub-solution of (\ref{r:gradu}) on $[-1/\sqrt{\ve},\infty)$. Furthermore, for any $T \geq 0$, $w(t-T)$ is a strict stationary sub-solution of (\ref{r:gradu}) on $[-1/\sqrt{\ve}+T,\infty)$.
\end{lemma}

\begin{proof}
(i) As $\tilde{w}$ is strictly increasing and $\tilde{w}(0)=\pi$, clearly $t_1$ is unique and $t_1 > 0$. To show $t_1 \leq 2 \sqrt{ \mu / \ve}$, it suffices to show that $\tilde{w}(t^*) > \pi + 2\mu^{1/2}$ for $t^*=  2 \sqrt{ \mu / \ve}$. This is straightforward by the mean-value theorem and by noting that the derivative of $4 \arctg e^t$ is $\geq 9/5$ on $[0,\sqrt{2}/4]$. 

(ii) This follows by elementary calculation, applying (i) and the standing assumption $\sqrt{\mu}\leq 1/4$.
	
(iii) As $w(0)>\pi$, it suffices to show that $\tilde{w}(-\mu^{1/2}/2) > \pi$. For $t^*=-\mu^{1/2}/2$, $|t^*-t_1| \leq \sqrt{\mu} / \sqrt{\ve}$, thus then value of the polynomial in the definition of $w$ is $ \leq 27 \mu^2$. It is elementary to show that $\tilde{w}(-\mu^{1/2}/2) < \pi-27\mu^2$, applying (A2).
	
(iv) It is easy to see that it suffices to show that
$$ 
\tilde{\mathcal{F}}(w):= w_{tt} - \ve \sin w(t) - \ve \mu | \sin w(t) | - \ve \mu (1-\cos w(t)) > 0
$$
for $t \in [-1/\sqrt{\ve},\infty)$. 
Note that $ \tilde{w}_{tt} = (\ve (1 - 2 \mu^{1/2}))\sin \tilde{w}(t)$. For $t \geq t_1$, $\sin \tilde{w}(t)=\sin w(t)$ by definition. For $t \in [-1/\sqrt{\ve},t_1]$, by the strict monotonicity of both $w(t),\tilde{w}(t)$ we see that $ \pi /2 \leq \tilde{w}(t) \leq w(t) \leq 3 \pi /2$, thus $\sin \tilde{w}(t) \geq \sin w(t)$. In both cases we thus get
\begin{align}
\tilde{\mathcal{F}}(w) & = (\ve (1 - 2 \mu^{1/2}))\sin \tilde{w}(t) + 6 \ve^{3/2} \mu^{1/2}((t_1-t) \vee 0)
- \ve \sin w(t) - \ve \mu | \sin w(t) | - \ve \mu (1-\cos w(t)) \notag \\
& \geq 6 \ve^{3/2} \mu^{1/2}((t_1-t) \vee 0) - 2 \ve \mu^{1/2} \sin w(t)- \ve \mu | \sin w(t) | - \ve \mu (1-\cos w(t)). \label{r:deff}
\end{align}
Denote the expression (\ref{r:deff}) by $\mathcal{F}(w(t))$. It suffices to show now that $\mathcal{F}(w(t)) > 0$.

Consider first the case $w(t) \geq \pi + 2 \mu^{1/2}$, which is equivalent to $t \geq t_1$. As
in this case, $\sin w < 0$, and always $\mu < \mu^{1/2}$, we have
\begin{equation}
\mathcal{F}(w) > \ve \mu^{1/2}|\sin w | - \ve \mu (1-\cos w). \label{r:fw1}
\end{equation}
For any $w \in (0,\pi)$, the inequality
\begin{equation}
|\sin w | \geq \frac{1}{2}|w - \pi | ( 1-\cos w ). \label{r:cosin}
\end{equation}
holds. Inserting this and $|w-\pi| \geq 2 \mu^{1/2}$ in (\ref{r:fw1}) we get $\mathcal{F}(w) > 0$.

Let $w \in [\pi,\pi + 2 \mu^{1/2}]$, or equivalently $ t \in [t_0,t_1]$. As $\mathcal{F}(w(t_1))>0$, it suffices to show that $\mathcal{F}(w(t))-\mathcal{F}(w(t_1)) \geq 0$. Calculating we get
\begin{equation}
 \mathcal{F}(w(t)) - \mathcal{F}(w(t_1)) \geq  6 \ve^{3/2} \mu^{1/2}|t-t_1| - ( 2 \ve \mu^{1/2}+ \ve \mu ) |\sin w(t) - \sin w(t_1)| - \ve \mu | \cos w(t) - \cos w(t_1)|. \label{r:temp}
\end{equation}
Using the mean-value theorem, (\ref{r:boundder}) and $\mu^{1/2}\leq 1/2$, we get
\begin{align*}
 ( 2 \ve \mu^{1/2}+ \ve \mu ) |\sin w(t) - \sin w(t_1)| & \leq \frac{5}{2}\ve \mu^{1/2} \cdot 2 \ve^{1/2}|t-t_1| = 5 \ve^{3/2} \mu^{1/2}|t-t_1|, \\
 \ve \mu | \cos w(t) - \cos w(t_1)| & \leq \frac{1}{2}\ve \mu^{1/2} \cdot 2 \ve^{1/2}|t-t_1| \leq  \ve^{3/2} \mu^{1/2}|t-t_1|.
\end{align*}
We combine it with (\ref{r:temp}) to get $\mathcal{F}(w(t))-\mathcal{F}(w(t_1)) \geq 0$.

We now also know that $\mathcal{F}(w(t_0))>0$. In the case $t \in [-1/\sqrt{\ve},t_0]$ which is equivalent to $\sin w(t)\in [\pi/2,\pi]$, it suffices to show that  $\mathcal{F}(w(t))-\mathcal{F}(w(t_0))$. We again obtain that $\mathcal{F}(w(t)) - \mathcal{F}(w(t_0))$ is equal to the right-hand side of (\ref{r:temp}) with $t_0$ instead of $t_1$. The rest of the proof is analogous to the previous case.
\end{proof}

\begin{proof}[Proof of Lemma \ref{l:subsuper}]
	We take $w$ constructed in Lemma \ref{l:construction} and set $z^+ = w(t - t_0)$, $z^-=2\pi - w(-t + t_0)$. The claims (i)-(iv) are now straightforward, (v) can be easily checked by direct calculation, and (vi) follows from the definition and the relations
	$2\pi - 4 \arctg x< 4/x$, $4 \arctg(1/x)=2\pi - 4 \arctg x$, holding for all $ x >0$. By the Mean Value Theorem, by noting that $|d(\arctg e^x)/dt | \leq 4 e^{-|x|}$, the construction and bounds on $t_0$, $t_1$, we easily get that for $t \geq 0$,
	$$ |w(t-t_0) - u^{(\ve)}(t)| \ll e^{-\frac{1}{2}\sqrt{\ve}t}(\ve \sqrt{\mu} \: t + \sqrt{\ve \mu}) \ll \sqrt{\ve \mu},$$
	which completes the proof.
\end{proof}

We now construct one-sided minimizers and prove Proposition \ref{p:stable}. We fix $c,t_0,v_0$ until the end of the section, and construct $q^+$, $q^-$ is analogous. For a positive integer $k$, we consider the functional
\begin{equation}
\mathcal{L}_{\omega,k}(q) = \int_{t_0}^{t_0+k}L_{\omega}(u(t),v(t),t)dt. \label{r:intminimum}
\end{equation}
We construct minimizers $q_k$ of (\ref{r:intminimum}), and then obtain $q^+$ as their limit.

\begin{lemma} The functional $\mathcal{L}_{c,k}$ attains its minimum $q_k$ over all $q=(u,v) \in H^1([t_0,t_0+k])^2$ such that 
\begin{equation}
q(t_0)=(\pi,v_0), u(t_1)=2 \pi. \label{r:relative}
\end{equation}
Then $q_k \in H^2([t_0,t_0+k])^2$, and is a solution of Euler-Lagrange equations on $(t_0,t_0+k)$.
\end{lemma}

\begin{proof}
	The Tonelli theorem \cite[Appendix 1]{Mather:91} implies that for a fixed $v_1$, such a minimum is attained over all $q=(u,v)$ such that $v(t_1)=v_1$, as the conditions for the Tonelli theorem to hold in the non-autonomous case are satisfied. It is easy to show that it suffices to consider $v_1$ from a closed interval, which by compactness and continuity of the minimum of (\ref{r:intminimum}) in $v_1$ implies existence of such a minimizing $q_k$. Furthermore, by the Tonelli theorem, $q_k\in H^2([t_0,t_0+k])^2$, and $q_k$ is a solution of the Euler-Lagrange equations on $(t_0,t_0+k)$.
\end{proof}

We denote by $q_k$ the (not necessarily unique) minimizer of (\ref{r:intminimum}) satisfying (\ref{r:relative}). We always set $u_k(t)=2\pi$ for $t \geq t_0+k$. 

\begin{lemma} \label{l:minprop}
	The minimizer $q_k=(u_k,v_k)$ satisfies for all $k \geq k_0$, $k_0$ sufficiently large:
	
	(i) For all $t \in [t_0,t_0+k)$, $0 < u_k(t) < 2\pi$,
	
	(ii) For all $ t > t_0+4\mu / \sqrt{\ve}$, $u_k(t) > \pi$,
	
	(iii) For all $t \geq t_0$, $u_k(t) > \pi - 1/4$,
	
	(iv) For all $t \geq t_0$, $u_k(t) \geq z^+(t-t_0)$. Furthermore, for all $0 \leq T \leq 3/(4 \sqrt{\ve})$,
	$$
	u_k \geq z^+(t-t_0-T).
	$$
\end{lemma}

\begin{proof}
	(i) Assume that for some $t_0 < t^* < t_0+k$, $u_k(t)=2\pi$. We define
	\begin{align*}
	\tilde{q} & = 
	\begin{cases} 
	\tilde{q}(t)=q_k(t) & \text{for }t \in [t_0,t^*], \\
	\tilde{q}(t)=(2\pi,v^k(t^*)) & \text{for }t \in [t^*,t_0+k].
	\end{cases}
	\end{align*}
	Now it is easy to check by direct calculation that $\mathcal{L}_{c,k}(\tilde{q}) \leq \mathcal{L}_{c,k}(q_k)$, thus $\tilde{q}$ minimizes $\mathcal{L}_{c,k}$ and must be a solution of the Euler-Lagrange equations on $(t_0,t_0+k)$. We deduce that $\tilde{u} \equiv 2\pi$ which is in contradiction to $\tilde{u}(t_0)=u_k(t_0)=\pi$. By continuity we get the right-hand side of (i). Similarly we show $u_k(t)> 0$, otherwise we replace the segment between two intersections of $0$ with $u^k(t)=0$ and get a contradiction.
	
	(ii) By (\ref{s:bound}), we can find $k_0$ large enough so that for any $k \geq k_0$, 
	\begin{equation}
	\mathcal{L}_{c,k}(q_k) \leq 4\sqrt{\ve(1+3\mu/2)}. \label{r:lckup}
	\end{equation}
	Assume $u_k(t_0+d) = \pi$ for some $t_0+d \in (t_0,t_0+k)$. Again by (\ref{s:bound}), as $q^k$ is a minimizer, we see that
	\begin{equation}
	\int_{t_0+d}^{t_0+k}L_{\omega}(q_k,(q_k)_t,t)dt \geq 4\sqrt{\ve(1-\mu)}. \label{r:lckdown}
	\end{equation}
	From (\ref{r:lckup}) and (\ref{r:lckdown}) and the standing assumption $\mu \leq 1/16$ we obtain the upper bound
	\begin{equation}
	\int_{t_0}^{t_0+d}L_{\omega}(q_k,(q_k)_t,t)dt \leq 4\sqrt{\ve(1+3\mu/2)} - 4\sqrt{\ve(1-\mu)} \leq 6 \sqrt{\ve} \mu.
	\label{r:upperLc}
	\end{equation}
	Let $t^*=(d/2) \wedge 1$, and let $\pi -a$ be the minimal value of $u_k$ on $[t_0,t_0+t^*]$, $0 \leq a < \pi$. It is easy to see that
	\begin{align*}
		\int_{t_0}^{t_0+t^*}L_{\omega}(q_k,(q_k)_t,t)dt & \geq 	\int_{t_0}^{t_0+t^*} \left( \frac{1}{2}((u_k)_t)^2 + \ve(1-\mu)(1 - \cos u_k)\right) dt \\
		& \geq \frac{1}{2 t^*}a^2+ \ve(1-\mu)\left(2-\frac{a^2}{2}\right)t^* \\
		& \geq 2 \ve (1-\mu)t^* \geq \frac{3}{2}\ve t^*.
	\end{align*}
	Repeating that over $[t_0+d-t^*,t_0+d]$ we get $\int_{t_0}^{t_0+d}L_{\omega}(q_k,(q_k)_t,t) \geq 3\ve t^*$, thus by (\ref{r:upperLc}), $t^* \leq 2 \mu /\sqrt{\ve} \leq 1$. By definition, $d \leq 4 \mu /\sqrt{\ve} $. The claim follows by continuity of $u_k$.
	
	(iii) Using the same notation as in (ii) and by (\ref{r:upperLc}), we easily see that
	$$
	6 \sqrt{\ve} \mu \geq \int_{t_0}^{t_0+d}L_{\omega}(q^k,q^k_t,t)dt \geq \frac{2}{d}a^2 \geq \frac{\sqrt{\ve}}{2 \mu}a^2,
	$$
	thus $a^2 \leq 12 \mu^2$. By (A2), $a \leq 1/4$.
	% reference on the standing assumption \mu \leq 1/16 goes here
	
	(iv) We prove it by using Lemma \ref{l:stationary} in two steps. First we show that 
	\begin{equation}
	u_k(t) > z^+(t-t_0-T) \text{ for all } t \geq t_0+T-3/(4\sqrt{\ve}) \label{r:ineq}
	\end{equation}
	and all $T \geq 3/(4\sqrt{\ve})$. By definition and Lemma \ref{l:subsuper}, (i), (\ref{r:ineq}) holds for $T=k+3/(4\sqrt{\ve})$. Assume the contrary and find the infimum $T^*$ of $T \geq 3/(4\sqrt{\ve})$ for which (\ref{r:ineq}) holds. By compactness and continuity, we have that (\ref{r:ineq}) holds for $T=T^*$ with $\geq$ instead of $>$. However, by construction, (ii),(iii) and Lemma \ref{l:subsuper}, (v), the strict inequality in (\ref{r:ineq}) holds for $t \in \lbrace t_0+T-3/(4\sqrt{\ve}),t_0+k \rbrace $, thus by Lemma \ref{l:stationary} we must have strict inequality in (\ref{r:ineq}) and $T^*=3/(4\sqrt{\ve})$.
		
	Now we show that 
	\begin{equation}
	u_k(t) > z^+(t-t_0-T) \text{ for all } t \geq t_0 \label{r:ineq2}
	\end{equation}
	for all $T \in (0, 3/(4\sqrt{\ve})]$. Again we find the infimum $T^*$ for which (\ref{r:ineq2}) holds. We obtain contradiction analogously to the previous step, using $u_k(t) > z^+(t-t_0-T)$ for $t \in \lbrace t_0,t_0+k \rbrace$; unless $T^*=0$ and (\ref{r:ineq2}) holds with $\geq$ instead of $>$, which we need to show.
\end{proof}

\begin{proof}[Proof of Proposition \ref{p:stable}]
	It is easy to check that for $k \geq k_0$, $k_0$ as in Lemma \ref{l:minprop}, $q_k$ is uniformly bounded in $H^2_{\text{loc}}([t_0,\infty))^2$, i.e. that for any $T >k_0$, we have that $q_k$, $k\geq T$, is uniformly bounded in $H^2([t_0,t_0+T])^2$. Indeed, uniform bounds in $k$ on $||(q_k)_t||_{L^2([t_0,t_0+T])^2}$ follow from the fact that $\mathcal{L}_{\omega,k}(q_k)$ is by definition decreasing in $k$, and the uniform bound $|V(q,t)| \leq \ve(1+\mu)$. The uniform bound in $k$ on $||q_k||_{L^2([t_0,t_0+T])^2}$ follows from that and $q(t_0)=(\pi,v_0)$. The uniform bound in $k$ on $||(q_k)_{tt}||_{L^2([t_0,t_0+T])^2}$ is deduced from the fact that $q_k$ is a solution of Euler-Lagrange equations, and the uniform bounds $|V_u| \leq \ve(1+\mu)$, $|V_v|\leq \ve \mu$. 
	
	Now by diagonalization we find a convergent subsequence of $q_k$ in $H^1_{\text{loc}}([t_0,\infty))^2$ converging to some $q^0 \in H^1_{\text{loc}}([t_0,\infty))^2$. It is straightforward to show that $\lim_{k \rightarrow \infty} \mathcal{L}_{\omega,k}(q_k) = S^+_{\omega}(t_0,v_0)$, as $\mathcal{L}_{\omega,k}(q_k)$ is decreasing, bounded by $S^+_{\omega}(t_0,v_0)$ from below, and we can arbitrarily well approximate $S^+_{\omega}(t_0,v_0)$ with $\mathcal{L}_{\omega,k}(q_k)$ for $k$ large enough. 
	
	By construction and Lemma \ref{l:minprop}, (iv), $\lim_{t\rightarrow \infty }u^0(t)= 2\pi$. 
	By the Fatou Lemma, $\int_{t_0}^{\infty}L_{\omega}(q^0,q^0_t,t) \leq S^+_{\omega}(t_0,v_0)$, thus by the definition of $S^+_{\omega}(t_0,v_0)$, the equality must hold, and the existence is proved.
	
	Now, if $q^+$ is any one-sided minimizer at $(c,t_0,v_0)$, it must be by the Tonelli theorem a solution of Euler-Lagrange equations on $(t_0,\infty)$, and $C^4$ because of the regularity of solutions of ordinary differential equations.
	
	The proof of (ii) is analogous to the proof of Lemma \ref{l:minprop}, (iv).
\end{proof}

\subsection{Proof from Section \ref{s:Theorem3}}

Let $\tilde{\ve}=\sqrt{\ve(1+2\mu^{1/2})}$, and let $\delta = 2\sqrt{\tilde{\ve} \mu}$. Denote by $u^{(\tilde{\ve})} = 4 \arctg (e^{\sqrt{\tilde{\ve}}\: t})$ the separatrix solution of the unperturbed pendulum $u_{tt} = \tilde{\ve}\sin u$.

\begin{lemma} \label{l:pendulum}
If $\tilde{u}$ is a solution of $u_{tt} = \tilde{\ve}\sin u$, $u(0)=\pi$, $u_t(0)=2\sqrt{\tilde{\ve}}(1 + \delta^2)$, then for some absolute $c_{30}>0$, we can find $t_1>0$ such that
\begin{equation}
t_1 \leq \frac{c_{30}}{\sqrt{\tilde{\ve}}} |\log \delta |,
\end{equation}
we have $2\pi + \delta /2 \leq \tilde{u}(t_1) \leq 2\pi + 32\delta$, and for all $t \in [0, t_1]$, $|\tilde{u}(t)-u^{(\tilde{\ve})}(t) | \leq 32 \delta$.
\end{lemma}

\begin{proof} Let
$w^{(1)}=\tilde{u}-u^{(\tilde{\ve})}$, let $t \geq 0$, and let $t_1 > 0$ be the unique (by monotocity) $t$ such that $\tilde{u}(t_1)=2\pi + \delta$.
Clearly $\tilde{u}(t)\geq u^{(\tilde{\ve})}(t)$, thus $w^{(1)}(t) \geq 0$. By definition, $w^{(1)}$ is solves the differential inequality $w^{(1)}_{tt}(t) \leq \ve |\tilde{u}(t)-u^{(\tilde{\ve})}(t)| = \ve w^{(1)}(t)$. Now if
$w^{(2)}= w^{(1)}_t - \sqrt{\tilde{\ve}}w^{(1)}$, $w^{(3)}= w^{(1)}_t + \sqrt{\tilde{\ve}}w^{(1)}$, we have $w^{(2)}(0)=w^{(3)}(0)=2\sqrt{\tilde{\ve}} \cdot \delta^2$, and they are solutions of $w^{(2)}_t \leq - \sqrt{\tilde{\ve}}w^{(2)}$, $w^{(3)}_t \leq  \sqrt{\tilde{\ve}}w^{(3)}$. Thus by the Gronwall Lemma,
\begin{align}
w^{(2)}(t) \leq 2 \sqrt{\tilde{\ve}} \: \delta^2 \: \ e^{-\sqrt{\tilde{\ve}}\: t}, \hspace{5ex} w^{(3)}(t) \leq 2  \sqrt{\tilde{\ve}} \: \delta^2 \: e^{\sqrt{\tilde{\ve}}\: t}. \label{r:w2w3}
\end{align}
By the conservation of energy $H(u)=u_t^2/2 - (1-\cos u)$, and by $\tilde{u}(t) \geq u^{(\tilde{\ve})}(t)$ and $\tilde{u}_t, u^{(\tilde{\ve})}_t \geq 0$ we easily obtain that $w^{(1)}_t \geq 0$, thus the right-hand side of (\ref{r:w2w3}) implies for all $t \in [0, t_1]$,
\begin{equation}
   w^{(1)}(t) \leq 2\delta^2 \: e^{\sqrt{\tilde{\ve}}\: t_1}. \label{r:w2a}
\end{equation}
By the conservation of energy, it is easy to show that for all $t \geq 0$, $\tilde{u}_t(t)\geq 2 \sqrt{\tilde{\ve}} \: \delta$. Inserting that in the left-and side of (\ref{r:w2w3}) we obtain
\begin{equation}
w^{(1)}(t_1) \geq 2\delta - \frac{1}{\sqrt{\tilde{\ve}}}u_t^{(\tilde{\ve})}(t_1) - 2 \delta^2 \: e^{-\sqrt{\tilde{\ve}}\: t_1}. \label{r:w3a}
\end{equation}
Now for all $t \geq 0 $, as $2 \pi - 4 \arctg x < 4 /x$, we have  $2\pi - u^{(\tilde{\ve})}(t) \leq 4 e^{-\sqrt{\tilde{\ve}}t}$. By the conservation of energy and $(1-\cos(2\pi-x))\leq x^2/2$, we now get that $u^{(\tilde{\ve})}_t(t) < 4 e^{-\sqrt{\tilde{\ve}}t}$. Thus choosing $t_1$ so that $4 e^{-\sqrt{\tilde{\ve}}t}=\delta / 4$, we get $2\pi - \delta/4 < u^{(\tilde{\ve})}(t_1) < 2\pi$, (\ref{r:w2a}) becomes $w^{(1)}(t) \leq 32 \delta$, and as $\delta^2 \leq 1$, (\ref{r:w3a}) becomes $w^{(1)}(t) \geq \delta $, which implies the claim.  
\end{proof}

\begin{proof}[Proof of Lemma \ref{l:subsuper2}]
Take $\tilde{u}$ from Lemma \ref{l:pendulum}, and set
\begin{equation*}
w(t) =
\begin{cases}
\tilde{u}(t) & t \geq t_1 \\
\tilde{u}(t) - \ve^{3/2}\mu^{1/2}(t_1-t)^3 & t \leq t_1,
\end{cases}
\end{equation*}
where $t_1$ is chosen so that $\tilde{w}(t)=\pi+2\mu^{1/2}$. Then as in Lemma \ref{l:construction}, we can find $0 < t_0 < \mu^{1/2}$ such that $w(t_0)=\pi$. We set $\tilde{z}^+(t)=w(t-t_0)$, $\tilde{T}=t_1-t_0$, and $\tilde{z}^-=2\pi-w(-t+t_0)$. The rest of the proof is analogous to the proof of Lemma \ref{l:subsuper}, using also the relation $|u^{(\ve)}-u^{(\tilde{\ve})}|\ll \sqrt{\ve \mu}$ to obtain (v).
\end{proof}

\section{Appendix C: Proofs of bounds on derivatives}

This Appendix is dedicated to the proof of Proposition \ref{p:Bthree}. 
In the first subsection we obtain weighted a-priori bounds on $q^0$ and the potential $V$ and its derivatives. In the second subsection we by careful differentiation we obtain a differential inequality which proves the 
invariance of sets.

\subsection{A-priori bounds on weighted integrals}
Throughout the subsection we assume $q \in \mathcal{B}_1$. 

\begin{lemma} \label{l:weights}
	Let $j,m \in \lbrace 1,2,3,4,5,6 \rbrace$. Then for some absolute constant $c_{31} >0$,
	\begin{align}
	\ve^{(j+1)/2} \int_{\mathbb{R}}  e^{-\lambda(\tau)|t-\tau|} e^{-\frac{m}{2}\sqrt{\ve}\dt}dt & \leq c_{31}\lambda(\tau)^j, \label{r:boundi} \\
	\int_{\mathbb{R}}\frac{1}{L_{k(t)}^{j+1}} e^{-\lambda(\tau)|t-\tau|} dt & \leq c_{31} \lambda(\tau)^j. \label{r:boundLk}
	\end{align}
\end{lemma}

\begin{proof} To show (\ref{r:boundi}), without loss of generality let 
	\begin{equation}
	\tau \in [\tilde{T}_k-2L_{k-1},\tilde{T}_{k+1}+2L_k] \label{r:wheretau}
	\end{equation}
	for a fixed $k \in \mathbb{Z}$, i.e. $||\tau||=|\tau-\tilde{T}_k|.$ By applying the definition of $\dt$ and substituting $t -T_j+\tilde{T}_k$ instead of $t$ in the second line below, we get
	\begin{align}
	 \int_{\mathbb{R}}  e^{-\lambda(\tau)|t-\tau|} e^{-\frac{m}{2}\sqrt{\ve}\dt}dt & = 
	 \sum_{j=-\infty}^{\infty}\int_{T_j-2L_{j-1}}^{T_j+2L_j} e^{-\lambda(\tau)|t-\tau|} e^{-\frac{m}{2}\sqrt{\ve}|t-T_j|}dt \notag \\
	 & = \sum_{j=-\infty}^{\infty}\int_{\tilde{T}_k-2L_{j-1}}^{\tilde{T}_k+2L_j} e^{-\lambda(\tau)|t-\tilde{T}_k+T_j-\tau|} e^{-\frac{m}{2}\sqrt{\ve}|t-\tilde{T}_k|}dt. \label{r:exponent}
	\end{align}
Now by the assumption (\ref{r:wheretau}), for $j=k-1,k,k+1$ we have $|T_j-\tau| \geq |\tilde{T}_k-\tau|$. For $j \geq k+1$, $|T_j-\tau|\geq |T_j-\tilde{T}_{k+1}|+|\tilde{T}_k-\tau|\geq 4L(|j-k|-1)+|\tilde{T}_k-\tau|$. By analogous consideration for $j \leq k-1$, we conclude that for all $j \in \mathbb{Z}$,
\begin{equation}
|T_j-\tau| \geq |\tilde{T}_k-\tau| + 4L(|j-k|-1) = \dtau + 4L((|j-k|-1)\vee 0). \label{r:Ttau1}
\end{equation}
Now by definition of $\lambda(\tau)$, we have $\sqrt{\ve} \: m/2 \geq \sqrt{\ve}/2 \geq \lambda(\tau)+ \sqrt{\ve}/4$. Combining it with (\ref{r:Ttau1}), we can estimate the exponent in (\ref{r:exponent}):
\begin{align*}
\lambda(\tau)|t-\tilde{T}_k+T_j-\tau|+\frac{m}{2}\sqrt{\ve}|t-\tilde{T}_k| & \geq \lambda(\tau)|T_j-\tau|-\lambda(\tau)|t-\tilde{T}_k|+\frac{m}{2}\sqrt{\ve}|t-\tilde{T}_k| \\
& \geq \lambda(\tau)\dtau + \lambda(\tau)4L((|j-k|-1)\vee 0) + \frac{1}{4}\sqrt{\ve}|t-\tilde{T}_k|.
\end{align*}
Thus (\ref{r:exponent}) is less or equal than
\begin{equation}
e^{-\lambda(\tau)\dtau} \left( 2+\sum_{j=-\infty}^{\infty}e^{-\lambda(\tau)4L\cdot j}\right) \int_{-\infty}^{\infty}e^{-\frac{1}{4}\sqrt{\ve}|t-\tilde{T}_k|} = \frac{4}{\sqrt{\ve}}e^{-\lambda(\tau)\dtau}  \frac{1}{1-e^{-4\lambda(\tau)L}}. \label{r:exp2}
\end{equation}
If $\lambda(\tau)=\sqrt{\ve}/4$, then $4L\lambda(\tau)\geq 1$, thus right-hand side of (\ref{r:exp2}) is $\ll \ve^{-1/2}$. This implies (\ref{r:boundi}). Otherwise $\lambda(\tau)=8\log \dtau / \dtau$, so the right-hand side of (\ref{r:exp2}) becomes
\begin{equation}
\frac{4}{\sqrt{\ve}} \frac{||\tau||^{-8}}{1 - \dtau^{-4L/\dtau}}. \label{r:denom}
\end{equation}
Now by substitution $z=(1/||\tau||)\log(1/||\tau||)$, noting that by all the assumptions, $z \in (-1/2,0)$, we easily show by Taylor expansion that the denominator in (\ref{r:denom}) is $\ll ||\tau||\log (1/||\tau||)| \leq ||\tau||^2$. Thus (\ref{r:denom}) is $\ll \ve^{-1/2} ||\tau||^{-6} \leq  \ve^{-1/2} \lambda(\tau)^6$, which implies (\ref{r:boundi}).
	
To show (\ref{r:boundLk}), let $\tau \in [\tilde{T}_k,\tilde{T}_{k+1}]$ for fixed $k \in \mathbb{Z}$. As $L_k \geq 4L \geq 1/\sqrt{\ve}$,  %CHECK THIS, REFERENCE
	\begin{align}
	\int_{\mathbb{R}}\frac{1}{L_{k(t)}^{j+1}} e^{-\lambda(\tau)|t-\tau|} dt  & \leq \frac{1}{L_k^{j+1}} \int_{\tilde{T}_k}^{\tilde{T}_{k+1}}e^{-\lambda(\tau)|t-\tau|}dt +
	\ve^{\frac{j+1}{2}} e^{-\lambda(\tau)|\tau-\tilde{T}_k|}\int_{-\infty}^{\tilde{T}_k} e^{-\lambda{\tau}|t-\tilde{T}_k|}dt \notag \\ 
	&  \quad + 	\ve^{\frac{j+1}{2}} e^{-\lambda(\tau)|\tau-\tilde{T}_{k+1}|}\int_{\tilde{T}_{k+1}}^{\infty} e^{-\lambda(\tau)|t-\tilde{T}_{k+1}|}dt \notag \\
	& \leq \frac{2}{\lambda(\tau) \: L_k^{j+1}} + \frac{2 \ve^{\frac{j+1}{2}}}{\lambda(\tau)}e^{-\lambda(\tau)||\tau||}. \label{r:sumLk}
	\end{align}
	As $L_k \geq 1/\lambda(\tau)$, %CHECK THIS, REFERENCE
	the first summand in (\ref{r:sumLk}) is always $\ll \lambda(\tau)^j$. If $\lambda(\tau)=\sqrt{\ve}/8$, the second summand is $\ll \ve^{j/2} \ll \lambda(\tau)^j$. Otherwise, as $\ve \leq 1$, and as $||\tau|| \geq 1/ (8 \lambda(\tau))$,
	\begin{align*}
	\frac{\ve^{2\frac{j+1}{2}}}{\lambda(\tau)}e^{-\lambda(\tau)||\tau||} \leq \frac{1}{\lambda(\tau)}e^{-8 \log ||\tau||} 
	& \leq \frac{1}{||\tau||^8 \lambda(\tau)} \ll \lambda(\tau)^6 \leq \lambda(\tau)^j,
	\end{align*}
	which completes (\ref{r:boundLk}).
\end{proof}	

\begin{lemma} \label{l:Cinitial} 
	There exists an absolute constant $c_{32} > 0$ so that
	\begin{equation}
	||V_u||^2_{L^2_{\tau}(\mathbb{R})} \leq c_{32} \lambda(\tau)^3, \quad \quad ||V_v||^2_{L^2_{\tau}(\mathbb{R})}  \leq c_{32} \lambda(\tau)^3. \label{r:Vuupper}
	\end{equation}
\end{lemma}	

\begin{proof} As $q \in \mathcal{B}_1$, by (\ref{r:main1}), we get
\begin{align}
|\sin u(t)| & \leq |u-2k(t)\pi| \ll e^{-\frac{1}{2}\sqrt{\ve}\dt }, \label{r:sinrel} \\
|1-\cos u(t)| & \leq \frac{1}{2}|u-2k(t)\pi|^2 \ll e^{-\sqrt{\ve}\dt }. \label{r:cosrel}
\end{align}
By definition and uniform bounds on $f$ and its derivatives, we have 
\begin{align}
|V(u,v,t) & \leq 2 \ve (1-\cos u(t)) \ll \ve \: e^{-\sqrt{\ve}\dt }, \notag \\
|V_u(u,v,t)| & \leq 2 \ve (1-\cos u(t)+|\sin u(t)|) \ll \ve \: e^{-\frac{1}{2}\sqrt{\ve}\dt }, \label{r:Vubound} \\
|V_v(u,v,t)| & \leq \ve \mu (1- \cos u(t)) \ll \ve \mu \: e^{-\sqrt{\ve}\dt }. \label{r:Vvbound}
\end{align}
It suffices now to apply (\ref{r:boundi}).
\end{proof}

\begin{lemma} 
There exists an absolute constant $c_{33}>0$ so that
\begin{align}
||u^0_t||_{\Ltau}^2 & \leq c_{33} \lambda(\tau), & ||v^0_t-\omega_{k(.)}||_{\Ltau}^2 & \leq c_{33} \lambda(\tau), \label{r:qlamone} \\
||u^0_{tt}||_{\Ltau}^2 & \leq c_{33} \lambda(\tau)^3, & ||v^0_{tt}||_{\Ltau}^2 & \leq c_{33} \lambda(\tau)^3, \label{r:qlamtwo} \\
||u^0_{ttt}||_{\Ltau}^2 & \leq c_{33} (\varpi^2+1)\lambda(\tau)^3, & ||v^0_{ttt}||_{\Ltau}^2 & \leq c_{33} (\varpi^2+1) \lambda(\tau)^3. \label{r:qlamthree}
\end{align}
\end{lemma}
\begin{proof}
This follows by a straightforward calculation from Lemmas \ref{l:q*bounds} and \ref{l:weights}. 
\end{proof}

\subsection{Proof of invariance of $\mathcal{B}_2$, $\mathcal{B}_3$, $\mathcal{B}_4$}

By assumption that $q \in \mathcal{A}$, we have that $q \in H^3_{\text{loc}}(\mathbb{R})^2$. Whenever we require higher derivatives in the proofs to follow, we assume that we evaluate all on a dense, sufficiently smooth subset and then extend the claims of the Lemmas by continuity to the entire set as required.

\begin{lemma} \label{l:poincare1}
	If $q=(u,v) \in \mathcal{B}_2$, then
	\begin{align}
	|| u_t-u^0_t||^4_{\Ltau} &\leq c_{34} \left( \lambda(\tau)^2 + \lambda(\tau)^{-1}|| u_{tt}-u^0_{tt}||^2_{\Ltau}\right),    \label{r:PoinU} \\
	|| v_t-v^0_t||^4_{\Ltau} &\leq c_{35}  \left( \lambda(\tau)^2M^4 + \lambda(\tau)^{-1}M^2|| v_{tt}-v^0_{tt}||^2_{\Ltau}\right),         \label{r:PoinV} 
	\end{align}
	for some absolute constants $c_{34},c_{35} \geq 1$.
\end{lemma}

\begin{proof} Denote by $X=|| u_t-u^0_t||^2_{\Ltau}$ and by $Y=|| u_{tt}-u^0_{tt}||^2_{\Ltau}$. By partial integration and the Cauchy-Schwartz inequality, we get
\begin{align}
X & =\int_{\mathbb{R}}e^{-\lambda(\tau)|t-\tau|} (u_t(t)-u^0_t(t))^2 dt  \notag \\ & \leq \lambda(t)\int_{\mathbb{R}}e^{-\lambda(\tau)|t-\tau|}|u(t)-u^0(t)||u_t(t)-u^0(t)|dt + \int_{\mathbb{R}}e^{-\lambda(\tau)|t-\tau|}|u(t)-u^0(t)||u_{tt}(t)-u^0_{tt}(t)|dt \notag \\
& \ll \lambda(\tau)^{1/2} ||u-u^0||_{L^{\infty}(\mathbb{R})} \: X^{1/2} + \lambda(\tau)^{-1/2} ||u-u^0||_{L^{\infty}(\mathbb{R})} \: Y^{1/2}. \label{r:poinA}
\end{align}
Now either $X \ll \lambda(\tau) ||u-u^0||^2_{L^{\infty}(\mathbb{R})}$ or $X \ll \lambda(\tau)^{-1/2} ||u-u^0||_{L^{\infty}(\mathbb{R})}$, thus
\begin{equation}
 X^2 \ll \lambda(\tau)^2 ||u-u^0||^4_{L^{\infty}(\mathbb{R})} + \lambda(\tau)^{-1} ||u-u^0||^2_{L^{\infty}(\mathbb{R})}.
\end{equation}
From (\ref{r:Uzero}) and (\ref{r:main1}) we have $||u-u^0||_{L^{\infty}(\mathbb{R})} \leq c_6+c_7$, which gives (\ref{r:PoinU}).

To show (\ref{r:PoinV}), it suffices to note that by (\ref{r:main2}), $||v-v^0||_{L^{\infty}(\mathbb{R})} \leq c_8M$. The rest of the proof is analogous. 
\end{proof}

\begin{lemma} \label{l:b2one} 
	The set of all $q \in \mathcal{B}_1$ such that $||u_t-u^0_t||^2_{\Ltau} \leq c_9 \lambda(\tau)$ is $\mathcal{A}$-relatively $\xi$-invariant for some absolute constant $c_9>0$.
\end{lemma}

\begin{proof}
	First note that
	\[
	 u_s-(u_{tt} - u^0_{tt}) = -V_u + u^0_{tt}, 
	\]
	thus by squaring it we get
	\begin{equation}
	- 2 u_s(u_{tt} - u^0_{tt}) \leq - u_s^2 - (u_{tt} - u^0_{tt})^2 + 2V_u^2 + 2(u^0_{tt})^2, \label{r:uu1}
	\end{equation}
	where we write $V_u = \partial_u V(u,v,t)$.
	Differentiating with respect to $s$, by partial integration and taking into account that $u^0$ is constant , we obtain
	\begin{align}
	 \frac{d}{ds}||u_t-u^0_t||^2_{\Ltau} & = 2 \int_{\mathbb{R}} e^{-\lambda(\tau)|t-\tau|}(u_t-u^0_t)u_{ts} dt \notag \\
	 & \leq -  2\int_{\mathbb{R}} e^{-\lambda(\tau)|t-\tau|}(u_{tt}-u^0_{tt})u_{ts}dt  + 2 \lambda(\tau)\left| \int_{\mathbb{R}} e^{-\lambda(\tau)|t-\tau|}(u_{t}-u^0_{t})u_{s} dt \right|	\label{r:uu2}
	\end{align}
	The first summand is by inserting (\ref{r:uu1}),
	\begin{equation}
	 - 2\int_{\mathbb{R}} e^{-\lambda(\tau)|t-\tau|}(u_{tt}-u^0_{tt})u_{ts}dt \leq -||u_s||^2_{\Ltau}-||u_{tt}-u_{tt}^*||^2_{\Ltau}+2||V_u||^2_{\Ltau}+2||u^0_{tt}||^2_{\Ltau}. \label{r:uu3}
	\end{equation}
	By Young's inequality, we have that
	\begin{equation}
	2 \lambda(\tau)\left| \int_{\mathbb{R}} e^{-\lambda(\tau)|t-\tau|}(u_{t}-u^0_{t})u_{s} dt \right| \leq \lambda(\tau)^2||u_{t}-u^0_{t}||^2_{\Ltau}+ ||u_s||^2_{\Ltau}. \label{r:uu4}
	\end{equation}
	Inserting (\ref{r:uu3}) and (\ref{r:uu4}) into (\ref{r:uu2}), we see that
	\begin{equation}
	\frac{d}{ds}||u_t-u^0_t||^2_{\Ltau} \leq - ||u_{tt}-u_{tt}^0||^2_{\Ltau} + \lambda^2(\tau)||u_t-u^0_t||^2_{\Ltau} + 2||V_u||^2_{\Ltau}+2||u^0_{tt}||^2_{\Ltau}. \label{r:uu5}
	\end{equation}
	By substituting $X=||u_t-u_t^0||^2_{\Ltau}$, $Y=||u_{tt}-u_{tt}^0||^2_{\Ltau}$, and using (\ref{r:Vuupper}) and (\ref{r:qlamtwo}), we obtain from (\ref{r:uu5})
	\begin{equation}
	  \frac{d}{ds}X\leq -Y + \lambda^2(\tau) X + 2(c_{32}+c_{33})\lambda(\tau)^3. \label{r:uu5a}
	\end{equation}
	(\ref{r:PoinU}) can be written as
	\begin{equation}
	- \frac{1}{2}Y \leq -\frac{\lambda(\tau)}{2c_{34}} X^2 + \frac{1}{2}\lambda(\tau)^3. \label{r:YX1}
	\end{equation}
	Clearly 
	\begin{equation}
	0 \leq \frac{\lambda(\tau)}{2c_{34}} X^2  - 2\lambda(\tau)^2X + 2c_{34}\lambda(\tau)^3, \label{r:YX2}
	\end{equation}
	thus by summing (\ref{r:YX1}) and (\ref{r:YX2}) we see that
	\[
		- \frac{1}{2}Y +\lambda(\tau)^2 X \leq -\lambda(\tau)^2 X +(2c_{34}+1/2) \lambda(\tau)^3.
	\]
	Summing it with (\ref{r:uu5a}) and substituting back $X$ and $Y$, we obtain a differential inequality
	\begin{align}
	\frac{d}{ds}||u_t-u^0_t||^2_{\Ltau} & \leq - \frac{1}{2} ||u_{tt}-u_{tt}^*||^2_{\Ltau}   -  \lambda(\tau)^2||u_t-u^0_t||^2_{\Ltau} + c_9 \lambda(\tau)^3 \label{r:b1ineq} \\
	& \leq - \lambda(\tau)^2||u_t-u^0_t||^2_{\Ltau} + c_9 \lambda(\tau)^3, \label{r:b1ineq2}
	\end{align}
	where $c_9=2(c_{32}+c_{33}+c_{34})+1/2$. The claim follows from the Gronwall's lemma and Lemma \ref{l:B1}, i.e. $\mathcal{A}$-relative invariance of $\mathcal{B}_1$.	
\end{proof}

\begin{lemma} \label{l:b2two}
	The set of all $q \in \mathcal{B}_1$ such that $||v_t-v^0_t||^2_{\Ltau} \leq c_9(M^2+1) \: \lambda(\tau)$ is $\mathcal{A}$-relatively $\xi$-invariant for some absolute constant $c_9>0$.
\end{lemma}

\begin{proof} Denote by $\tilde{X}=||v_t-v^0_t||^2_{\Ltau}$, $\tilde{Y}=||v_{tt}-v^0_{tt}||^2_{\Ltau}$. All the steps in the proof are analogous to the proof of Lemma \ref{l:b2one} up to the equations (\ref{r:YX1}), (\ref{r:YX2}), instead of which we have
	\begin{align*}
	- \frac{1}{2}\tilde{Y} & \leq -\frac{\lambda(\tau)}{2c_{35} M^2} \tilde{X}^2 + \frac{1}{2}\lambda(\tau)^3M^2, \\
	0 & \leq \frac{\lambda(\tau)}{2c_{35}M^2} \tilde{X}^2  - 2\lambda(\tau)^2 \tilde{X} + 2c_{35}\lambda(\tau)^3 M^2.
	\end{align*}
	We thus obtain the differential inequality
		\begin{align}
		\frac{d}{ds}||v_t-v^0_t||^2_{\Ltau} & \leq - \frac{1}{2}||v_{tt}-v^0_{tt}||^2_{\Ltau}- \lambda(\tau)^2||v_t-v^0_t||^2_{\Ltau} + c_9 (M^2+1)\lambda(\tau)^3 \\
		& \leq - \lambda(\tau)^2||v_t-v^0_t||^2_{\Ltau} + c_9 (M^2+1)\lambda(\tau)^3.
		\end{align}
	which completes the proof analogously as in Lemma \ref{l:b2one}.
\end{proof}

Lemmas \ref{l:b2one} and \ref{l:b2two} compete the proof of $\mathcal{A}$-relative $\xi$-invariance of $\mathcal{B}_2$. 

\begin{lemma} There exists a constant $c_{10}>0$ such that the set of all $q \in \mathcal{B}_2$ satisfying (\ref{r:defB3u}), (\ref{r:defB3v}) is an $\mathcal{A}$-relative $\xi$-invariant set.	
\end{lemma}

\begin{proof} The constant $c_{36}$ may change from line to line in the proof.
As $u_{ts}-u_{ttt}=D_tV_u$, we get
	\[
	- 2u_{ttt}u_{ts} = - u_{ttt}^2-u_{ts}^2 +(D_tV_u)^2,
	\]
thus
\[
	\frac{d}{ds}u^2_{tt} =  2 u_{tt}u_{tts}= 2(u_{tt}u_{ts})_t - 2u_{ttt}u_{ts} = 2(u_{tt}u_{ts})_t - u_{ttt}^2-u_{ts}^2 +(D_tV_u)^2.
\]
Calculating the weighted integral, we get by partial integration and the Young's inequality:
\begin{align}
\frac{d}{ds}||u_{tt}||^2_{\Ltau} & = -||u_{ttt}||^2_{\Ltau}-||u_{ts}||^2_{\Ltau}+ 2\int_{\mathbb{R}}e^{-\lambda(\tau)|t-\tau|}(u_{tt}u_{ts})_t dt + ||D_tV_u||^2_{\Ltau} \notag \\
& \leq -||u_{ttt}||^2_{\Ltau}-||u_{ts}||^2_{\Ltau}+ \lambda(\tau)^2||u_{tt}||^2_{\Ltau}+||u_{ts}||^2_{\Ltau}
+ ||D_tV_u||^2_{\Ltau} \notag \\
& \leq -||u_{ttt}||^2_{\Ltau} + \lambda(\tau)^2||u_{tt}||^2_{\Ltau}
+ ||D_tV_u||^2_{\Ltau}. \label{r:dutt}
\end{align}

By careful differentiation, while using uniform bounds on $f$ and its derivatives, and as $\mu \leq 1$ and $\varpi \geq 1$ by definition, we get
\begin{align*}
|D_tV_u| & \ll \ve \mu |u_t| + \ve |u_t| + \ve\mu (1-\cos u + |\sin u| )|v_t| + \ve \mu (1-\cos u + |\sin u|)) \\
& \ll \ve |u_t-u^0_t| + \ve|u^0_t|+\ve |v_t - v^0_t| + \ve |v^0_t - \omega_{k(t)}| + \ve (1-\cos u + |\sin u| )\varpi.
\end{align*}
Now by weighted integration and applying (\ref{r:sinrel}), (\ref{r:cosrel}), (\ref{r:boundi}), (\ref{r:defB2u}) and (\ref{r:defB2v}) we obtain
\begin{equation}
||D_tV_u||^2_{\Ltau} \leq c_{36} \: \ve^2 (M^2+\varpi^2) \lambda(\tau). \quad \label{r:DtVu}
\end{equation}
From (\ref{r:qlamtwo}), (\ref{r:b1ineq}) and $\lambda(\tau)\ll \sqrt{\ve}$ we get
\begin{align}
	\frac{d}{ds}||u_t-u^0_t||^2_{\Ltau}  \leq - \frac{1}{2} ||u_{tt}-u_{tt}^0||^2_{\Ltau}  + c_{36}\ve \lambda(\tau) 
	\leq - \frac{1}{4} ||u_{tt}||^2_{\Ltau}  + c_{36}\ve \lambda(\tau). \label{r:b1ineq2a} 
\end{align}

Now, as $q \in \mathcal{B}_2$, we have that $ 0 \leq -||u_t-u^0_t||_{\Ltau}+c_9 \lambda(\tau)$. Multipying it with $\ve^2$ and (\ref{r:b1ineq2a}) with $\ve$ and summing it, we obtain
\begin{equation}
\frac{d}{ds}\ve||u_t-u^0_t||^2_{\Ltau} \leq - \frac{1}{4}\ve ||u_{tt}||^2_{\Ltau} - \ve^2||u_t-u^0_t||_{\Ltau} + c_{36}\ve^2 \lambda(\tau). \label{r:b1ineq3} 
\end{equation}
Denote by $Z=||u_{tt}||^2_{\Ltau}+\ve||u_t-u^0_t||^2_{\Ltau}$. Summing (\ref{r:dutt}), (\ref{r:DtVu}) and (\ref{r:b1ineq3}) and using $\lambda(\tau)^2 \leq \ve /8$, we obtain
\begin{align}
\frac{d}{ds}Z  & \leq -||u_{ttt}||^2_{\Ltau} - \frac{1}{4}\ve Z + c_{36}(M^2+\varpi^2) \ve^2 \lambda(\tau)  \label{r:secondu} \\
 & \leq - \frac{1}{4}\ve Z + c_{36} \ve^2 (M^2+\varpi^2) \lambda(\tau),   \notag
\end{align}
which by the Gronwall's lemma and $\mathcal{A}$-relative $\xi$-invariance of $\mathcal{B}_2$ proves the first part of the claim.

Analogously we obtain
\begin{align}
\frac{d}{ds}||v_{tt}||^2_{\Ltau} & \ll -||v_{ttt}||^2_{\Ltau} + \lambda(\tau)^2||v_{tt}||^2_{\Ltau}
+ ||D_tV_v||^2_{\Ltau}, \label{r:dvtt} \\
|D_tV_v| & \ll \ve \mu \left( |u_t-u^0_t| + |u^0_t|+ |v_t - v^0_t| + |v^0_t - \omega_{k(t)}| + (1-\cos u)\varpi \right), \notag
\end{align}
thus as $\mu \leq 1$,
\begin{equation}
||D_tV_v||^2_{\Ltau} \leq c_{36} \: \ve^2 (M^2+\varpi^2) \lambda(\tau).
\end{equation}
If $\tilde{Z}=||v_{tt}||^2_{\Ltau}+\ve||v_t-v^0_t||^2_{\Ltau}$, we analogously to (\ref{r:secondu}) obtain
\begin{align}
\frac{d}{ds}\tilde{Z} & \leq -||v_{ttt}||^2_{\Ltau} - \frac{1}{4}\ve \tilde{Z} + c_{36} \ve^2 (M^2+\varpi^2) \lambda(\tau) \label{r:secondv} \\
& \leq - \frac{1}{4}\ve \tilde{Z} + c_{36} \ve^2 (M^2+\varpi^2) \lambda(\tau),  \notag
\end{align}
which completes the proof analogously as for $Z$.
\end{proof}

\begin{lemma} There exists a constant $c_{11}>0$ such that the set of all $q \in \mathcal{B}_3$ satisfying (\ref{r:defB4u}), (\ref{r:defB4v}) is an $\mathcal{A}$-relative $\xi$-invariant set.
\end{lemma}

\begin{proof}	
Analogously as in (\ref{r:dutt}), (\ref{r:dvtt}), and as $\lambda(\tau)$ we easily get
\begin{align}
\frac{d}{ds}||u_{ttt}||^2_{\Ltau} &\leq -||u_{tttt}||^2_{\Ltau} + \frac{\ve}{16}||u_{ttt}||^2_{\Ltau}
+ ||D_{tt}V_u||^2_{\Ltau}, \label{r:duttt}  \\
\frac{d}{ds}||v_{ttt}||^2_{\Ltau} & \leq -||v_{tttt}||^2_{\Ltau} + \frac{\ve}{16}||v_{ttt}||^2_{\Ltau}
+ ||D_{tt}V_v||^2_{\Ltau}.  \label{r:dvttt}
\end{align}
By the Sobolev inequalities and as $q \in \mathcal{B}_2$ and $\lambda(\tau)\leq \sqrt{\ve}$, we easily deduce that for $q \in \mathcal{B}_3$,  
\begin{align*}
||u||_{L^{\infty}(\mathbb{R})}& \ll  (M+\varpi) \sqrt{\ve}, \\
||v-\omega_{k(.)}||_{L^{\infty}(\mathbb{R})}& \ll (M+\varpi) \sqrt{\ve}.
\end{align*}
Using that and analogously as in the Proof of Lemma 14.7 (the calculation is analogous and routine, thus omitted), we see that
\begin{align}
||D_{tt}V_u||^2_{\Ltau} & \ll (M^4 + \varpi^4 ) \ve^2 \lambda(\tau), \label{r:du2} \\
||D_{tt}V_v||^2_{\Ltau} & \ll (M^4 + \varpi^4) \ve^2 \lambda(\tau).  \label{r:dv2}
\end{align}
Setting 
\begin{align*}
W & = ||u_{ttt}||^2_{\Ltau} + ||u_{tt}||^2_{\Ltau}+\ve||u_t-u^0_t||^2_{\Ltau} = ||u_{ttt}||^2_{\Ltau}+Z, \\
\tilde{W} & = ||v_{ttt}||^2_{\Ltau} + ||v_{tt}||^2_{\Ltau}+\ve||v_t-v^0_t||^2_{\Ltau} = ||v_{ttt}||^2_{\Ltau}+\tilde{Z},
\end{align*}
and summing  (\ref{r:dutt}),(\ref{r:duttt}) and (\ref{r:du2});  respectively (\ref{r:dvtt}), (\ref{r:dvttt}) and (\ref{r:dv2}), we obtain
\begin{align*}
\frac{d}{ds}W  &\leq - \frac{1}{4}\ve W + c_{37} \ve^2 (M^4+\varpi^4) \lambda(\tau), \\
\frac{d}{ds}\tilde{W} & \leq - \frac{1}{4}\ve \tilde{W} + c_{37} \ve^2 (M^4+\varpi^4) \lambda(\tau).
\end{align*}
which by the Gronwall's lemma and $\mathcal{A}$-relative $\xi$-invariance of $\mathcal{B}_3$ completes the proof.
\end{proof}

\section{Appendix D: The parity lemma}

This Appendix is dedicated to the proof of Lemma \ref{l:intersect} in Section \ref{s:invariant}. For clarity of the argument, we give definitions and prove a generalized claim in an abstract setting.

Assume $U$ is an open, bounded subset of $\mathbb{R}^2$. Let $\lu,\lv : [s_0,s_1] \times [t_0,t_1]$ be continuous functions satisfying the following properties:

(i) If $\lu (t,s)=0$, then $(t,\lv (s,t)) \notin \partial U)$,

(ii) For all $s \in [s_0,s_1]$, $u(s_0,t_0) <0$ and $u(s_0,t_1) >0$.

We say that $\lu,\lv$ intersect $U$ for some $s \in [s_0,s_1]$, if there exists $t \in [t_0,t_1]$ such that $\lu (s,t) =0$ and $(t,\lv(s,t)) \in U$ (clearly by (i), $(t,\lv(s,t))$ is then in the interior of $U$). We can count the number of times $\lu,\lv$ intersect $U$ in the following sense. For a fixed $s$, let $\varUpsilon(s)=\lbrace t_0,t_1 \rbrace \cup \lbrace t\in [t_0,t_1], \: (t,v(s,t)) \in \partial U) \rbrace$. By assumptions, $\lu(t,s) \neq 0$ for all $t \in \varUpsilon(s)$. We define the relation of equivalence $\sim$ on $\varUpsilon(s)$ with $t_1 \sim t_2$ whenever for all $t_3 \in \varUpsilon(s)$ such that $t_1 \leq t_3 \leq t_2$, we have that $\lu(t_1,s)$, $\lu(t_2,s)$ and $\lu(t_3,s)$ have the same sign. Let $\tilde{\varUpsilon}(s) = \varUpsilon(s) / \sim$ with the induced topology. As by assumptions $\varUpsilon(s)$ is a closed subset of a compact set, $\tilde{\varUpsilon}(s)$ is compact. By (i) and the compactness of $\partial U$, we see that $\tilde{\varUpsilon}(s)$ is totally disconnected, thus $\tilde{\varUpsilon}(s)$ is finite. By (ii), $|\tilde{\varUpsilon}(s)| \geq 2$.

Consider $|\tilde{\varUpsilon}(s)|-1$ segments $(t_k,t_{k+1}) \subset \varUpsilon(s)^c$, where $t_k,t_{k+1} \in \varUpsilon(s)$ and $t_k \not\sim t_{k+1}$. Then there exists at least one zero $t \in (t_k,t_{k+1})$ (i.e. $\lu(t,s)=0$), and all the zeroes $t \in (t_k,t_{k+1})$ are either all in $U$ or all in $\bar{U}^c$ (i.e. for all $t \in (t_k,t_{k+1})$ such that $\lu(t,s)=0$, we have either that for all such $t$, $(t,\lv(s,t)) \in U$, or for all such $t$,  $(t,\lv(s,t)) \in \bar{U}^c$). 

\begin{defn} We say that $\lu$, $\lv$ as above for a given $s \in [s_0,s_1]$ intersect $U$ exactly $n(s)$ times, if $n(s)$ is the number of segments $(t_k,t_{k+1}) \subset \varUpsilon(s)^c$ for which $t_k,t_{k+1} \in \varUpsilon(s)$ and $t_k \not\sim t_{k+1}$, such that all the zeroes $t \in (t_k,t_{k+1})$ are in $U$.
\end{defn}

Now we have:

\begin{proposition} {\bf The Parity Lemma.} \label{p:parity}
	The parity of $n(s)$ is constant for all $s \in [s_0,s_1]$.
\end{proposition}

In particular, we have that if $n(s_0)$ is odd, then $\lu,\lv$ intersect $U$ for all $s \in [s_0,s_1]$.

\begin{proof}
	Consider for some $\delta \geq 0$ the closed $\delta$-neighborhood of $\partial U$, denoted by $U(\delta)$ (clearly $U(0)= \partial U$). By the assumptions, $U(\delta)$ is compact; and by compactness, there exists $\delta_0 > 0$ so that there are no zeroes in $U(\delta_0)$ (i.e. for all $(s,t) \in [s_0,s_1] \times [t_0,t_1]$, if $(t,\lv(s,t)) \in U(\delta_0)$, then $\lu(s,t) \neq 0$).
	
	If $0 \leq \delta \leq \delta_0$, let $\varUpsilon(s,\delta)$ be all $t$ such that $(t,v(s,t)) \in U(\delta)$. Analogously as in the case $\delta=0$, $\varUpsilon(s,0)=\varUpsilon(s)$, we define $\tilde{\varUpsilon}(s,\delta)$ which is again finite with cardinality $\geq 2$, and $n(s,\delta) \leq |\tilde{\varUpsilon}(s,\delta)|-1$.
	
	Fix $s \in [s_0,s_1]$ for now. Note that $\delta \mapsto |\tilde{\varUpsilon}(s,\delta)|$, $\delta \mapsto n(s,\delta)$ are non-decreasing for $\delta \in [0,\delta_0]$. This follows from the construction and the natural continuous embedding $\varUpsilon(s,\delta) \rightarrow \varUpsilon(s,\delta')$ for $\delta < \delta'$. Thus $|\tilde{\varUpsilon}(s,\delta)|$, $n(s,\delta)$ strictly increase for at most finitely many times $0 < \delta_1(s)<...<\delta_{k(s)}(s) \in [0,\delta_0]$. Again by construction it is easy to see that these $\delta_j(s)$, $j=1,...,k(s)$ are characterized as these $\delta \in [0,\delta_0]$ for which at least one entire equivalence class in $\varUpsilon(s,\delta)$ lies on the boundary of $U(\delta)$. 
	
	Furthermore, for some $\delta_j(s)> 0$, we have that $n(s,\delta_j(s))-n(s,\delta_j(s)^-)$ is even. Indeed, by considering the equivalence classes which are in $\tilde{\varUpsilon}(s,\delta_j(s))$, but are not in $\tilde{\varUpsilon}(s,\delta_j(s)^-)$, we see that the number of equivalence classes must increase by an even number (as the signs of $u(s,t)$ alternate on equivalence classes), and also $n(s,\delta_j(s))-n(s,\delta_j(s)^-)$ must be an even number (as the zeroes "appear" in pairs of zeroes in $U$, respectively $\bar{U}^c$).
	
	Now, if $\delta \not\in \lbrace \delta_1(s),...\delta_{k(s)}(s) \rbrace $, it is easy to see that for some small neighborhood $(s-\nu,s+\nu)$, $\nu >0$, we have that for all $s^* \in  (s-\nu,s+\nu)$, $|\tilde{\varUpsilon}(s^*,\delta)|$ and $n(s^*,\delta)$ are constant. This is clearly true, as then no equivalence classes in $\varUpsilon(s,\delta)$ are entirely on the boundary of $U(\delta)$, thus persist for sufficiently small (uniformly by finiteness) perturbation. Also $|\tilde{\varUpsilon}(s,\delta)|=|\tilde{\varUpsilon}(s,\delta')|$ for sufficiently small $\delta'-\delta >0$, so there can be no new equivalence classes in $\varUpsilon(s^*,\delta)$ for sufficiently small $\nu$.
	
	We deduce from all that that the parity of $n(s)$ does not change on some small interval $(s-\nu,s+\nu)$. We see that by choosing some $\delta \neq \delta_j(s)$, $0 <\delta < \delta_0$. Then the parity of $n(s)=n(s,0)$ and $n(s,\delta)$ is the same, and $n(s,\delta)$ does not change for sufficiently small $\nu > 0$. The claim follows easily by contradiction.	
\end{proof}

\begin{proof}[Proof of Lemma \ref{l:intersect}]
	We set $\tilde{u}(s,t)=u(s,t)-(2k+1)\pi$, $\tilde{v}(s,t)=v(s,t)$, $t_0 = \tilde{T}_{k-1}$, $t_1 = \tilde{T}_{k+1}$, and $U = \tilde{\mathcal{N}}_k$. Now the property (i) follows from the definition of $\mathcal{A}$, and (ii) by Lemma \ref{l:B6two}. Proposition \ref{p:parity} thus implies Lemma \ref{l:intersect}.
\end{proof}

{\small
{\em Authors' addresses}:
{\em Sini\v{s}a Slijep\v{c}evi\'{c}}, Department of Mathematics, Bijeni\v{c}ka 30, University of Zagreb, Croatia
 e-mail: \texttt{slijepce@\allowbreak math.hr}.

}

\end{document}